\documentclass[a4paper,11pt,twoside,reqno]{amsart}
\usepackage[utf8]{inputenc}
\usepackage[plainpages=false,pdfpagelabels=true]{hyperref}
\usepackage{amssymb,amsthm,amsmath,amscd,mathtools,mathrsfs}
\usepackage[margin=1in]{geometry}
\usepackage{slashed}
\usepackage{comment}
\usepackage{caption}
\newtheorem{Satz}{Theorem}[section]
\newtheorem{Prop}[Satz]{Proposition}
\newtheorem{Thm}[Satz]{Theorem}
\newtheorem{Cor}[Satz]{Corollary}
\theoremstyle{definition}
\newtheorem{Dfn}[Satz]{Definition}
\newtheorem{Bem}[Satz]{Remark}
\newtheorem{Hyp}[Satz]{Hypothesis}
\newtheorem{remark}[Satz]{Remark}
\newtheorem{example}[Satz]{Example}

\newcommand{\Ric}{\operatorname{Ric}}

\newcommand{\rd}{\mathrm{rd}}

\allowdisplaybreaks[1]
\numberwithin{equation}{section}
\renewcommand{\epsilon}{\varepsilon}
\title{Einstein Metrics and Equivariant Harmonic Maps:\\
The Einstein Detection Principle}
\author{Anna Siffert}
\address{Universität M\"unster, Mathematisches Institut\\
Einsteinstr. 62\\
48149 M\"unster\\
Germany}
\email{asiffert@uni-muenster.de}
\begin{document}

\begin{abstract}
We develop a response-theoretic framework for local deformation problems of
compact cohomogeneity-one Einstein metrics.  After fixing the Einstein
normalization, the two-ended boundary-value problem is reformulated as a
finite-dimensional matching problem for smooth singular-orbit Einstein germs.
Under normalized Jacobi nondegeneracy and fixed equivariant probe data, each
nearby metric in a chosen smooth parameterized Einstein family carries a
locally unique harmonic probe.  The associated Jacobi and Green operators form
a core analytical package whose differential defines a universal first-order
response.  We prove that all admissible package-generated observations factor
through this response and identify its kernel as the universal invisible
space.  Injectivity yields first-order detection; for unrestricted smooth
local scalar postprocessings, finitely many scalar observations then give
local reconstruction.  We also give a finite operator-pairing criterion for
testing injectivity and record its scalar cohomogeneity-one specialization.
The round model and a numerical Einstein shooting example illustrate the
scope of the construction without being used in the abstract theorems.
\end{abstract}

\maketitle

\section{Introduction}
Einstein metrics occupy a central position in Riemannian geometry as canonical
metrics and critical points of the Einstein--Hilbert functional. Their
existence, construction and classification have been studied extensively; see,
for example,
\cite{Anderson,Besse,CheegerTian,Koiso}. By contrast, considerably less is
known about the local geometry of Einstein deformation spaces, particularly in
situations where the Einstein equations arise naturally as nonlinear
boundary-value problems rather than as elliptic equations on closed
manifolds.

\smallskip

Compact cohomogeneity-one Einstein manifolds provide a particularly natural
setting for studying such questions. Here the Einstein equations reduce to a
system of nonlinear ordinary differential equations together with
regular-singular smoothness conditions at the singular orbits
\cite{Bohm1999,EschenburgWang}. After fixing the Einstein constant, choosing a
unit-speed normal geodesic and quotienting by the remaining equivariant
diffeomorphism freedom, one obtains the normalized Einstein boundary-value
problem.

\smallskip

The first objective of the present paper is to reformulate this boundary-value
problem intrinsically. We show that it can be expressed as an
\emph{Einstein matching problem}: compatible Einstein germs are propagated
from the singular orbits to a canonical matching hypersurface, where they are
compared by means of an \emph{Einstein matching map}. This map is smooth
between finite-dimensional manifolds, and its zero set locally represents the
normalized Einstein boundary-value problem in the chosen gauge and
normalization. The matching construction isolates the
nonlinear geometric content of the Einstein equations and provides the
geometric foundation for all subsequent developments.

\smallskip

Building on this geometric reformulation, we introduce an auxiliary
variational theory based on canonically associated equivariant harmonic maps.
Rather than analysing the Einstein matching map directly, we study the
analytical structures generated by these auxiliary probes and show that,
after elimination of the probe variables, they descend naturally to a chosen
smooth parameterized Einstein family.

\smallskip

More precisely, after fixing a normalized reference probe and the full probe
data, normalized Jacobi nondegeneracy gives every sufficiently nearby Einstein
metric a unique harmonic probe on the local branch through that reference
probe. Throughout the paper, the adjective
\emph{canonical} always means canonical relative to the fixed Einstein
normalization together with the chosen probe data. The harmonic probes are
introduced purely for analytical purposes. They neither modify the Einstein
equations nor enlarge the underlying deformation problem. Instead, they
generate a natural core analytical package consisting of the canonical
Jacobi and Green operators. Finite-dimensional response operators and their
associated observables are then obtained from this package through admissible
observations and differentiation.

\smallskip

The central structural observation is that the auxiliary probe variables admit
a canonical nonlinear elimination. Although the probe fields themselves
disappear, the analytical structures generated by them survive the
elimination procedure and become canonically attached to the chosen smooth
parameterized Einstein family. Their relationship to the underlying Einstein
deformation problem is established through the Einstein matching problem
developed in Section~\ref{chap:intrinsic-moduli}.

\smallskip

These ideas lead to the central guiding principle of the paper.
\begin{quote}
\emph{
A suitably chosen auxiliary variational theory can generate a canonical
analytical response theory for a nonlinear geometric deformation problem
without changing the underlying geometric problem.
}
\end{quote}

\subsection{Main results}
The paper establishes four principal results.
\begin{enumerate}
\item
We reformulate the normalized Einstein boundary-value problem as an
\emph{Einstein matching problem}. Propagating compatible smooth Einstein germs
from the singular orbits to a canonical matching hypersurface defines the
associated Einstein matching map, a smooth map between finite-dimensional
manifolds. Its zero set provides a finite-dimensional realization of the local
normalized Einstein boundary-value problem in the chosen gauge.
\item
After fixing the reference probe and the complete probe data, normalized Jacobi
non\-de\-gen\-er\-acy gives a unique local continuation of that probe for every
sufficiently nearby Einstein metric, depending smoothly on the metric. This probe
assignment generates a natural core analytical package consisting of the
canonical Jacobi and Green operators. Finite-dimensional response operators
and associated observables are obtained from this package by admissible
observations and differentiation.
\item
The auxiliary probe variables admit a canonical nonlinear elimination.
Although the probe fields themselves disappear, the analytical structures
generated by them survive the elimination procedure and descend naturally to a
chosen smooth parameterized Einstein family. In this way, the auxiliary
variational theory gives rise to a canonical finite-dimensional analytical
response theory attached to the underlying Einstein deformation problem.
\item
The descended analytical package yields the
\emph{First-Order Einstein Detection Principle} together with conditional
results on local observability, distinction, rigidity and finite
reconstruction. Consequently, infinitesimal properties of a chosen smooth
parameterized Einstein family can be detected through the analytical response
of the associated harmonic probes.
\end{enumerate}
Taken together, these results establish a canonical analytical response theory
for local Einstein deformation problems. Their principal consequence is the
Einstein Detection Principle: suitably chosen auxiliary variational theories
generate canonical analytical structures that survive elimination and provide
a finite-dimensional response theory capable of detecting and, under explicit
injectivity hypotheses, locally reconstructing infinitesimal Einstein
deformations without altering the underlying geometric problem.

\subsection{Relation to previous work}
The present work lies at the intersection of three established areas of
differential geometry: the local deformation theory of Einstein metrics,
cohomogeneity-one Einstein geometry, and the theory of equivariant harmonic
maps.

\smallskip

The local deformation theory of compact Einstein metrics has a long history.
Its fundamental ingredients include gauge fixing by the Ebin slice,
elliptic linearization of the Einstein equations, infinitesimal Einstein
deformations and finite-dimensional Kuranishi reduction; see
Ebin~\cite{Ebin}, Koiso~\cite{Koiso}, Besse~\cite{Besse},
Biquard~\cite{Biquard}, and the recent survey of
Schwahn and Semmelmann~\cite{SchwahnSemmelmann}. The present paper is fully
compatible with this framework but adopts a different perspective: rather
than beginning with the linearized Einstein operator, we first reformulate
the nonlinear boundary-value problem itself.

\smallskip

The second ingredient is the theory of compact cohomogeneity-one Einstein
metrics. The regular-singular initial-value theory developed by
Eschenburg--Wang~\cite{EschenburgWang} provides the local Einstein germs used
throughout the paper, while the global existence theory includes B\"ohm's
construction of infinitely many inhomogeneous Einstein metrics
\cite{Bohm} together with numerous subsequent developments.
The finite-dimensional Einstein matching problem introduced here is built
upon this classical ODE framework but is adapted specifically to the local
two-ended boundary-value problem.

\smallskip

The third ingredient is the theory of equivariant harmonic maps.
General symmetry reductions were developed by
Eells--Ratto~\cite{EellsRatto} and Urakawa~\cite{Urakawa}; the
cohomogeneity-one self-map setting was studied by
P\"uttmann and Siffert~\cite{PuttmannSiffert}; equivariant Jacobi theory by
Branding and Siffert~\cite{BrandingSiffert}; and singular-orbit initial-value
theory by Siffert~\cite{SiffertIVP}. These results provide the analytical
tools used in the present work.

\smallskip

The contribution of this paper is not a new result within any one of these
three theories individually. Rather, it combines them into a unified
framework for studying local Einstein deformation problems.
The Einstein boundary-value problem is first reduced to a finite-dimensional
Einstein matching problem. An auxiliary harmonic-map theory is then coupled
to this geometric reduction, generating canonical analytical structures that
can be transported to a chosen smooth parameterized Einstein family.

\smallskip

The synthesis developed here combines the matching formulation, canonical harmonic probes, Jacobi--Green theory, elimination of the auxiliary variables, and finite-dimensional detection and reconstruction into a single local framework.

\smallskip

Accordingly, the novelty of the present paper lies neither in the individual
geometric constructions nor in the underlying analytical ingredients.
Instead, it lies in the response-theoretic architecture that connects them:
the auxiliary harmonic variables are eliminated, the resulting Jacobi--Green
package is descended to a chosen Einstein family, and every admissible
package-generated observation factors infinitesimally through a single
universal response differential.  This identifies a universal invisible
subspace and separates the genuinely geometric question of information loss
in the probe package from the finite-dimensional choice of observations.
Under the stated richness and injectivity hypotheses, the same structure
yields finite scalar detection and local reconstruction.  These implications,
together with the matching formulation that places them in the
cohomogeneity-one Einstein boundary-value problem, constitute the Einstein
Detection Principle proved here.

\subsection{Structure of the paper}
The paper is divided into two conceptually distinct parts.

\smallskip

The first part develops the geometric foundation of the theory. Beginning
with the normalized Einstein boundary-value problem, we introduce the
Einstein matching problem, construct the associated Einstein matching map,
and show how the local normalized Einstein moduli problem is represented by
its zero set. This finite-dimensional reformulation provides the geometric
framework for all subsequent analytical developments.

\smallskip

The second part develops the associated analytical theory. We introduce the
canonical equivariant harmonic probes, establish their existence and smooth
dependence on the Einstein metric, and develop the resulting Jacobi- and
Green-theoretic framework. After canonically eliminating the auxiliary probe
variables, the resulting analytical structures descend naturally to a chosen
smooth parameterized Einstein family and define the analytical response
theory developed in the remainder of the paper.

\smallskip

The later sections establish the universal response identity, formulate the
Einstein Detection Principle, and prove the associated local observability,
distinction and reconstruction theorems. The paper concludes with a nontrivial numerical realization for a
cohomogeneity-one Einstein shooting problem.

\subsection{Acknowledgements}
Parts of the exposition and manuscript preparation benefited from discussions with OpenAI's ChatGPT, which was used as an interactive assistant for mathematical exposition, proofreading, and structural suggestions. The author takes full responsibility for the mathematical content of the paper.

\smallskip

For the present revision, the use of ChatGPT was substantially extended. In addition to assistance with exposition and manuscript organization, it was used interactively in the development and refinement of mathematical arguments, in exploring possible proof strategies, in identifying potential gaps and hidden assumptions, in checking calculations and logical dependencies, in formulating and auditing auxiliary results, and in assisting with computational and \LaTeX{} aspects of the work. Its outputs were repeatedly checked, modified, or rejected by the author as appropriate. The author independently determined which arguments and results to include and retains full responsibility for all mathematical claims, proofs, computations, and conclusions in the revised manuscript.

\section{The Einstein Matching Problem}
\label{chap:intrinsic-moduli}
The purpose of this section is to reformulate the normalized Einstein
boundary-value problem as a finite-dimensional matching problem. Rather than
solving the global boundary-value problem directly, we propagate compatible
local Einstein germs from the two singular orbits to a canonical matching
hypersurface and compare the resulting matching data.

\smallskip

This construction leads to the \emph{Einstein matching map}, a smooth map
between finite-dimensional manifolds whose zero set provides a
finite-dimensional realization of the local normalized Einstein moduli
problem. The matching formulation isolates the nonlinear geometric content of
the Einstein equations and provides the geometric foundation for the
analytical framework developed in the remainder of the paper.

\smallskip

Throughout this section we assume familiarity with the basic theory of compact
cohomogeneity-one manifolds and refer to
\cite{AlexandrinoBettiol,GroveZiller} for background. General references for
Einstein manifolds and the regular-singular Einstein initial-value problem
include
\cite{Besse,EschenburgWang}.

\subsection{The normalized Einstein problem}
Let $G$ act smoothly on a compact manifold $M$ with cohomogeneity one and
orbit space
$M/G\cong[0,T].$
Let $H$ denote the principal isotropy subgroup. Throughout the paper we
consider $G$-invariant Einstein metrics satisfying
$$\Ric(g)=\lambda g, \qquad \lambda>0,$$
where the Einstein constant $\lambda$ is fixed once and for all.

\smallskip

We furthermore choose a unit-speed normal geodesic joining the two singular
orbits and fix the standard normalization described below, thereby removing
the remaining equivariant reparametrization freedom.

\smallskip

This leads to the following local boundary-value problem.
\begin{Dfn}[Normalized Einstein problem]
Let $g_0$ be a fixed normalized cohomogeneity-one Einstein metric.
\smallskip
The \emph{normalized Einstein problem} consists of determining all
normalized $G$-invariant Einstein metrics sufficiently close to $g_0$ that
satisfy the smoothness conditions at both singular orbits, modulo
equivariant diffeomorphisms preserving the chosen normalization.
\end{Dfn}

At this stage no smooth structure is assumed on the corresponding solution
set. Rather than studying this nonlinear boundary-value problem directly, our
goal is to replace it by an equivalent finite-dimensional matching problem.

\smallskip

More precisely, the following subsections construct the Einstein matching
map and show that its zero set provides a finite-dimensional realization of
the local normalized Einstein moduli problem. This finite-dimensional
realization forms the geometric starting point for the analytical response
theory developed in the remainder of the paper.

\subsection{Smooth parameterized Einstein families}
\label{subsec:parameterized-einstein-families}
The analytical constructions developed in this paper are local on a chosen
smooth finite-dimensional parameter manifold. They do not require the ambient
class of normalized Einstein metrics, nor the corresponding Einstein moduli
problem, to possess a smooth manifold structure.

\smallskip

Instead, the basic geometric object is a smooth finite-dimensional family of
normalized Einstein metrics.

\begin{Dfn}[Smooth parameterized Einstein family]
Let $M$ be the underlying manifold.
\smallskip
A \emph{smooth parameterized Einstein family} consists of
\begin{enumerate}
\item
a smooth finite-dimensional manifold $S$;
\item
a family of normalized Einstein metrics
$\{g_p\}_{p\in S},$
depending smoothly on the parameter $p$.
\end{enumerate}
\end{Dfn}
Here smooth dependence means that, in local coordinates on
$S\times M$, the coefficients of the metric tensor depend smoothly on both
the parameter variables and the spatial variables.

\smallskip

No smooth structure is assumed on the ambient class of normalized Einstein
metrics. In particular, the parameter manifold $S$ is not required to be a
local Einstein moduli space or to contain all nearby Einstein metrics.
Rather, it is part of the chosen data on which the subsequent analytical
theory is formulated.

\smallskip

Let $\operatorname{Met}(M)$ denote the space of smooth Riemannian metrics on $M$.
The family determines a smooth map
$$\mathbf g:S\longrightarrow \operatorname{Met}(M),\qquad p\longmapsto g_p.$$
Accordingly, a tangent vector $X\in T_pS$ determines the infinitesimal metric
variation $D\mathbf g_p[X]$. A nonzero parameter direction need not give a
nonzero metric variation unless the parameterization is immersive. Whenever
we speak below of tangent vectors of $S$ as genuine infinitesimal Einstein
deformations, we either assume that $\mathbf g$ is immersive on the
neighbourhood under consideration or interpret the statement after applying
$D\mathbf g_p$. No assertion is made that every infinitesimal Einstein
deformation is realized by the chosen family. To prevent ambiguity between the parameter manifold and the ambient metric space, we do not identify $p\in S$ with the metric $g_p$. Throughout the paper, $p$ denotes a parameter point, $g_p=\mathbf g(p)$ denotes the associated Einstein metric, and set-theoretic operations involving subsets of the ambient metric space are taken through the map $\mathbf g$.

\begin{remark}
The present framework includes local gauge-fixed Einstein deformation
families, smooth finite-dimensional Einstein moduli spaces whenever these
exist, and more generally any smooth finite-dimensional parameterization of
normalized Einstein metrics.

\smallskip

In cohomogeneity-one Einstein problems one frequently constructs a smooth
parameter space of local Einstein germs or propagated local Einstein
solutions. Whenever one obtains a smooth finite-dimensional family consisting
entirely of globally defined normalized Einstein metrics, the theory
developed in this paper applies directly to that family.
\end{remark}

\subsection{Functional-analytic conventions}
The analytical theory developed in the subsequent sections is formulated in a
fixed functional-analytic framework. We therefore collect here the standing
Banach- and Hilbert-space conventions used throughout the remainder of the
paper.

\smallskip

Fix an integer
$k\ge4$
and a Hölder exponent
$0<\alpha<1.$
All nonlinear constructions are first carried out in the corresponding
$C^{k,\alpha}$ Banach completions. Smoothness and higher regularity are then
obtained from the elliptic and regular-singular regularity theory.
Whenever spectral statements are made, they refer to the corresponding
self-adjoint $L^2$ realizations.

\smallskip

Throughout the paper we adopt the following conventions.

\begin{enumerate}
\item
\textbf{Configuration spaces.}
The spaces of invariant metrics, Einstein germs and equivariant harmonic maps
are closed Banach submanifolds determined by the prescribed symmetry,
endpoint and normalization conditions.
\item
\textbf{Jacobi operators.}
Every endpoint condition defining a Jacobi operator is elliptic.
Whenever spectral theory is used, the endpoint and normalization conditions
are assumed to determine a self-adjoint realization.
\item
\textbf{Local trivialisations.}
Nearby bundles and varying $L^2$ spaces are identified by fixed smooth local
trivialisations. All differentiability statements for operator families are
understood with respect to these identifications.
\item
\textbf{Operator families.}
After passing to a local trivialization, every operator family acts between
fixed Banach or Hilbert spaces with a common operator domain.
Accordingly, smoothness of an unbounded operator family means smoothness of
the associated bounded operator from the fixed graph norm space into the
fixed Hilbert space. Equivalently, the operator coefficients depend smoothly
on the parameter while the operator domain remains fixed.
\end{enumerate}

The assertions in items (1)--(4) are standing hypotheses of the abstract
analytical theory, not merely notational conventions. In particular, the
Banach-submanifold property of the constrained mapping spaces, ellipticity of
the endpoint conditions, existence of a common graph domain after
trivialization and, whenever spectral theory is invoked, self-adjointness of
the chosen realization are assumed on the parameter neighbourhood under
consideration. In a concrete cohomogeneity-one model each of these properties
must be verified for the corresponding regular-singular problem. None follows
from Fredholmness or formal self-adjointness alone.

\subsection{Einstein evolution}
Let $\gamma:[0,T]\rightarrow M$ be the fixed unit-speed normal geodesic joining
the two singular orbits. Along the regular part of the orbit space the
manifold is naturally identified with
$$(0,T)\times G/H,$$
and every $G$-invariant metric can be written in the form
$g=dt^2+g_t,$
where
$(g_t)_{t\in(0,T)}$
is a smooth one-parameter family of $G$-invariant metrics on the principal
orbit $G/H$.

\smallskip

After the cohomogeneity-one reduction, the Einstein equations become a
nonlinear system of ordinary differential equations governing the evolution
of the family $(g_t)$. Rather than describing this evolution in terms of
individual metric coefficients, we use intrinsic geometric quantities
associated with the principal orbits.

\smallskip

The shape operator of the principal orbit at time $t$ is
$$L_t=\frac12 g_t^{-1}\dot g_t,$$
and its trace
$$H(t)=\operatorname{Tr}(L_t)$$
is the mean curvature of the hypersurface
$\{t\}\times G/H.$
The function $H$ will play a fundamental role in the sequel, since its
monotonicity provides the canonical matching hypersurfaces used throughout
the paper.

\subsubsection{Monotonicity of the mean curvature}
The normal Einstein equation implies the evolution equation
$$H' = -\lambda-\operatorname{Tr}(L_t^{\,2}).$$
Since
$$\lambda>0, \qquad \operatorname{Tr}(L_t^{\,2})\ge0,$$
it follows immediately that the mean curvature decreases strictly along every
regular Einstein trajectory.

\smallskip

Although this monotonicity property is classical (see, for example,
\cite{Bohm1999}), it is one of the fundamental ingredients of the present
paper. It provides a canonical family of transverse hypersurfaces on which
the Einstein matching construction is based.

\begin{Prop}[Mean-curvature monotonicity {\cite{Bohm1999}}]
\label{Prop_monotonH}
Let $g$ be a $G$-invariant Einstein metric satisfying
$$\Ric(g)=\lambda g, \qquad \lambda>0.$$
Then the mean curvature is a strict Lyapunov function for the Einstein
evolution. In particular,
$H'(t)<0$
along every regular Einstein trajectory.
\smallskip
Consequently, every regular level set
$\Sigma_c=\{H=c\}$
is a smooth hypersurface transverse to the Einstein evolution.
\end{Prop}

\smallskip

Since the mean curvature is strictly monotone, its regular level sets provide
a canonical foliation of a neighbourhood of every regular Einstein
trajectory. Unlike a coordinate hypersurface, this foliation is determined
intrinsically by the Einstein evolution itself and therefore requires no
auxiliary choices.

\subsubsection{Boundary-value problem}
A central problem in the theory of compact cohomogeneity-one Einstein metrics
is the Einstein boundary-value problem; see, for example,
\cite{Bohm1999,EschenburgWang}. One prescribes smooth Einstein germs at the
two singular orbits (or, equivalently, the corresponding regularity
conditions at the endpoints of the orbit space) and asks whether these data
extend to a global smooth Einstein metric.

\smallskip

After the cohomogeneity-one reduction, this becomes a nonlinear two-point
boundary-value problem for the Einstein evolution. The first step is
therefore to understand the finite-dimensional spaces of smooth Einstein
germs near the singular orbits and the way in which these germs propagate
under the Einstein flow.

\smallskip

The central observation underlying the present paper is that the global
boundary-value problem can be reformulated as a finite-dimensional matching
problem. Rather than comparing the propagated Einstein germs only at the
second singular orbit, we compare them on a canonical transverse
hypersurface determined by the mean curvature. This leads to the Einstein
matching map introduced in the following subsection and provides the
geometric foundation for the analytical response theory developed later in
the paper.

\subsection{Smooth Einstein germs}
The Einstein matching problem is built from the local Einstein geometries near
the two singular orbits. These local geometries are represented by smooth
Einstein germs and therefore constitute the fundamental local objects of the
construction.

\smallskip

Regularity across a singular orbit imposes algebraic compatibility conditions
on the Taylor expansion of an invariant Einstein metric. After imposing these
smoothness conditions, only finitely many Taylor coefficients remain free.
These admissible free jet parameters provide local coordinates for the
corresponding space of smooth Einstein germs.

\smallskip

The underlying regular-singular initial-value analysis is due to
Eschenburg--Wang~\cite{EschenburgWang}. Their general theorem supplies local
existence under a representation-theoretic hypothesis and, importantly, also
exhibits finite-dimensional nonuniqueness when only the lowest-order singular
orbit data are prescribed. Accordingly, the proposition below does not
attribute uniqueness from the initial metric and shape operator to
Eschenburg--Wang. Instead, it isolates as an explicit hypothesis the stronger
jet-level uniqueness property required by the matching construction.

\begin{Prop}[Finite-dimensional Einstein germ manifolds]
\label{prop_finitegerms}
Fix a cohomogeneity-one group diagram, a positive Einstein constant and one
of the two singular orbit types.

\smallskip

Assume that the singular-orbit Einstein problem has been placed in the
regular-singular framework of Eschenburg--Wang near the chosen orbit and that,
for the particular group diagram under consideration, the following additional
local properties hold:
\begin{enumerate}
\item
the smoothness conditions determine a finite-dimensional manifold of
admissible initial jets;
\item
the higher-order Taylor coefficients depend smoothly on the admissible free
jet parameters;
\item
every sufficiently nearby admissible jet is realized by a unique smooth
invariant Einstein germ.
\end{enumerate}
Then the space of smooth invariant Einstein germs near the reference germ is
a finite-dimensional smooth manifold. Moreover, every sufficiently nearby
Einstein germ is uniquely determined by its admissible free jet parameters.
\end{Prop}

\begin{proof}
By assumption, the smoothness relations determine a finite-dimensional smooth
manifold of admissible initial jets. The regular-singular recursion expresses
every higher-order Taylor coefficient smoothly in terms of the admissible
free jet parameters, while the assumed existence and uniqueness theorem
associates a unique smooth Einstein germ to every admissible jet.
Consequently, the jet-to-germ correspondence defines a local smooth
parametrization of the germ manifold.
\end{proof}

Thus the admissible free jet parameters furnish local coordinates on the
manifold of smooth Einstein germs.
\smallskip
Applying Proposition~\ref{prop_finitegerms} to the two singular orbit types
produces finite-dimensional germ manifolds
$$G_- \qquad\text{and}\qquad G_+.$$
These two manifolds form the geometric input for the Einstein matching
construction developed below.
\smallskip
We propagate these local Einstein germs along the Einstein
evolution to a canonical matching hypersurface.

\subsection{The Einstein Cauchy-data manifold}
The reduced Einstein evolution is naturally formulated on a finite-dimensional
phase space of invariant Cauchy data rather than directly on the space of
metrics.
\smallskip
Fix the principal orbit $G/H$ together with the positive Einstein constant
$\lambda>0.$
We impose the following local Cauchy-data regularity hypothesis. In a
neighbourhood of the reference Cauchy datum, the chosen normalization together
with the Einstein Hamiltonian constraint cuts out a smooth finite-dimensional
submanifold of the ambient invariant Cauchy-data space; the reduced Einstein
vector field is smooth and tangent to this submanifold; and the corresponding
initial-value problem has unique solutions depending smoothly on the initial
datum for as long as the trajectory remains in the regular region. Denote this
local constraint manifold by $\mathcal P$.
Its points are pairs $(g,L)$, where
\begin{itemize}
\item
$g$ is a $G$-invariant Riemannian metric on the principal orbit;
\item
$L:T(G/H)\rightarrow T(G/H)$ is a $g$-symmetric endomorphism;
\item
the chosen geometric normalization has been imposed;
\item
the Einstein Hamiltonian constraint is satisfied;
\item
the remaining admissibility conditions guaranteeing local existence and
uniqueness for the reduced Einstein evolution hold.
\end{itemize}

Under this local regularity hypothesis, $\mathcal P$ is the natural phase
space for the reduced Einstein ordinary differential equations. Every regular Einstein metric determines a
trajectory
$(g(t),L(t)) \in \mathcal P,$
and conversely every admissible initial datum sufficiently close to the
reference trajectory determines a unique local Einstein evolution by the
standard local existence and uniqueness theory.

\smallskip

The mean curvature defines a smooth function
$$H:\mathcal P\longrightarrow\mathbb R, \qquad H(g,L)=\operatorname{Tr}_g(L).$$
Let
$X_{\mathrm{Ein}}$
denote the vector field on $\mathcal P$ generating the reduced Einstein
evolution.

\smallskip

By Proposition~\ref{Prop_monotonH},
$dH\!\left(X_{\mathrm{Ein}}\right)<0$
at every regular point of $\mathcal P$. Consequently,
$dH\neq0$
throughout the regular region of the Einstein evolution.

\smallskip

It follows that every regular level set
$H^{-1}(c)$
is a smooth hypersurface transverse to the Einstein flow.

\smallskip

These canonical transverse hypersurfaces provide the geometric setting for
the Einstein matching construction introduced in the following subsection.

\subsection{Einstein matching hypersurface}
The strict monotonicity of the mean curvature provides a canonical family of
hypersurfaces transverse to the Einstein evolution. These hypersurfaces serve
as canonical cross-sections on which propagated Einstein germs can be
compared.

\smallskip

\begin{Prop}
Let
$H:\mathcal P\longrightarrow\mathbb R$
denote the mean-curvature function, and let $c$ be a regular value of $H$.
Then the level set
$\Sigma_c=H^{-1}(c)$
is a smooth hypersurface in $\mathcal P$.
\smallskip
Moreover, every sufficiently nearby regular Einstein trajectory intersects
$\Sigma_c$ at most once.
\end{Prop}
\begin{proof}
By Proposition~\ref{Prop_monotonH},
$dH(X_{\mathrm{Ein}})<0.$
Hence
$dH\neq0$
at every point of every regular Einstein trajectory. The Regular Value
Theorem therefore implies that $\Sigma_c$ is a smooth hypersurface.

\smallskip

Since the mean curvature decreases strictly along every regular Einstein
trajectory, no trajectory can intersect the same level set twice.
\end{proof}

Among the regular level sets of $H$, one will play a distinguished role.

\begin{Dfn}[Einstein matching hypersurface]
Assume that the reference Einstein trajectory reaches $H=0$ at an interior
regular time $t_*$ and remains in the regular Cauchy-data region on a compact
time interval containing $t_*$. The hypersurface
$\Sigma:=H^{-1}(0)$
through the reference matching point is called the \emph{Einstein matching
hypersurface}.
\end{Dfn}

Since $dH(X_{\mathrm{Ein}})<0$, the reference crossing is transverse. Smooth
dependence of the reduced flow on initial data and the Implicit Function
Theorem therefore imply that, after restricting to initial data sufficiently
close to the reference datum, every such trajectory exists up to the
corresponding hitting time and intersects $\Sigma$ exactly once. The
existence-to-the-section requirement is essential: strict monotonicity alone
gives uniqueness of a crossing, but not existence of one.

\smallskip

Thus $\Sigma$ is a geometrically distinguished local cross-section for the
Einstein evolution and provides a unique matching point for every nearby
trajectory that crosses the chosen level.

\smallskip

Consequently, every normalized Einstein metric belonging to the chosen smooth
parameterized Einstein family determines a canonical matching point on
$\Sigma$. This defines the matching-point map
$j:S\longrightarrow\Sigma$. Because the family $p\mapsto g_p$ is smooth and
the $H=0$ crossing is transverse, the same hitting-time argument used in
Proposition~\ref{prop:smooth-propagation} shows that $j$ is smooth after
shrinking $S$ if necessary.

\subsection{Propagation to the Einstein matching hypersurface}
Every smooth Einstein germ in either germ manifold
$$G_- \qquad\text{or}\qquad G_+$$
determines a unique local Einstein evolution in the Einstein Cauchy-data
manifold $\mathcal P$.

\smallskip

After restricting to sufficiently small neighbourhoods of the reference
germs, every propagated Einstein trajectory intersects the Einstein matching
hypersurface $\Sigma$ exactly once. Consequently, every nearby Einstein germ
determines a unique point of matching data on $\Sigma$.

\smallskip
This defines the propagation maps
$\Phi_\pm:G_\pm\longrightarrow\Sigma.$
The geometric reduction may therefore be summarized by the diagram
$$
\begin{array}{ccc}
G_- && G_+\\[1ex]
\Big\downarrow{\Phi_-} &&
\Big\downarrow{\Phi_+}\\[1ex]
&\Sigma&
\end{array}
$$
Thus the local Einstein boundary-value problem is reduced to comparing two
finite-dimensional families of matching data inside the common Einstein
matching hypersurface.

\smallskip
\begin{Prop}[Smooth propagation]
\label{prop:smooth-propagation}
Assume that the two reference germ trajectories reach the reference matching
point in the regular region and that the reduced flow satisfies the local
existence and smooth-dependence hypothesis stated above on compact intervals
up to that crossing. Then, after restricting to sufficiently small
neighbourhoods of the reference germs, every corresponding trajectory exists
up to its unique $H=0$ crossing and the propagation maps
$\Phi_\pm:G_\pm\longrightarrow\Sigma$
are smooth.
\end{Prop}
\begin{proof}
On every compact subinterval of the regular orbit interval, the reduced
Einstein equations form a smooth finite-dimensional system of ordinary
differential equations. Standard smooth dependence of ODE solutions on their
initial conditions therefore yields smooth dependence of the propagated
Einstein trajectories on the initial Einstein germs; see
\cite[Chapter~V]{Hartman}.

\smallskip
Since the Einstein evolution is transverse to the matching hypersurface
$\Sigma$, the corresponding hitting time depends smoothly on the initial
germ by the Implicit Function Theorem. Composing the smooth solution map,
the smooth hitting-time map and the evaluation map yields the smooth
propagation maps.
\end{proof}

\smallskip
Each singular orbit therefore determines a finite-dimensional manifold of
matching data on the common Einstein matching hypersurface. The original
local Einstein boundary-value problem is thereby reduced to comparing these
two propagated families.

\smallskip
The next subsection packages this comparison into a single smooth
finite-dimensional map, the \emph{Einstein matching map}.

\subsection{Matching map}
The propagation maps reduce the local Einstein boundary-value problem to the
comparison of the propagated Einstein germ manifolds inside the common
Einstein matching hypersurface.
\smallskip
Intrinsically, a pair
$(u,v)\in G_-\times G_+$
is said to satisfy the Einstein matching condition if
$$\Phi_-(u)=\Phi_+(v).$$
Thus the Einstein boundary-value problem is reduced to deciding whether the
two propagated germ manifolds intersect inside the matching hypersurface.

\smallskip
For analytical purposes it is convenient to represent this intrinsic matching
condition in local coordinates.

\begin{Dfn}[Local representative of the Einstein matching map]
Choose a coordinate chart
$$\chi:V\subset\Sigma\longrightarrow\mathbb R^d, \qquad d=\dim\Sigma,$$
containing the matching point
$q_0 = \Phi_-(u_0) = \Phi_+(v_0)$
of the reference Einstein metric.

\smallskip
After possibly shrinking the germ neighbourhoods, assume
$\Phi_\pm(G_\pm)\subset V.$
The corresponding local representative of the Einstein matching map is
$\mathcal M_\chi: G_-\times G_+ \longrightarrow \mathbb R^d,$
defined by
$$\mathcal M_\chi(u,v) = \chi(\Phi_-(u)) - \chi(\Phi_+(v)).$$
\end{Dfn}

\smallskip
Its zero set is independent of the chosen coordinate chart, since
$\mathcal M_\chi(u,v)=0 \iff \Phi_-(u)=\Phi_+(v).$
Likewise, if two nearby regular mean-curvature levels are both crossed by all
trajectories under consideration, the Einstein flow gives a local
diffeomorphism between the corresponding cross-sections. The two matching
problems are then locally equivalent under this flow identification.

\begin{Prop}[Coordinate invariance]
\label{prop:coordinate-independence}
Let $\chi:V\rightarrow\mathbb R^d$ and
$\widetilde\chi:\widetilde V\rightarrow\mathbb R^d$ be coordinate charts on
$\Sigma$ with $q_0\in V\cap\widetilde V$, and let $\mathcal M_\chi$ and
$\mathcal M_{\widetilde\chi}$ be the corresponding residual maps. If
$\Phi_-(u_0)=\Phi_+(v_0)=q_0$, then
$$
D\mathcal M_{\widetilde\chi}(u_0,v_0)
=
D(\widetilde\chi\circ\chi^{-1})_{\chi(q_0)}
\circ D\mathcal M_\chi(u_0,v_0).
$$
Consequently, the kernel, rank and surjectivity of the linearized matching
map are independent of the chart, while the images for two charts are carried
into one another by the indicated transition isomorphism. In particular,
transversality of a matching point is intrinsic.
\end{Prop}
\begin{proof}
Set $F=\widetilde\chi\circ\chi^{-1}$. Near the matching pair one has
$\widetilde\chi\circ\Phi_\pm=F\circ\chi\circ\Phi_\pm$, but in general
$\mathcal M_{\widetilde\chi}\neq F\circ\mathcal M_\chi$ because the matching
residual is a difference of two chart values. Differentiating the two branches
separately at the common point $q_0$ gives
$$
D\mathcal M_{\widetilde\chi}(\xi,\zeta)
=DF_{\chi(q_0)}D\chi_{q_0}D\Phi_-(\xi)
-DF_{\chi(q_0)}D\chi_{q_0}D\Phi_+(\zeta),
$$
which equals $DF_{\chi(q_0)}D\mathcal M_\chi(\xi,\zeta)$. Since
$DF_{\chi(q_0)}$ is an isomorphism, the stated intrinsic properties follow.
\end{proof}
\begin{remark}
The map
$\mathcal M_\chi$
depends on the chosen coordinate chart, whereas the underlying matching
problem
$\Phi_-(u)=\Phi_+(v)$
is intrinsic.

\smallskip
Proposition~\ref{prop:coordinate-independence} shows that every first-order
property of the matching map depends only on the intrinsic matching problem.
Accordingly, whenever no confusion can arise, we simply write
$\mathcal M$
for an arbitrary local representative.
\end{remark}

\smallskip
The normalized Einstein boundary-value problem has therefore been reduced to
the finite-dimensional nonlinear equation
$\mathcal M=0.$
Its solution set provides a finite-dimensional realization of the local
normalized Einstein moduli problem and forms the geometric foundation for the
analytical theory developed in the remainder of the paper.

\subsection{Reduction of the Einstein boundary-value problem}
We now arrive at the central geometric result of this section.
Every ingredient entering the Einstein matching problem has been constructed,
$$
G_-,
\qquad
G_+,
\qquad
\Sigma,
\qquad
\Phi_-,
\qquad
\Phi_+.
$$
Together these objects provide a finite-dimensional realization of the local
normalized Einstein boundary-value problem.

\begin{Thm}[Intrinsic Reduction Theorem]
\label{thm:intrinsic-reduction}
Fix a cohomogeneity-one group diagram, a positive Einstein constant and the
geometric normalizations introduced above.
\smallskip
Assume the hypotheses established in the preceding subsections, namely the
regular-singular Einstein germ theory of
Proposition~\ref{prop_finitegerms}, the local smooth Cauchy-data constraint
manifold, the reduced Einstein evolution, the Einstein matching hypersurface
and the smooth propagation maps. Assume also that the normal orientation and
the definition of the shape operator are fixed consistently on the two sides,
that equality of points of $\mathcal P$ means equality of the full reduced
Einstein Cauchy data required for uniqueness of the regular ODE, and that the
chosen normalization removes the residual equivariant gauge up to any
explicitly retained stabilizer.
\smallskip
After possibly shrinking to sufficiently small neighbourhoods of the
reference Einstein metric, every propagated Einstein germ intersects the
Einstein matching hypersurface
$\Sigma=\{H=0\}$
exactly once.
\smallskip
Then the following statements hold.
\begin{enumerate}
\item
There is a natural local correspondence between normalized Einstein metrics
sufficiently close to the reference metric and compatible pairs of propagated
endpoint germs. It is one-to-one after the stated normalization has removed
the residual equivariant gauge; if a stabilizer is retained, the
correspondence is understood modulo that stabilizer.
\item
Under this correspondence, the local normalized Einstein boundary-value
problem is equivalent to the finite-dimensional matching equation
$\mathcal M(u,v)=0,$
where
$\mathcal M: G_-\times G_+ \longrightarrow \mathbb R^{\dim\Sigma}$
is any local representative of the Einstein matching map.
\end{enumerate}
Consequently, the zero set of the Einstein matching map provides a
finite-dimensional realization of the local normalized Einstein moduli
problem.
\end{Thm}
\begin{proof}
By Proposition~\ref{prop_finitegerms}, each admissible endpoint jet determines
a unique smooth Einstein germ. By smooth propagation, each such germ has a
unique matching point on $\Sigma$. If two propagated germs have the same
matching point, then they determine the same reduced Cauchy data there, so
uniqueness for the regular Einstein initial-value problem identifies the two
regular evolutions. They therefore glue to one normalized Einstein metric.
Conversely, every nearby normalized Einstein metric determines its two
endpoint germs, and both propagate to its unique matching point. Hence nearby
normalized Einstein metrics are in one-to-one correspondence with pairs
$(u,v)$ satisfying $\Phi_-(u)=\Phi_+(v)$, equivalently
$\mathcal M(u,v)=0$.
\end{proof}

\smallskip
The Intrinsic Reduction Theorem completes the geometric part of the paper.
The original nonlinear Einstein boundary-value problem has been reduced to a
finite-dimensional nonlinear matching problem. This geometric reduction
provides the foundation for the analytical response theory developed in the
remainder of the paper.

\begin{Bem}
The theorem separates the theory into two conceptually distinct stages.

\smallskip
The geometric stage constructs the Einstein matching problem and its local
finite-dimensional realization.

\smallskip
The analytical stage, developed in the remainder of the paper, first
constructs a response theory on a chosen smooth parameterized Einstein family.
Only under the separate matching-data extension hypothesis does this response
theory factor through the matching hypersurface and thereby provide
information about infinitesimal matching data. No later detection theorem
formulated purely on the parameter manifold $S$ depends on this additional
extension hypothesis.
\end{Bem}

\subsection{Transversality and local structure}
The Intrinsic Reduction Theorem reduces the local normalized Einstein
boundary-value problem to the study of the differential of the Einstein
matching map. The local geometry of the corresponding finite-dimensional
realization is therefore governed by this differential.

\begin{Dfn}[Transverse matching point]
A matching point
$(u_0,v_0)\in G_-\times G_+$
is called \emph{transverse} if the differential
$$D\mathcal M_{(u_0,v_0)} : T_{(u_0,v_0)}(G_-\times G_+) \longrightarrow \mathbb R^{\dim\Sigma}$$
is surjective.

\smallskip
A normalized Einstein metric is called \emph{transverse} if its associated
matching point is transverse.
\end{Dfn}

By Proposition~\ref{prop:coordinate-independence}, transversality is
independent of the chosen coordinate representative of the Einstein matching
map.

\begin{Prop}[Local structure]
\label{prop:local-structure}
Suppose that the reference Einstein metric is transverse.

\smallskip
Then the solution set of the local Einstein matching problem is a smooth
manifold of dimension
$\dim G_-+\dim G_+-\dim\Sigma.$
\end{Prop}
\begin{proof}
Since the differential of the Einstein matching map is surjective, the
Implicit Function Theorem implies that
$\mathcal M^{-1}(0)$
is a smooth submanifold of
$G_-\times G_+$
having codimension
$\dim\Sigma.$
The dimension formula follows immediately.
\end{proof}
\begin{Cor}[Local rigidity]
If
$\dim G_-+\dim G_+=\dim\Sigma,$
then every transverse matching point is locally isolated.
\end{Cor}
\begin{Cor}[Local families]
If
$\dim G_-+\dim G_+>\dim\Sigma,$
then every transverse matching point belongs to a smooth local family of
solutions of dimension
$\dim G_-+\dim G_+-\dim\Sigma>0.$
\end{Cor}

\smallskip
These results show that the differential of the Einstein matching map
controls the local geometry of the finite-dimensional realization of the
Einstein boundary-value problem. The remainder of the paper develops
analytical tools for studying this differential through the response theory
generated by the auxiliary harmonic probes.

\subsection{A completely explicit model: the round join sphere}
\label{sec:explicit-round-join}
The abstract construction developed in the preceding subsections can be carried
out explicitly for the round sphere viewed as a cohomogeneity-one join. The
purpose of this example is not to prove new results about round metrics, but to
illustrate the germ spaces, the Einstein evolution, the matching hypersurface,
the propagation maps and the Einstein matching map in a completely explicit
model.

\smallskip
Let $p,q\ge1$, let $g_{S^p}$ and $g_{S^q}$ denote the unit round metrics, and
consider the cohomogeneity-one action of
$$SO(p+1)\times SO(q+1)$$
on $S^{p+q+1}$. The principal orbit is $S^p\times S^q$, while the orbit space
is a closed interval.

\subsubsection*{The homothetic round family}
For $r>0$, the round metric of radius $r$ is
$g_r=dt^2+a_r(t)^2g_{S^p}+ b_r(t)^2g_{S^q},$
where
$$
a_r(t)=
r\sin\left(\frac{t}{r}\right),
\qquad
b_r(t)=
r\cos\left(\frac{t}{r}\right),
$$
and
$$
0\le t\le T_r=\frac{\pi r}{2}.
$$
Its Einstein constant is
$$\Ric(g_r)= \lambda_r g_r, \qquad \lambda_r =\frac{p+q}{r^2}.$$
Thus $r$ parametrizes the homothetic family of round Einstein metrics. After
fixing the normalization $\lambda=p+q,$
the unique normalized member of this homothetic family is the unit round
metric corresponding to $r=1$.

\subsubsection*{The extended Einstein Cauchy-data space}
Because the Einstein constant varies along the homothetic family, it is useful
first to place all round metrics in an extended Cauchy-data space in which
$\lambda$ is allowed to vary.
Writing
$g=dt^2+a(t)^2g_{S^p}+b(t)^2g_{S^q},$
and setting
$$u=a',\qquad v=b',$$
the Einstein equations reduce to the usual cohomogeneity-one Einstein
equations together with the Hamiltonian constraint
$\mathcal C_\lambda(a,b,u,v)=0.$
Define the extended admissible Cauchy-data space by
$$
\widehat{\mathcal P}=
\left\{(a,b,u,v,\lambda):
a>0,\,b>0,\,
\lambda>0,\,
\mathcal C_\lambda(a,b,u,v)=0
\right\}.
$$
For each fixed value of $\lambda$, the corresponding normalized Cauchy-data
manifold is the slice
$$
\mathcal P_\lambda=
\left\{
(a,b,u,v):
a>0,\,
b>0,\,
\mathcal C_\lambda(a,b,u,v)=0
\right\}.
$$
The mean curvature is
$H=p\frac{u}{a}+q\frac{v}{b}.$
Hence the extended matching hypersurface is
$$
\widehat{\Sigma}=
\left\{
(a,b,u,v,\lambda)\in\widehat{\mathcal P}:H=0
\right\},
$$
while for fixed $\lambda$ the normalized matching hypersurface is
$$
\Sigma_\lambda
=\left\{
(a,b,u,v)\in\mathcal P_\lambda:H=0
\right\}.
$$

\subsubsection*{The Einstein germ spaces}
At the left singular orbit, the smoothness conditions are
$$
a(0)=0,\qquad a'(0)=1,\qquad
b(0)>0,\qquad b'(0)=0,
$$
while the analogous conditions hold at the right singular orbit.
For the homothetic round family, the corresponding smooth Einstein germs are
$$\gamma_-(r),\qquad\gamma_+(r),$$
and therefore
$$\widehat G_-^{\rd} =\{\gamma_-(r):r>0\}, \qquad \widehat G_+^{\rd}=\{\gamma_+(r):r>0\}.$$
Both are naturally identified with $(0,\infty)$.
After fixing the Einstein normalization $\lambda=p+q$, one has $r=1$, and the
normalized round germ spaces reduce to the singletons
$$G_-^{\rd}= \{\gamma_-(1)\},\qquad G_+^{\rd} =\{\gamma_+(1)\}.$$

\subsubsection*{The matching point}
For the round metric $g_r$, the mean curvature is
$$
H_r(t)=\frac{1}{r}
\left(p\cot\left(\frac{t}{r}\right)
-q\tan\left(\frac{t}{r}\right)\right).
$$
The equation $H_r(t)=0$ is equivalent to
$\tan^2(t/r)=p/q$. Since $0<t/r<\pi/2$, it has the unique solution
$$t_*(r)=r\arctan\sqrt{\frac{p}{q}}.$$
At this point,
$$a_r(t_*(r)) =r\sqrt{\frac{p}{p+q}},\qquad b_r(t_*(r))=r\sqrt{\frac{q}{p+q}},$$
and
$$a_r'(t_*(r))=\sqrt{\frac{q}{p+q}}, \qquad b_r'(t_*(r))=-\sqrt{\frac{p}{p+q}}.$$
Hence the extended matching datum is
$$
\widehat m(r)=\left(
r\sqrt{\frac{p}{p+q}},
r\sqrt{\frac{q}{p+q}},
\sqrt{\frac{q}{p+q}},
-\sqrt{\frac{p}{p+q}},
\frac{p+q}{r^2}
\right)\in\widehat{\Sigma}.
$$
Equivalently, after fixing $\lambda=\lambda_r$, the corresponding normalized
matching datum is
$$
m_{\lambda_r}(r)
=\left(
r\sqrt{\frac{p}{p+q}},
r\sqrt{\frac{q}{p+q}},
\sqrt{\frac{q}{p+q}},
-\sqrt{\frac{p}{p+q}}
\right)\in\Sigma_{\lambda_r}.
$$
In the fixed normalization $\lambda=p+q$, one has $r=1$ and therefore
$$
m_0=\left(
\sqrt{\frac{p}{p+q}},
\sqrt{\frac{q}{p+q}},
\sqrt{\frac{q}{p+q}},
-\sqrt{\frac{p}{p+q}}
\right)\in\Sigma_{p+q}.
$$
\smallskip
These formulas satisfy the matching equation directly:
$$p\,a_r'(t_*)/a_r(t_*)+q\,b_r'(t_*)/b_r(t_*)=0.$$
Thus the displayed matching datum lies on the stated hypersurface without any
additional numerical input.

\subsubsection*{Propagation maps}
In the extended formulation, propagation of the left and right round Einstein
germs to the extended matching hypersurface gives
$\widehat\Phi_-^{\rd}:\widehat G_-^{\rd}\longrightarrow\widehat\Sigma,$
and
$\widehat\Phi_+^{\rd}:\widehat G_+^{\rd}\longrightarrow
\widehat\Sigma,
$
with
$$\widehat\Phi_-^{\rd}(\gamma_-(r)) =\widehat m(r)= \widehat\Phi_+^{\rd}(\gamma_+(r)).$$
After fixing the normalization $\lambda=p+q$, the propagation maps reduce to
$
\Phi_-^{\rd}:G_-^{\rd}\longrightarrow\Sigma_{p+q},
$
and
$
\Phi_+^{\rd}
:G_+^{\rd}
\longrightarrow
\Sigma_{p+q},
$
and both send the unique round germ to the matching point $m_0$.
\subsubsection*{The matching problem}
Before choosing coordinates, the extended round matching problem is encoded by
$$
\widehat\Psi^{\rd}
=\left(
\widehat\Phi_-^{\rd},
\widehat\Phi_+^{\rd}
\right):
\widehat G_-^{\rd}
\times
\widehat G_+^{\rd}
\longrightarrow
\widehat\Sigma
\times
\widehat\Sigma.
$$
A pair of round germs extends to one global round Einstein metric precisely
when
\[
\widehat\Psi^{\rd}\bigl(\gamma_-(r_-),\gamma_+(r_+)\bigr)
\in\Delta_{\widehat\Sigma}.
\]
Equivalently,
\[
\widehat m(r_-)=\widehat m(r_+).
\]
Since the final component of $\widehat m(r)$ is
$\frac{p+q}{r^2},$
this equality is equivalent to
$r_-=r_+.$
Thus, under the identification of round Einstein metrics with compatible pairs
of endpoint germs, the extended matching locus is
$\left\{ (r_-,r_+)\in(0,\infty)^2: r_-=r_+ \right\}.$
This locus corresponds precisely to the homothetic family of round Einstein
metrics.

\subsubsection*{A coordinate matching residual}
To recover the coordinate formulation of the matching map, choose a local
coordinate chart
$\chi:V\subset\widehat\Sigma\longrightarrow \mathbb R^{\dim\widehat\Sigma}$
near the reference matching point. The corresponding extended round matching
map is
$$
\widehat{\mathcal M}^{\rd}_\chi(r_-,r_+)
=\chi\left(\widehat m(r_-)\right)-
\chi\left(\widehat m(r_+)\right).
$$
Its zero set is coordinate independent and satisfies
$$\left( \widehat{\mathcal M}^{\rd}_\chi \right)^{-1}(0)=\{(r,r):r>0\}.$$
If one instead uses the ambient coordinates $(a,b,u,v,\lambda)$ merely to
define an ambient matching residual, then the residual is represented by
$$
\widehat m(r_-)-\widehat m(r_+)
=\left(
(r_--r_+)\sqrt{\frac{p}{p+q}},
(r_--r_+)\sqrt{\frac{q}{p+q}},
0,
0,
(p+q)
\left(
\frac{1}{r_-^2}-
\frac{1}{r_+^2}
\right)
\right).
$$
This ambient residual has the same zero set, although it should not be
confused with a coordinate representation of the matching map on
$\widehat\Sigma$ itself.
\subsubsection*{The normalized round problem}
After fixing the Einstein normalization $\lambda=p+q$, the homothetic
parameter is fixed by $r=1$. The normalized germ spaces are singletons, the
normalized propagation maps both take the value $m_0$, and the normalized
matching locus consists of the single compatible pair
$(\gamma_-(1),\gamma_+(1)).$
Thus, within the homothetic round subproblem considered in this example, the
fixed normalization leaves only the unit round metric. No uniqueness claim is
made here for the full cohomogeneity-one Einstein problem with the same group
action.
\begin{Bem}
This example illustrates two distinct but compatible viewpoints. The extended
formulation displays the full homothetic round family by allowing the Einstein
constant to vary, while the normalized formulation fits exactly into the
fixed-$\lambda$ theory developed in the preceding subsections. In both
formulations, the germ spaces, the Cauchy-data space, the matching
hypersurface, the propagation maps and the matching equation can be computed
explicitly.
For more complicated cohomogeneity-one Einstein metrics, such as B\"ohm
metrics, the same geometric construction applies, although the propagation
maps are generally no longer available in closed form.
\end{Bem}

\subsection{The analytical obstacle}
\label{sec:remaining-problem}
Section~2 reduces the local normalized Einstein boundary-value problem to a
finite-dimensional Einstein matching problem. The differential of the
matching map controls the local structure of its zero set, but the analytical
theory developed below does not assume that this differential can already be
recovered from the auxiliary probe construction.

\smallskip
The first analytical task is more modest and logically prior: given a chosen
smooth parameterized Einstein family, construct canonical analytical data that
vary smoothly along that family and determine their first-order response to
Einstein variations. This construction requires only the harmonic probe
problem and normalized Jacobi nondegeneracy.

\smallskip
Rather than studying the linearized Einstein evolution directly, we associate,
under the normalized Jacobi nondegeneracy hypothesis introduced in the next
section, a canonically determined equivariant harmonic map with every nearby
Einstein metric belonging to the chosen smooth parameterized Einstein family.
These auxiliary harmonic probes possess elliptic, variational and spectral
structures while introducing no additional Einstein degrees of freedom.

\smallskip
The probes generate the core analytical package consisting of the Jacobi and
Green operators. After the separate naturality hypothesis is imposed, this
package descends to the chosen parameterized Einstein family. Its differential
then defines the universal response differential used in the Einstein
Detection Principle.

\smallskip
A further connection with the Einstein matching construction is available
only under Hypothesis~\ref{hyp:matching-dependence}. Under that additional
hypothesis the package extends smoothly to matching data and its differential
factors through the infinitesimal matching-point map. Thus the logical
dependence is
$$
\text{probe existence}
\Longrightarrow
\text{core package on }S
\Longrightarrow
\text{response theory on }S,
$$
while the separate implication
$$
\text{matching-data extension}
\Longrightarrow
\text{factorization through }\Sigma
$$
connects that response theory to the Einstein matching hypersurface.
Neither implication contains the other.

\smallskip
The remainder of the paper develops these two layers separately. This
separation is essential: the Einstein Detection Principle on a chosen smooth
parameterized Einstein family depends on the first layer, whereas statements
about arbitrary tangent directions of the matching hypersurface additionally
depend on the second.

\section{Canonical Harmonic Probes}
\label{sec:canonical-harmonic-probes}
The preceding section reduced the local normalized Einstein boundary-value
problem to the finite-dimensional Einstein matching problem. The remaining
analytical challenge is to understand the differential of the Einstein
matching map.

\smallskip
Rather than analysing the linearized Einstein evolution directly, we
introduce an auxiliary variational theory based on canonically associated
equivariant harmonic maps. Under the normalized Jacobi nondegeneracy
hypothesis introduced below, every nearby Einstein metric belonging to the
chosen smooth parameterized Einstein family admits a unique canonically
determined equivariant harmonic map. These harmonic maps serve as
\emph{auxiliary variational probes} of the underlying Einstein geometry.

\smallskip
The harmonic probes do not modify the Einstein equations and introduce no
additional Einstein degrees of freedom. Their purpose is entirely analytical.
Through their Jacobi theory they generate a canonical analytical package
consisting of Jacobi, Green and response operators together with their
associated analytical invariants. After canonical elimination of the
auxiliary probe variables, these analytical structures descend naturally to
the chosen smooth parameterized Einstein family and form the basis of the
response theory developed in the remainder of the paper.

\smallskip
Throughout this section, let $(M,g_0)$ denote a fixed normalized
cohomogeneity-one Einstein manifold. Let
$S$
be the chosen smooth parameterized Einstein family introduced in
Section~2, choose a reference parameter $p_0\in S$ with
$g_{p_0}=g_0$, and carry out the analytical constructions locally near
$p_0$ in $S$.

\subsection{Probe data}
Before constructing the canonical harmonic probes we fix the auxiliary
geometric data defining the probe problem. Once these data have been fixed,
every subsequent construction is canonical relative to this choice.
\smallskip
Throughout the remainder of the paper we therefore fix the following probe
data.
\begin{enumerate}
\item a compact Riemannian target $(N,h)$;
\item a compatible action of the symmetry group on $N$;
\item an equivariant homotopy class of maps $f:M\longrightarrow N$;
\item the endpoint, smoothness and normalization conditions used to remove the
relevant probe symmetries;
\item a normalized reference harmonic probe $f_0:(M,g_0)\longrightarrow(N,h)$
in the chosen class.
\end{enumerate}

Accordingly, every subsequent use of the term \emph{canonical} means the
unique local continuation of this fixed reference probe, relative to the
Einstein normalization and the complete probe data above. It does not assert
global uniqueness of harmonic maps in the chosen homotopy class.

\smallskip
In the applications considered later, the target manifold $N$ will typically
be a compact homogeneous space, so that equivariance reduces the harmonic-map
equation to a nonlinear ordinary differential boundary-value problem on the
orbit interval. The local theory developed below, however, depends only on
the abstract properties stated in the following hypotheses and does not
require a particular choice of target manifold.

\smallskip
The existence and uniqueness theorem below shows that, once these data have
been fixed, the reference probe has a unique local continuation for nearby
metrics. Consequently, the resulting local probe branch introduces no
additional Einstein degrees of freedom.

\subsection{Equivariant mapping spaces}
Having fixed the auxiliary probe data, we now introduce the functional-analytic
setting in which the harmonic probe equation is studied. In contrast to the
finite-dimensional Einstein matching problem developed in the previous section,
the harmonic-map equation is naturally formulated on Banach manifolds.

\smallskip
Throughout this section we work within the functional-analytic framework fixed
in Section~2. In particular, all mapping spaces, bundles and operator families
are understood with the Banach-space conventions introduced there.

\smallskip
Assume, as part of the standing functional-analytic hypotheses, that the
prescribed equivariance, endpoint, regularity and normalization constraints
cut out a smooth Banach submanifold of the $C^{k+2,\alpha}$ mapping space.
Denote this Banach manifold by $\mathcal H^{k+2,\alpha}$.

\smallskip
For every
$f\in\mathcal H^{k+2,\alpha},$
the tangent space is
$T_f\mathcal H^{k+2,\alpha} \subset C_{\mathrm{eq}}^{k+2,\alpha}(f^*TN),$
consisting of equivariant sections satisfying the corresponding linearized
boundary and normalization conditions.

\smallskip
The additional two derivatives relative to the metric regularity ensure that
the tension field is a well-defined smooth nonlinear elliptic operator.

\smallskip
Accordingly, the tension field defines a smooth section of the associated
Banach bundle over
$\mathcal H^{k+2,\alpha}.$
Its zero set consists precisely of the equivariant harmonic maps compatible
with the fixed probe data.

\smallskip
The linearization of the tension field at a smooth harmonic map is the
equivariant Jacobi operator. By the conventions established in
Section~2, this operator is regarded as an elliptic Fredholm operator acting
between fixed Banach spaces after local trivialization. Whenever spectral
properties are considered, we pass to its corresponding self-adjoint
$L^2$ realization.

\smallskip
The functional-analytic framework employed here is standard in equivariant
harmonic-map theory; see
\cite{EellsRatto,Palais,TaylorPDE2}.

\smallskip
The next subsection studies the local solution set of the harmonic-map
equation in this Banach manifold and establishes the existence of the
canonical harmonic probes associated with nearby Einstein metrics.

\subsection{The nonlinear probe operator}
The Banach manifold introduced above provides the natural functional-analytic
setting for the harmonic probe equation. We now introduce the nonlinear
operator whose zero set defines the auxiliary harmonic-map problem. This
operator forms the analytical core of the probe construction.

\smallskip
For every Riemannian metric sufficiently close to the reference metric
$g_0$, define
$\mathcal F(g,f)=\tau_g(f),$
where
$f:(M,g)\longrightarrow (N,h)$
is an equivariant map and $\tau_g(f)$ denotes its tension field.

\smallskip
Thus the nonlinear equation
$\mathcal F(g,f)=0$
is precisely the equivariant harmonic-map equation. Its solution set consists
of those pairs $(g,f)$ for which $f$ is harmonic with respect to the metric
$g$. The canonical harmonic probes constructed below are obtained by
restricting this solution set to metrics belonging to the chosen smooth
parameterized Einstein family.

\smallskip
After passing to the fixed local trivializations introduced in
Section~2, the probe equation is represented by a smooth map
$\mathcal F: \mathcal G\times \mathcal H^{k+2,\alpha} \longrightarrow \mathcal Y^{k,\alpha},$
where
$\mathcal G$
is a Banach manifold of normalized cohomogeneity-one metrics containing the
chosen smooth parameterized Einstein family
$S\subset\mathcal G,$
and
$\mathcal Y^{k,\alpha}$
denotes the Banach bundle of equivariant
$C^{k,\alpha}$ sections introduced above.

\smallskip
The smoothness of $\mathcal F$ is standard. It follows from the smooth
dependence of the Levi--Civita connection, the trace operator and the
nonlinear composition terms on both the metric and the map variables; see
\cite{EellsRatto,Palais,TaylorPDE2}.

\begin{Bem}
For cohomogeneity-one manifolds, the above nonlinear elliptic equation
reduces to a singular ordinary differential boundary-value problem on the
orbit interval. Consequently, every construction developed below admits an
equivalent formulation in terms of ordinary differential equations.
\end{Bem}

\subsection{Reference probe and normalized Jacobi nondegeneracy}
We next linearize the probe equation at the distinguished reference probe.
The resulting Jacobi operator governs the local solvability of the nonlinear
harmonic-map equation and becomes the fundamental analytical object
throughout the remainder of the paper.

\smallskip
Let
$f_0:M\longrightarrow N$
be the distinguished equivariant harmonic map associated with the reference
Einstein metric
$g_0.$
Thus
$\mathcal F(g_0,f_0)=0.$
The derivative of the nonlinear probe operator with respect to the map
variable is the normalized Jacobi operator
$$
J_{g_0}
=
D_f\mathcal F(g_0,f_0):
T_{f_0}\mathcal H^{k+2,\alpha}
\longrightarrow
\mathcal Y^{k,\alpha}_{f_0}.
$$
The imposed normalization removes the infinitesimal symmetries associated
with the chosen normalization conditions. Additional Jacobi fields may,
however, still exist. To exclude them, we impose the following
nondegeneracy hypothesis.

\begin{Hyp}[Normalized Jacobi nondegeneracy]
\label{hyp:normalized-jacobi-nondegeneracy}
The normalized Jacobi operator
$J_{g_0}: T_{f_0}\mathcal H^{k+2,\alpha} \longrightarrow \mathcal Y^{k,\alpha}_{f_0}$
is an isomorphism.
\end{Hyp}

Under this hypothesis, the Banach-space Implicit Function Theorem yields a
unique smooth family of equivariant harmonic maps depending on the ambient
metric. Restricting this family to the chosen smooth parameterized Einstein
family
$S$
therefore produces a unique smoothly varying family of canonical harmonic
probes.

\smallskip
The normalized Jacobi nondegeneracy hypothesis is the fundamental analytical
assumption underlying the present paper. It guarantees the existence,
uniqueness and smooth dependence of the canonical harmonic probes and
provides the starting point for the Jacobi, Green and response theories
developed in the subsequent sections.

\subsection{Local existence and uniqueness of normalized probes}
\label{sec:probe-existence-smoothness}
The preceding subsections introduced the nonlinear probe operator and the
normalized Jacobi operator. Under the normalized Jacobi nondegeneracy
hypothesis, the Banach-space Implicit Function Theorem yields a unique smooth
family of equivariant harmonic maps depending on the ambient metric.
Restricting this family to the chosen smooth parameterized Einstein family
produces the canonical harmonic probes.

\smallskip
This result provides the analytical foundation for the probe assignment used
throughout the remainder of the paper.
\begin{Thm}[Local existence and uniqueness of canonical harmonic probes]
\label{thm:probe-existence}
Assume Hypothesis~\ref{hyp:normalized-jacobi-nondegeneracy} and the standing
Banach-bundle and smoothness hypotheses of Section~2, so that after local
trivialization $\mathcal F$ is a smooth map between fixed Banach spaces and
$D_f\mathcal F(g_0,f_0)$ is the bounded isomorphism appearing in that
hypothesis.

\smallskip
Then there exist neighbourhoods
$$g_0\in\mathcal U\subset\mathcal G, \qquad f_0\in\mathcal V\subset\mathcal H^{k+2,\alpha},$$
and a unique smooth map
$$\Pi:\mathcal U\longrightarrow\mathcal V, \qquad \Pi(g)=f_g,$$
such that
$\mathcal F(g,f_g)=0$
for every
$g\in\mathcal U.$
Moreover, every solution of
$\mathcal F(g,f)=0$
lying in
$\mathcal V$
coincides with
$f_g.$
Let
$S_{\mathcal U}:=\mathbf g^{-1}(\mathcal U)=\{p\in S:g_p\in\mathcal U\}$.
Then the composition $p\mapsto\Pi(g_p)$ is a unique smooth family of
canonical harmonic probes on $S_{\mathcal U}$ near $p_0$.
\end{Thm}
\begin{proof}
The map $\mathcal F$ is smooth and
$D_f\mathcal F(g_0,f_0)=J_{g_0}$ is an isomorphism by
Hypothesis~\ref{hyp:normalized-jacobi-nondegeneracy}. The Banach-space
Implicit Function Theorem therefore gives neighbourhoods $\mathcal U,\mathcal
V$ and a unique smooth map $\Pi:\mathcal U\to\mathcal V$ with
$\mathcal F(g,\Pi(g))=0$. Local uniqueness is part of the same theorem.
Smoothness of the resulting probe follows from the stated elliptic, or
equivalently equivariant regular-singular, regularity theory.
\end{proof}
\begin{Dfn}[Canonical harmonic probe assignment]
With $S_{\mathcal U}=\mathbf g^{-1}(\mathcal U)$, the map
$\Pi_E:S_{\mathcal U}\longrightarrow\mathcal H^{k+2,\alpha}$ defined by
$\Pi_E(p)=\Pi(g_p)$ is called the \emph{canonical harmonic probe assignment}
on the chosen Einstein family. When no confusion can arise we continue to
write $f_p=f_{g_p}$ and suppress the subscript $E$.
\end{Dfn}

\subsection{The core analytical package}
The canonical harmonic probe assignment determines a collection of analytical
objects that are available under the hypotheses established so far. We refer
to this collection as the \emph{core analytical package}. It consists of the
canonical operator-theoretic structures required for the response theory
developed in the remainder of the paper and does not require any additional
spectral assumptions.

\begin{Dfn}[Core analytical package]
For every $p\in S_{\mathcal U}$, the \emph{core analytical package} is
$\mathcal A_{\mathrm{core}}(p)=(J_p,G_p)$, where
$J_p:=J_{g_p}$ denotes the normalized Jacobi operator of the probe $f_p$ and
$G_p:=J_p^{-1}$. Thus $\mathcal A_{\mathrm{core}}$ is a map on the parameter
manifold; its dependence on the Einstein metric is through
$p\mapsto g_p$.
\smallskip
The response operator introduced in the subsequent sections is constructed
canonically from this core analytical package.
\end{Dfn}
\smallskip
The core analytical package is canonically attached to the chosen smooth
parameterized Einstein family through the canonical harmonic probe
assignment. Every theorem concerning the Einstein Detection Principle depends
only on this package and on the structures canonically derived from it.

\subsection{Extended spectral data}
The core analytical package introduced above is sufficient for the
construction of the response theory and for all results concerning the
Einstein Detection Principle.

\smallskip
Under additional analytical hypotheses, however, the core package may be
enriched by further spectral information associated with the normalized
Jacobi operator. Depending on the particular setting, these additional
structures may include spectral projectors, heat kernels, zeta functions,
regularized determinants and Morse indices.

\smallskip
These objects are collectively referred to as the \emph{extended spectral
data}. They are not part of the core analytical package and are introduced
only when the corresponding additional hypotheses are explicitly assumed.
Unless stated otherwise, no theorem in the present paper depends on these
extended spectral invariants.

\smallskip
The canonical harmonic probe assignment allows every Einstein metric in the
chosen smooth parameterized Einstein family to be studied through its
associated harmonic probe. Consequently, every operator-theoretic
construction developed in the subsequent sections is obtained by applying
canonical constructions to this smoothly varying family of harmonic probes.
The response theory itself depends only on the core analytical package,
whereas the extended spectral data provide additional information when
available.

\begin{Bem}
\begin{enumerate}
\item
Theorem~\ref{thm:probe-existence} provides the analytical foundation for the
entire probe framework. Rather than producing an isolated harmonic map, it
constructs a smooth family of canonical harmonic probes parametrized by the
chosen smooth parameterized Einstein family. This smooth dependence on the
Einstein metric is the essential ingredient underlying all subsequent
operator-theoretic constructions.
\item
The Jacobi, Green and response operators developed in the following sections
belong to the core analytical package. Additional spectral invariants are
introduced only when required for specific applications and always under
their corresponding analytical hypotheses.
\item
The probe construction is purely local. No global existence or uniqueness
statement for harmonic probes is asserted.
\end{enumerate}
\end{Bem}

\subsection{A canonical probe on the round sphere}
\label{subsec:round-sphere-probe}
We now illustrate the abstract construction in a completely explicit setting.
This example serves three purposes.
\begin{enumerate}
\item
It exhibits the canonical harmonic probe associated with a distinguished
Einstein metric.
\item
It verifies the normalized Jacobi nondegeneracy hypothesis by an explicit
spectral computation.
\item
It illustrates the construction of the canonical harmonic probe assignment
provided by Theorem~\ref{thm:probe-existence}.
\end{enumerate}

\begin{example}[The identity probe on the round sphere]
\label{ex:round-probe}
We identify the canonical harmonic probe, compute its Jacobi operator,
verify Hypothesis~\ref{hyp:normalized-jacobi-nondegeneracy}, and finally
apply Theorem~\ref{thm:probe-existence}.
\subsubsection*{The canonical harmonic probe}
Let
$$
(M,g_0)
=
\left(
S^{n+1},
dt^2+\sin^2(t)\,g_{S^n}
\right),
\qquad
0\le t\le\pi,
$$
where
$n\ge2.$
As target we choose the same round sphere
$$(N,h) = \left( S^{n+1}, du^2+\sin^2(u)\,g_{S^n} \right).$$
The round metric satisfies
$\Ric(g_0)=n\,g_0.$
We let
$SO(n+1)$
act in the standard way on the
$S^n$
factor of both the domain and target.

\smallskip
We consider equivariant maps of degree one of the form
$$f_u(t,\omega) = (u(t),\omega),$$
subject to
$$u(0)=0, \qquad u(\pi)=\pi,$$
together with the standard smoothness conditions at the singular orbits.

\smallskip
Up to the inessential constant factor
$\frac12\operatorname{vol}(S^n),$
the harmonic-map energy is
$$E(u) = \int_0^\pi \left( u'^2 + n\frac{\sin^2u}{\sin^2t} \right) \sin^nt\,dt.$$
Its Euler--Lagrange equation is
$$u'' + n\cot t\,u' - \frac n2 \frac{\sin(2u)}{\sin^2t} = 0.$$
The identity solution
$u_0(t)=t$
satisfies this equation.

\smallskip
Hence
$$f_0 = \operatorname{Id}_{S^{n+1}} : (S^{n+1},g_0) \longrightarrow (S^{n+1},g_0)$$
is the distinguished harmonic map associated with the reference Einstein
metric and therefore serves as the reference harmonic probe.

\subsubsection*{The normalized Jacobi operator}
To verify
Hypothesis~\ref{hyp:normalized-jacobi-nondegeneracy},
consider
$$u_\varepsilon = t+\varepsilon v + O(\varepsilon^2).$$
Linearization of the harmonic-map equation at
$u_0=t$
gives
$$v'' + n\cot t\,v' - n \frac{\cos(2t)}{\sin^2t} v = 0.$$
The corresponding normalized Jacobi operator is
$$J_0v = - v'' - n\cot t\,v' + n \frac{\cos(2t)}{\sin^2t} v.$$
The normalized endpoint conditions imply
$v(0)=v(\pi)=0.$
Every smooth equivariant variation therefore has the form
$v(t)=\sin t\,w(t),$
where
$w\in C^\infty([0,\pi]).$
Substituting this expression yields
$$J_0(\sin t\,w) = \sin t \left( - w'' - (n+2)\cot t\,w' + (1-n)w \right).$$

\subsubsection*{Spectral verification}
The operator
$$\mathcal L_{n+2} = - w'' - (n+2)\cot t\,w'$$
is precisely the radial Laplacian on the unit sphere
$S^{n+3},$
acting on the weighted space with measure
$\sin^{\,n+2}t\,dt.$
Its spectrum is
$$\mu_k = k(k+n+2), \qquad k=0,1,\ldots.$$
Consequently,
$$\lambda_k = k(k+n+2)+1-n$$
are the eigenvalues of the normalized Jacobi operator.

\smallskip
In particular,
$$\lambda_0 = 1-n < 0, \qquad \lambda_1 = 4,$$
and therefore
$$\lambda_k\neq0 \qquad (k\ge0).$$
Hence
$$0 \notin \operatorname{spec}(J_0),$$
so
$\ker J_0 = \{0\}.$
Since
$J_0$
is a self-adjoint regular-singular Sturm--Liouville operator with compact
resolvent, its self-adjoint
$L^2$
realization is bijective.

\smallskip
Standard elliptic regularity for regular-singular operators therefore implies
$$J_0: T_{f_0}\mathcal H^{k+2,\alpha} \longrightarrow \mathcal Y^{k,\alpha}_{f_0}$$
is an isomorphism.

\smallskip
Thus the normalized Jacobi nondegeneracy hypothesis is verified completely
explicitly in this example.

\subsubsection*{Construction of the canonical harmonic probes}
Theorem~\ref{thm:probe-existence}
therefore provides neighbourhoods
$$g_0\in\mathcal U, \qquad f_0\in\mathcal V,$$
together with a unique smooth harmonic probe assignment
$$\Pi: \mathcal U \longrightarrow \mathcal V, \qquad g \longmapsto f_g,$$
satisfying
$\tau_g(f_g)=0$
for every
$g\in\mathcal U.$
Restricting this assignment to the chosen smooth parameterized Einstein
family produces the canonical harmonic probes associated with nearby
Einstein metrics.

\smallskip
Accordingly, the identity map is a normalized nondegenerate reference probe,
and the abstract construction developed above becomes completely explicit in
this model example.
\end{example}

\begin{Bem}
\label{rem:round-probe-n-equals-one}
The restriction
$n\ge2$
is essential.
Indeed,
$\lambda_0=1-n,$
so for
$n=1$
one obtains
$\lambda_0=0.$
Consequently, the normalized Jacobi operator is no longer invertible.

\smallskip
This reflects the existence of a nontrivial family of harmonic self-maps of
the round two-sphere generated by conformal transformations. Additional
normalization conditions would therefore be required before
Theorem~\ref{thm:probe-existence} could be applied.
\end{Bem}

\section{Harmonic Probe Geometry}
The previous section established the existence of a canonical harmonic probe
assignment on the chosen smooth parameterized Einstein family. We now study
the differential geometry of this assignment.

\smallskip
Differentiating the probe assignment produces the first analytical structures
associated with the family of harmonic probes, namely the smoothly varying
Jacobi operators, Green operators and their interaction with infinitesimal
Einstein deformations. These constitute the analytical core from which the
response theory developed in the subsequent sections is built.

\subsection{Infinitesimal variation of the probe assignment}
The first variation of the canonical harmonic probe assignment provides the
fundamental analytical link between infinitesimal Einstein deformations and
the response of their associated harmonic probes. It is also the point at
which the Green operator appears naturally.

\smallskip
Let $p\in S_{\mathcal U}$ and let $X\in T_pS_{\mathcal U}$. Choose a smooth
curve $(-\varepsilon,\varepsilon)\ni t\longmapsto p(t)\in S_{\mathcal U}$
with $p(0)=p$ and $\dot p(0)=X$. Write $g_p=\mathbf g(p)$ and
$f_p=\Pi(g_p)$. Since the canonical harmonic probes satisfy
$$\mathcal F(g_{p(t)},f_{p(t)})=0,$$ differentiation at $t=0$ gives
$$
D_g\mathcal F(g_p,f_p)\bigl[D\mathbf g_p(X)\bigr]
+
D_f\mathcal F(g_p,f_p)\bigl[D(\Pi\circ\mathbf g)_p(X)\bigr]=0.
$$
Define $J_p:=D_f\mathcal F(g_p,f_p)$ and define the forcing operator along the
parameterized Einstein family by
$$
K_p:=D_g\mathcal F(g_p,f_p)\circ D\mathbf g_p:
T_pS_{\mathcal U}\longrightarrow\mathcal Y^{k,\alpha}_{f_p}.
$$
Thus $K_p$ is the pullback to parameter space of the ambient metric derivative
of the probe equation. This distinction is essential when the parameterization
$\mathbf g:S\to\operatorname{Met}(M)$ is not an embedding.

\smallskip
By Theorem~\ref{thm:probe-existence}, the ambient Jacobi family depends
smoothly on the metric, and hence $p\longmapsto J_p$ depends smoothly on the
parameter. Since $J_{p_0}=J_{g_0}$ is an isomorphism by
Hypothesis~\ref{hyp:normalized-jacobi-nondegeneracy}, invertibility persists
after shrinking $S_{\mathcal U}$ if necessary. Writing $G_p:=J_p^{-1}$, the
differentiated probe equation becomes
$$
D(\Pi\circ\mathbf g)_p(X)=-G_pK_p(X).
$$

\begin{Prop}[Linear response identity]
\label{prop:probe-response}
After shrinking $S_{\mathcal U}$ if necessary, the canonical harmonic probe
assignment $\Pi_E=\Pi\circ\mathbf g:S_{\mathcal U}\to\mathcal H^{k+2,\alpha}$
satisfies
$$
D(\Pi_E)_p=-G_pK_p:
T_pS_{\mathcal U}\longrightarrow T_{f_p}\mathcal H^{k+2,\alpha}.
$$
\end{Prop}
\begin{proof}
The identity follows from differentiating
$\mathcal F(\mathbf g(p),\Pi(\mathbf g(p)))=0$ and inverting $J_p$.
\end{proof}
\smallskip
This is the fundamental linear response law for the parameterized Einstein
family. If $\mathbf g$ is immersive, it may equivalently be read as a response
law for the corresponding infinitesimal metric deformations
$D\mathbf g_p(X)$. If $\mathbf g$ has redundant parameter directions, the
formula correctly annihilates them through the factor $D\mathbf g_p$.

\subsection{The harmonic probe bundle}
\label{sec:harmonic-probe-bundle}
The canonical harmonic probe assignment
$\Pi_E=\Pi\circ\mathbf g:S_{\mathcal U}\to\mathcal H^{k+2,\alpha}$
associates the unique local continuation of the fixed reference harmonic probe
with every parameter $p\in S_{\mathcal U}$. Write $f_p=\Pi_E(p)$ and define
$\mathscr H_p^{k+2,\alpha}:=T_{f_p}\mathcal H^{k+2,\alpha}$.

\begin{Dfn}[Harmonic probe bundle]
The \emph{harmonic probe bundle} is the smooth Banach vector bundle
$\pi_{\mathscr H}:\mathscr H^{k+2,\alpha}\to S_{\mathcal U}$ whose fibre over
$p$ is $\mathscr H_p^{k+2,\alpha}$.
\end{Dfn}
\smallskip
The smooth bundle structure is understood with respect to the local
trivializations fixed in the functional-analytic conventions of Section~2.
After shrinking $S_{\mathcal U}$ if necessary, the smoothly varying pullback
bundles $f_p^*TN$ and the subspaces satisfying the linearized equivariance,
endpoint and normalization conditions are identified with fixed Banach spaces.
Thus, locally,
$$
\mathscr H^{k+2,\alpha}\cong
S_{\mathcal U}\times\mathscr H_{p_0}^{k+2,\alpha}.
$$
Similarly, set $\mathscr Y_p^{k,\alpha}:=\mathcal Y_{f_p}^{k,\alpha}$.
These spaces assemble into a smooth Banach vector bundle
$\pi_{\mathscr Y}:\mathscr Y^{k,\alpha}\to S_{\mathcal U}$, locally
isomorphic to
$S_{\mathcal U}\times\mathscr Y_{p_0}^{k,\alpha}$.

\smallskip
For $p\in S_{\mathcal U}$ define
$$
J_p:=D_f\mathcal F(g_p,f_p):
\mathscr H_p^{k+2,\alpha}\longrightarrow\mathscr Y_p^{k,\alpha}.
$$
The family $p\mapsto J_p$ is smooth after local trivialization and hence
defines a smooth bundle morphism
$J:\mathscr H^{k+2,\alpha}\to\mathscr Y^{k,\alpha}$ covering the identity on
$S_{\mathcal U}$. Equivalently, in a fixed local trivialization it is a smooth
map
$$
S_{\mathcal U}\longrightarrow
\mathcal L\!\left(
\mathscr H_{p_0}^{k+2,\alpha},
\mathscr Y_{p_0}^{k,\alpha}
\right).
$$
After shrinking $S_{\mathcal U}$ if necessary,
Hypothesis~\ref{hyp:normalized-jacobi-nondegeneracy} and openness of
invertibility imply that every $J_p$ is an isomorphism.

\smallskip
The harmonic probe bundle, target bundle and Jacobi bundle morphism are thus
bundles over the \emph{parameter manifold} $S_{\mathcal U}$, while their
coefficients are determined by the associated metrics $g_p$. This separation
of base point $p$ from metric $g_p$ will be maintained below.

\subsection{The analytical bridge to the Einstein matching map}
\label{subsec:analytical-bridge-matching-map}
Theorem~\ref{thm:intrinsic-reduction} expresses the normalized Einstein
boundary-value problem as the finite-dimensional Einstein matching problem.
In a local coordinate chart on the Einstein matching hypersurface, the
Einstein matching map is represented by
$$\mathcal M_\chi(u,v) = \chi(\Phi_-(u)) - \chi(\Phi_+(v)).$$
At a matching pair $(u_0,v_0)$ with common matching point
$q_0 = \Phi_-(u_0) = \Phi_+(v_0),$
its differential is therefore
$$D\mathcal M_\chi(u_0,v_0) = D\chi_{q_0}\circ D\Phi_-(u_0) - D\chi_{q_0}\circ D\Phi_+(v_0).$$
Thus the first-order Einstein matching problem is governed by the
differentials of the two propagation maps. The purpose of the canonical
harmonic probes is to provide additional analytical structures through which
these propagated infinitesimal variations can be studied.

\smallskip
The connection between the probe theory and the Einstein matching problem
requires an additional compatibility assumption. Namely, the analytical
probe data associated with a propagated Einstein germ should depend only on
the resulting matching datum and not on the particular endpoint
parametrization from which that datum was obtained. We formulate this
requirement explicitly.
\begin{Hyp}[Matching-data extension]
\label{hyp:matching-dependence}
There exists a neighbourhood $U_\Sigma\subset\Sigma$ of the reference
matching point $q_0$ and a smooth map
$\widetilde{\mathcal A}:U_\Sigma\longrightarrow\mathcal X$ whose restriction
to every matching datum arising from a metric in the chosen Einstein family
agrees with its core analytical package. When the propagated one-sided germ
families are also assigned probe data, we additionally require their package
maps to agree with $\widetilde{\mathcal A}\circ\Phi_\pm$.
\end{Hyp}
\smallskip
This is an explicit extension hypothesis, not a consequence of the canonical
probe existence theorem. A Cauchy datum $q\in\Sigma$ need not by itself define
a global compact Einstein metric or a global harmonic probe, so the existence
of such an extension must be verified in any model in which response is to be
studied on all of $T_{q_0}\Sigma$. For response theory restricted to the
chosen family $S$, only the package map on $S$ is required.

\smallskip
One cannot infer such an extension merely from a smooth package map on
$j(S)$. For example, if $j(S)$ is a proper lower-dimensional subset of
$\Sigma$, the values of the package on $j(S)$ contain no information about
transverse derivatives in $T\Sigma$. Accordingly, every statement involving
$D\widetilde{\mathcal A}_{q_0}$ on arbitrary matching-data directions remains
conditional on Hypothesis~\ref{hyp:matching-dependence}.

\begin{Prop}[Factorization through matching data]
\label{prop:package-factorization}
Assume Hypothesis~\ref{hyp:matching-dependence}. Then
$\mathcal A_\pm=\widetilde{\mathcal A}\circ\Phi_\pm$ and hence
$D\mathcal A_\pm=D\widetilde{\mathcal A}\circ D\Phi_\pm$. At a matching pair
$(u_0,v_0)$ with common matching point $q_0$ and common package value $a_0$,
define the linear package mismatch
$$
\delta\mathcal A_{(u_0,v_0)}(\xi,\zeta)
:=D\mathcal A_-(u_0)[\xi]-D\mathcal A_+(v_0)[\zeta]
\in T_{a_0}\mathcal X.
$$
Then
$$
\delta\mathcal A_{(u_0,v_0)}(\xi,\zeta)
=D\widetilde{\mathcal A}_{q_0}
\bigl(D\Phi_-(u_0)[\xi]-D\Phi_+(v_0)[\zeta]\bigr).
$$
\end{Prop}
\begin{proof}
The first identities are the assumed factorization and the chain rule. At a
matching pair both package maps take the same value $a_0$, so their
differentials take values in the common tangent space $T_{a_0}\mathcal X$ and
may be subtracted. Substituting the two chain-rule identities gives the last
formula.
\end{proof}
\begin{Bem}
\label{rem:factorization-not-reconstruction}
The proposition is a factorization statement, not a reconstruction theorem.
Even when the extension exists, $D\widetilde{\mathcal A}_{q_0}$ may have a
nontrivial kernel. Recovery of matching variations therefore requires an
additional injectivity hypothesis on the relevant subspace of
$T_{q_0}\Sigma$. Moreover, none of the detection results formulated purely on
$S$ requires an extension of the package to an open subset of $\Sigma$.
\end{Bem}

\subsection{The Jacobi family}
\label{sec:jacobi-family}
The canonical harmonic probe assignment gives rise to a smoothly varying
family of normalized Jacobi operators. This family is the fundamental linear
analytical object associated with the probe construction. Its inverses yield
the Green operators forming the second component of the core analytical
package, while additional spectral constructions will be introduced only
under the corresponding supplementary hypotheses.

\smallskip
Recall that
$S_{\mathcal U}=\mathbf g^{-1}(\mathcal U),$
where $\mathcal U$ is the ambient metric neighbourhood supplied by
Theorem~\ref{thm:probe-existence}. For every
$p\in S_{\mathcal U},$
let
$f_p=\Pi(g_p)$
denote the corresponding canonical harmonic probe.
Linearizing the nonlinear probe operator with respect to the map variable at
$(g_p,f_p)$ gives the normalized Jacobi operator
$$J_p=D_f\mathcal F(g_p,f_p): \mathscr H_p^{k+2,\alpha} \longrightarrow \mathscr Y_p^{k,\alpha},$$
where
$\mathscr H_p^{k+2,\alpha}=T_{f_p}\mathcal H^{k+2,\alpha}$
is the Banach space of admissible equivariant variation fields satisfying the
linearized endpoint and normalization conditions, and
$\mathscr Y_p^{k,\alpha}=\mathcal Y_{f_p}^{k,\alpha}$
is the corresponding target space.

\smallskip
With the sign convention used throughout this paper, the Jacobi operator has
the geometric form
$$J_pX=\nabla^*\nabla X-\operatorname{tr}_{g_p}R^N\!\left(X,df_p(\,\cdot\,)\right)df_p(\,\cdot\,),$$
where $\nabla$ is the connection on $f_p^*TN$ induced by the
Levi--Civita connections of $(M,g_p)$ and $(N,h)$, and $\nabla^*$ denotes its
formal $L^2$ adjoint.
\begin{Prop}[Smooth Jacobi family]
\label{prop:smooth-jacobi-family}
The normalized Jacobi operators
$J_p: \mathscr H_p^{k+2,\alpha} \longrightarrow \mathscr Y_p^{k,\alpha}$
depend smoothly on
$p\in S_{\mathcal U}.$
Equivalently, after the local trivializations fixed in
Section~\ref{sec:harmonic-probe-bundle}, the family is represented by a smooth
map into the Banach space of bounded linear operators.
\end{Prop}
\begin{proof}
The nonlinear probe operator
$\mathcal F: \mathcal G\times\mathcal H^{k+2,\alpha} \longrightarrow \mathcal Y^{k,\alpha}$
is smooth, and the canonical probe assignment
$
p\longmapsto f_p
$
is smooth by Theorem~\ref{thm:probe-existence}. Hence
$
g\longmapsto (g_p,f_p)
$
is smooth.

\smallskip
The partial differential
$(g,f)\longmapsto D_f\mathcal F(g,f)$
depends smoothly on $(g,f)$. Therefore the composition
$p\longmapsto D_f\mathcal F(g_p,f_p)=J_p$
is smooth.
After identifying the varying fibres of
$\mathscr H^{k+2,\alpha}$ and
$\mathscr Y^{k,\alpha}$ by the local trivializations of
Section~\ref{sec:harmonic-probe-bundle}, this becomes a smooth family of
bounded operators between fixed Banach spaces.
\end{proof}
\smallskip
Thus the canonical harmonic probe assignment determines a smooth
operator-valued family
$p\longmapsto J_p.$
After shrinking $S_{\mathcal U}$ if necessary, the normalized Jacobi
operators remain invertible. Their inverses form the Green family introduced
below.

\subsection{Functional realizations of the Jacobi family}
The Jacobi family admits two complementary functional realizations. The
Hölder realization is adapted to the nonlinear Implicit Function Theorem and
smooth parameter dependence, whereas the self-adjoint $L^2$ realization
provides the natural framework for spectral theory.
\subsubsection{Hölder realization}
For every
$p\in S_{\mathcal U},$
the normalized Jacobi operator has the Hölder realization
$J_p: \mathscr H_p^{k+2,\alpha} \longrightarrow \mathscr Y_p^{k,\alpha}.$
After choosing the smooth local trivializations introduced in
Section~\ref{sec:harmonic-probe-bundle}, the Jacobi family is represented as
$$
p\longmapsto
J_p
\in
\mathcal L
\left(
\mathscr H_{p_0}^{k+2,\alpha},
\mathscr Y_{p_0}^{k,\alpha}
\right).
$$
\begin{Prop}[Smooth Hölder Jacobi family]
\label{prop:smooth-holder-jacobi-family}
The map
$p\longmapsto J_p$
is smooth as a family of bounded linear operators between the fixed Hölder
spaces determined by any of the admissible smooth local trivializations.
\end{Prop}
\begin{proof}
This is precisely Proposition~\ref{prop:smooth-jacobi-family} expressed in
the local trivializations of
Section~\ref{sec:harmonic-probe-bundle}.
\end{proof}

\subsubsection{\texorpdfstring{$L^2$}{L2} realization}
For spectral questions we use the corresponding Hilbert-space realization.
Let
$\mathscr L_p^2=L_{\mathrm{eq}}^2(f_p^*TN),$
with the $L^2$ structure induced by $g_p$ and $h$. The normalized Jacobi
operator is realized as a closed operator
$J_p:\mathcal D_p\longrightarrow\mathscr L_p^2,$
where
$\mathcal D_p \subset H_{\mathrm{eq}}^2(f_p^*TN)$
is the domain determined by the prescribed endpoint and normalization
conditions.
\begin{Prop}[Self-adjoint Fredholm realization]
\label{prop:self-adjoint-jacobi}
Assume that the prescribed endpoint and normalization conditions define a
self-adjoint elliptic realization of $J_p$. Then
$J_p:\mathcal D_p\longrightarrow\mathscr L_p^2$
is self-adjoint, bounded below and has compact resolvent.

\smallskip
Consequently, its spectrum is real and discrete, each eigenspace is
finite-dimensional, and the eigenvalues have no finite accumulation point.
Since the operator is bounded below, they may be ordered so that they tend to
$+\infty$.
\end{Prop}
\begin{proof}
The Jacobi operator is formally self-adjoint and has Laplace-type principal
part. By assumption, the endpoint and normalization conditions determine a
self-adjoint elliptic realization.

\smallskip
The curvature term is a bounded zeroth-order perturbation on the compact
manifold, so the resulting operator is bounded below. Elliptic regularity and
the compact embedding of the graph domain into $L^2$ imply compact
resolvent. The stated spectral properties then follow from the spectral
theorem for lower semibounded self-adjoint operators with compact resolvent.
\end{proof}

\begin{Cor}[Persistence of normalized nondegeneracy]
\label{cor:persistence-nondegeneracy}
After shrinking $S_{\mathcal U}$ if necessary, every normalized Jacobi
operator remains invertible in the Hölder realization:
$$J_p: \mathscr H_p^{k+2,\alpha} \longrightarrow \mathscr Y_p^{k,\alpha}.$$
In particular,
$\ker J_p=\{0\}.$
The Green operators
$G_p:=J_p^{-1}$
therefore exist in the Hölder realization and depend smoothly on $p$.

\smallskip
If, in addition, the self-adjoint realization of
Proposition~\ref{prop:self-adjoint-jacobi} is used, then its $L^2$ kernel is
also trivial.
\end{Cor}
\begin{proof}
Invertibility is an open property in the Banach space of bounded operators
between fixed Banach spaces. Since
$J_{g_0}$
is invertible by
Hypothesis~\ref{hyp:normalized-jacobi-nondegeneracy}, and the Hölder Jacobi
family depends smoothly on $p$ by
Proposition~\ref{prop:smooth-holder-jacobi-family}, every sufficiently nearby
$J_p$ remains invertible.

\smallskip
The inversion map on the open set of bounded invertible operators is smooth.
Hence
$g\longmapsto G_p=J_p^{-1}$
is smooth after local trivialization.

\smallskip
Finally, an element of the kernel of the self-adjoint $L^2$ realization is
smooth by elliptic regularity and satisfies the same endpoint and
normalization conditions. It therefore belongs to the Hölder kernel, which is
trivial. Hence the $L^2$ kernel is also zero.
\end{proof}

The Hölder realization governs the nonlinear theory, smooth dependence and
construction of the Green family. The self-adjoint $L^2$ realization provides
the additional Hilbert-space structure needed for spectral questions. These
two realizations are complementary and should not be conflated.

\paragraph{Second variation.}
Whenever the self-adjoint realization is in force, the normalized Jacobi
operator determines the quadratic form
$$Q_p(X) = \langle J_pX,X\rangle_{L^2(g_p)}.$$
Its negative spectral subspace defines the equivariant Morse index, while
$\operatorname{nul}(J_g) = \dim\ker J_g$
is the equivariant nullity.

\smallskip
These are invariants of the harmonic probe problem. They belong to the
extended spectral theory rather than to the minimal core package required for
the Einstein Detection Principle. Their possible significance for the
Einstein deformation problem is mediated through the response theory
developed in the subsequent sections.

\subsection{Green operators and spectral theory}
The normalized Jacobi family introduced above admits the standard
operator-theoretic constructions associated with elliptic operators. Among
these, the Green operators play the distinguished role in the present paper.
They form part of the core analytical package and enter directly into the
linear response law for the canonical harmonic probes.

\smallskip
The remaining constructions in this subsection belong to the extended
spectral theory. They are not required for the Einstein Detection Principle
and are used only when the additional spectral hypotheses stated below are
satisfied.

\subsubsection{Green operators}
By Corollary~\ref{cor:persistence-nondegeneracy}, after shrinking
$S_{\mathcal U}$ if necessary, every normalized Jacobi operator is invertible
in the Hölder realization.
\begin{Dfn}[Green operator]
For
$p\in S_{\mathcal U},$
the \emph{Green operator} of the normalized Jacobi operator is
$G_p:=J_p^{-1}.$
Thus
$G_p: \mathscr Y_p^{k,\alpha} \longrightarrow \mathscr H_p^{k+2,\alpha}$
is the solution operator for the normalized Jacobi equation
$J_pX=Y.$
\end{Dfn}

After the local trivializations fixed in
Section~\ref{sec:harmonic-probe-bundle}, the family
$p\longmapsto G_p$
depends smoothly on $p$. Indeed, the inversion map on the open set of
invertible bounded operators is smooth, and
$G_p=J_p^{-1}.$

\smallskip
The Green operator is the distinguished linear object governing the
first-order response of the harmonic probes. In particular, the response
identity obtained earlier takes the form
$$D(\Pi_E)_p=-G_pK_p.$$
Thus infinitesimal variation of the probe is obtained by applying the Green
operator to the forcing produced by the infinitesimal variation of the
underlying metric.

\subsubsection{Resolvents and spectral projectors}
For the self-adjoint $L^2$ realization, and for
$z\in\mathbb C\setminus\operatorname{spec}(J_p),$
define the resolvent
$$R_p(z)=(J_p-z)^{-1}.$$
Suppose that $\Gamma\subset\mathbb C$ is a closed contour lying in the
resolvent set of $J_p$ for every $p$ in a sufficiently small neighbourhood of $p_0$
and remaining uniformly separated from the spectrum. Under the fixed-domain
and local-trivialization conventions of Section~2, the resolvent depends
smoothly on the parameter:
$(p,z) \longmapsto R_p(z).$

\smallskip
The corresponding Riesz projector is
$$P_{\Gamma,p}=\frac{1}{2\pi i}\int_\Gamma(J_p-z)^{-1}\,dz.$$
It projects onto the finite-dimensional spectral subspace associated with the
part of the spectrum enclosed by $\Gamma$. Consequently,
$p\longmapsto P_{\Gamma,p}$
depends smoothly on $p$ as long as the enclosed spectral cluster remains
separated from the remainder of the spectrum.

\smallskip
No smooth labeling of individual eigenvalues through eigenvalue crossings is
asserted. Smooth dependence is instead naturally formulated for isolated
spectral clusters, or for individual simple eigenvalues away from crossings.

\subsubsection{Heat operators}
Assume the self-adjoint realization of
Proposition~\ref{prop:self-adjoint-jacobi}. Since $J_p$ is bounded below and
has compact resolvent, the spectral theorem defines the heat operator
$$e^{-tJ_p}, \qquad t>0.$$
For every fixed
$t>0,$
the operator $e^{-tJ_p}$ is smoothing and trace class.

\smallskip
Under the fixed-domain and smooth-family assumptions stated in Section~2,
standard parameter-dependent elliptic functional calculus (see, for example,
Kato~\cite{Kato} and Taylor~\cite{TaylorPDE2}) implies that, after local
trivialization,
$p\longmapsto e^{-tJ_p}$
depends smoothly on $p$ for every fixed $t>0$. In particular, its trace
$\operatorname{Tr}(e^{-tJ_p})$
defines a smooth spectral quantity as long as the hypotheses governing the
self-adjoint family remain valid.

\smallskip
The resolvent, the Riesz projectors and the heat operators are canonical
spectral constructions associated with the Jacobi family. Unlike the Green
operator, however, they are not required for the core response theory.

\subsection{Additional spectral invariants of the probe family}
The self-adjoint Jacobi family carries further invariants from elliptic
spectral theory. These quantities provide supplementary information about the
canonical harmonic probes but play no role in the proof of the Einstein
Detection Principle.

\smallskip
Accordingly, the constructions in this subsection belong to the extended
spectral data introduced earlier and are invoked only under their stated
additional hypotheses.

\subsubsection{Morse index, nullity and spectral flow}
For the self-adjoint realization, let
$E_p^-$
denote the direct sum of the negative eigenspaces of $J_p$. Since $J_p$ is
bounded below and has compact resolvent, this space is finite-dimensional.
\smallskip
The equivariant Morse index is
$\operatorname{ind}(J_p)=\dim E_p^-,$
while the equivariant nullity is
$\operatorname{nul}(J_p)=\dim\ker J_p.$
By Corollary~\ref{cor:persistence-nondegeneracy}, after shrinking
$S_{\mathcal U}$ if necessary,
$\operatorname{nul}(J_p)=0.$
Consequently, the Morse index is locally constant throughout this
nondegenerate neighbourhood.

\smallskip
More generally, along a continuous path of self-adjoint Fredholm
realizations for which eigenvalues are allowed to cross zero, the net signed
number of such crossings is measured by the spectral flow. With a fixed
spectral-flow sign convention, the change of Morse index is determined by
this spectral flow, with the overall sign depending on that convention.
Thus no sign convention for spectral flow is used elsewhere in the paper.

\subsubsection{Zeta functions and regularized determinants}
Further global spectral invariants may be defined under the standard
hypotheses of elliptic spectral theory.

\smallskip
Suppose first that the self-adjoint Jacobi operator is strictly positive.
Writing its eigenvalues with multiplicity as
$0<\lambda_0(p)\le\lambda_1(p)\le\cdots$, its spectral zeta function is
defined initially for $\operatorname{Re}(s)$ sufficiently large by
$$\zeta_{J_p}(s)=\sum_{k=0}^{\infty}\lambda_k(p)^{-s}.$$
For a positive elliptic differential operator of positive order on a closed
manifold, the standard heat-kernel/complex-power theory gives a meromorphic
continuation to the complex plane and regularity at $s=0$; see, for example,
Taylor~\cite{TaylorPDE2}. The zeta-regularized determinant is then
$$\det\nolimits_\zeta(J_p)=\exp\!\left(-\zeta_{J_p}'(0)\right).$$

\smallskip
For an invertible but not necessarily positive Jacobi operator, one may more
generally choose a common Agmon angle and define complex powers using the
corresponding spectral cut. The associated cut-dependent zeta function and
regularized determinant are then defined by the standard elliptic functional
calculus. Such a construction is asserted only on parameter regions for which
a common spectral cut remains valid.

\smallskip
Thus zeta functions and regularized determinants are additional spectral
invariants of the canonical harmonic probe family. They belong to the
extended spectral data rather than to the core analytical package, and no
result concerning the Einstein Detection Principle depends on their
existence.

\smallskip
The role of the core package in the Einstein deformation problem is instead
mediated by the response theory developed in the following sections. The
extended spectral data provide optional supplementary observables whenever
their additional analytical hypotheses are available.

\subsection{The extended analytical package}
The preceding subsections associate with the canonical harmonic probe family a
collection of operator-theoretic and spectral objects. The normalized Jacobi
and Green operators form the core analytical package required for the Einstein
Detection Principle. Under additional spectral hypotheses, this core package
may be enlarged by further constructions arising from elliptic spectral
theory.

\smallskip
We collect these optional constructions under the name
\emph{extended analytical package}.

\begin{Thm}[Extended analytical package]
\label{thm:analytical-package}
Let
$p\in S_{\mathcal U},$
and assume that the normalized Jacobi family remains invertible on
$S_{\mathcal U}$.

\smallskip
Suppose, whenever a spectral construction is invoked, that its hypotheses hold
uniformly on the parameter neighbourhood under consideration: in particular,
use a common self-adjoint graph domain after trivialization, a common
resolvent region for resolvents, a fixed contour separated from the spectrum
for Riesz projectors, and a common spectral cut for cut-dependent zeta
constructions.

\smallskip
Then the canonical harmonic probe assignment determines an extended
analytical package
$\mathcal A_{\mathrm{ext}}(p),$
whose available components include
\begin{enumerate}
\item
the normalized Jacobi operator
$J_p;$
\item
the Green operator
$G_p=J_p^{-1};$
\item
resolvent families
$(J_p-z)^{-1}$
on common regions of the resolvent set;
\item
Riesz projectors associated with spectrally separated clusters;
\item
the heat operators
$e^{-tJ_p}, \qquad t>0;$
\item
the quadratic form
$Q_p(X) = \langle J_pX,X\rangle_{L^2(g_p)};$
\item
the Morse index and nullity;
\item
spectral zeta functions and zeta-regularized determinants whenever the
corresponding positivity or common spectral-cut hypotheses hold.
\end{enumerate}
After the local trivializations fixed earlier, the Jacobi and Green families,
the resolvents on common resolvent regions, the Riesz projectors associated
with isolated spectral clusters and the heat operators depend smoothly on the
Einstein parameter under their respective hypotheses.

\smallskip
The Morse index is locally constant throughout a nondegenerate region, while
the nullity vanishes on $S_{\mathcal U}$.

\smallskip
Zeta functions and regularized determinants are defined and vary according to
the standard parameter-dependent elliptic theory only on regions where their
defining spectral hypotheses remain valid.
\end{Thm}
\begin{proof}
Smoothness of $J_g$ is Proposition~\ref{prop:smooth-holder-jacobi-family}.
Smoothness of $G_g=J_g^{-1}$ follows from smooth inversion on the open set of
bounded invertible operators. The resolvent identity and contour integration
give the corresponding smoothness statements for resolvents and separated
Riesz projectors. The heat-family assertion follows from the stated smooth
functional calculus for the self-adjoint elliptic realizations. The remaining
spectral assertions follow from compact resolvent. Zeta functions and
regularized determinants require, in addition, the common spectral-cut
hypothesis stated in the theorem.
\end{proof}

\begin{Bem}[Core and extended analytical packages]
\label{rem:core-vs-extended}
The distinction between the core and extended analytical packages is essential.

\smallskip
The \emph{core analytical package} is
$\mathcal A_{\mathrm{core}}(p)=(J_p,G_p).$
It is available under normalized Jacobi nondegeneracy and is sufficient for
the response theory and the Einstein Detection Principle developed below.

\smallskip
The \emph{extended analytical package}
$\mathcal A_{\mathrm{ext}}(p)$
contains the core package together with whichever additional spectral objects
are defined under the supplementary hypotheses of the preceding subsections.

\smallskip
These may include resolvents, spectral projectors, heat operators, Morse
indices, zeta functions and regularized determinants.

\smallskip
No theorem establishing the Einstein Detection Principle depends on the
existence of these additional spectral objects. They provide only a larger
class of possible analytical observables when the required hypotheses are
available.
\end{Bem}

\subsection{Limits of the auxiliary harmonic theory}
\label{sec:limits-probe-theory}
The preceding sections construct a smooth family of canonical harmonic probes
together with their core analytical package. These constructions provide a
canonical and analytically rich response mechanism, but they do not by
themselves solve the Einstein detection problem.

\smallskip
The essential distinction is between \emph{solvability} and
\emph{observability}.
\begin{center}
\emph{Unique solvability of the forced Jacobi equation does not imply
detectability of infinitesimal Einstein deformations.}
\end{center}

\smallskip
Let
$p\in S_{\mathcal U},$
and let
$X\in T_pS_{\mathcal U}$
be a parameter direction in the chosen smooth parameterized Einstein family, inducing the infinitesimal metric variation $D\mathbf g_p[X]$.

\smallskip
Normalized Jacobi nondegeneracy means that
$J_p: \mathscr H_p^{k+2,\alpha} \longrightarrow \mathscr Y_p^{k,\alpha}$
is an isomorphism.

\smallskip
Differentiating the harmonic probe equation gives the forced Jacobi equation
$$J_p\dot f=-K_p[X],$$
where
$$K_p=D_g\mathcal F(g_p,f_p)\circ D\mathbf g_p$$
is the forcing operator.

\smallskip
Hence
$$\dot f=-J_p^{-1}K_p[X]=-G_pK_p[X].$$
Equivalently,
$$D(\Pi_E)_p=-G_pK_p.$$
Thus normalized Jacobi nondegeneracy guarantees that every infinitesimal
variation of the Einstein family produces a uniquely determined infinitesimal
variation of the associated harmonic probe.

\smallskip
This is a solvability statement. It is not an injectivity statement.

\smallskip
Indeed, it may happen that
$D\mathbf g_p[X]\neq0$
while
$D(\Pi_E)_p[X]=0.$
More generally, even when the probe itself varies nontrivially, all selected
finite-dimensional observations of that variation may vanish. In either case,
the chosen probe package fails to distinguish the corresponding infinitesimal
Einstein deformation.

\smallskip
A simple warning is provided by constant harmonic probes. If
$df_p=0,$
then the metric forcing term
$K_p[X]$
may vanish identically or be highly degenerate. In such a situation the probe
may be perfectly nondegenerate as a harmonic map while carrying little or no
first-order information about the underlying Einstein deformation.

\smallskip
This shows that normalized Jacobi nondegeneracy and invertibility of the Green
operator are only the first analytical ingredients required for detection.
They guarantee that the probe response is well defined, but not that it is
informative.

\smallskip
This motivates the construction of admissible observations of the
core analytical package and to determine whether their first variations
separate the infinitesimal directions in
$T_pS_{\mathcal U}.$
This leads to the notions of invisible deformation, observability and response
rank introduced in the following sections.

\smallskip
Accordingly, the logical structure of the theory is
\[
\begin{aligned}
\text{Jacobi nondegeneracy}
&\Longrightarrow \text{unique probe response}\\
&\Longrightarrow \text{analytical observations}\\
&\Longrightarrow \text{observability},
\end{aligned}
\]
where only the final implication requires additional injectivity hypotheses.

\smallskip
The Einstein Detection Principle is precisely the framework that identifies
and analyzes these additional observability conditions.

\begin{Bem}[Standing hypotheses and dependency separation]
\label{rem:standing-hypotheses-dependency}
The response and detection results below use several logically distinct
inputs, and none is to be imported implicitly into another.

\smallskip
The canonical probe branch uses the Banach-space setup and normalized Jacobi
nondegeneracy. Smooth Jacobi and Green families depend on that branch and on
the fixed local trivializations. Descent from representative metrics to a
parameterized Einstein family additionally requires
Hypothesis~\ref{hyp:package-naturality} together with the stated smooth
compatibility of local representatives. The universal response differential
and all detection results formulated purely on $S$ depend on this descended
core package, but they do not require a matching-data extension.

\smallskip
By contrast, comparison with arbitrary tangent directions of the Einstein
matching hypersurface requires the separate
Hypothesis~\ref{hyp:matching-dependence}. Results using spectral projectors,
heat operators, zeta quantities or other extended spectral data require their
own supplementary spectral hypotheses and are not used in proving the core
Einstein Detection Principle.

\smallskip
Finally, observability, finite scalar detection and local reconstruction each
require additional injectivity or richness hypotheses stated at the point of
use. None of these conclusions follows from probe existence, Jacobi
nondegeneracy, descent, or matching-data locality alone.
\end{Bem}

\section{Canonical Response and Observability Theory}
\label{sec:detection-introduction}
The previous section constructed the core analytical package associated with
the normalized harmonic probes and described the additional spectral objects
that may be included in the extended analytical package under supplementary
hypotheses. It also showed that normalized Jacobi nondegeneracy guarantees
unique solvability of the metric-forced Jacobi equation, but does not by
itself provide a mechanism for distinguishing different infinitesimal
Einstein deformations.

\smallskip
The purpose of the present section is to develop such a mechanism. We
introduce a canonical response theory in which infinitesimal variations of a
chosen smooth parameterized Einstein family are mapped to variations of the
analytical structures generated by their associated harmonic probes. Suitable
observations of these variations lead to an intrinsic notion of
observability, describing when the resulting analytical responses distinguish
the underlying infinitesimal Einstein directions.

\smallskip
The response theory therefore provides the analytical bridge between the
canonical harmonic probe construction and the finite-dimensional Einstein
deformation problem. Together with the Einstein matching formulation developed
earlier, it supplies the foundation for the detection, distinction and local
reconstruction results established below.

\subsection{Analytical and geometric setup}
We now fix the geometric and analytical setting in which the response theory
will be formulated.

\smallskip
Let
$S$
be a smooth parameterized Einstein family in the sense of
Section~\ref{subsec:parameterized-einstein-families}, and let
$p_0\in S$
correspond to the reference Einstein metric
$g_{p_0}=g_0.$
After shrinking $S$ around $p_0$ if necessary, we assume that
$g_p\in\mathcal U$
for every $p\in S$, where $\mathcal U$ is the neighbourhood on which the
canonical harmonic probe assignment of
Theorem~\ref{thm:probe-existence} is defined.

\smallskip
We write
$S_{\mathcal U} = \{p\in S:g_p\in\mathcal U\}.$
After replacing $S$ by a sufficiently small neighbourhood of $p_0$, we may
and shall regard $S_{\mathcal U}$ as the parameter manifold on which all
subsequent response constructions are defined.

\smallskip
No assumption is made that $S_{\mathcal U}$ contains all nearby normalized
Einstein metrics or that the ambient normalized Einstein moduli space is a
smooth manifold. All infinitesimal statements below concern the tangent spaces
$T_pS_{\mathcal U}$
of the chosen smooth parameterized Einstein family.

\subsubsection*{Analytical framework}
For every
$p\in S_{\mathcal U},$
let
$f_p:=f_{g_p}=\Pi(g_p)$
denote the corresponding canonical harmonic probe.

\smallskip
The core analytical package is
$\mathcal A_{\mathrm{core}}(p) = (J_p,G_p),$
where
$J_p:=J_{g_p}$
is the normalized Jacobi operator and
$G_p:=J_p^{-1}$
is its Green operator.

\smallskip
After the local trivializations fixed in the preceding sections, the core
package may be regarded as a smooth map
$\mathcal A_{\mathrm{core}}: S_{\mathcal U} \longrightarrow \mathscr X_{\mathrm{core}},$
where $\mathscr X_{\mathrm{core}}$ denotes a suitable Banach manifold of
operator data containing the pairs $(J_p,G_p)$.

\smallskip
Whenever the supplementary spectral hypotheses of
Section~\ref{sec:jacobi-family} and
Theorem~\ref{thm:analytical-package} are satisfied, one may similarly form an
extended analytical package
$\mathcal A_{\mathrm{ext}}: S_{\mathcal U} \longrightarrow \mathscr X_{\mathrm{ext}}.$
The extended package may contain resolvents, Riesz projectors, heat operators
and further spectral invariants whenever these objects are defined under the
relevant hypotheses.

\smallskip
The distinction between these two packages will be maintained throughout the
response theory. The Einstein Detection Principle requires only the core
analytical package. Extended spectral data provide additional possible
observables but are not required for the detection theory itself.

\begin{Prop}[Smoothness of the core analytical package]
\label{prop:analytic-package-setup}
After possibly shrinking $S_{\mathcal U}$, the map
$\mathcal A_{\mathrm{core}}: S_{\mathcal U} \longrightarrow \mathscr X_{\mathrm{core}}$
is smooth.

\smallskip
Consequently, if
$\Psi: \mathscr X_{\mathrm{core}} \longrightarrow Y$
is any smooth map into a Banach manifold $Y$, then the associated observable
$\Psi\circ\mathcal A_{\mathrm{core}}: S_{\mathcal U} \longrightarrow Y$
is smooth.

\smallskip
The same conclusion holds for any component of the extended analytical
package on a region where that component is defined and has the smooth
parameter dependence established in the preceding section.
\end{Prop}

\begin{proof}
The parameterization
$p\longmapsto g_p$
is smooth by the definition of a smooth parameterized Einstein family.

\smallskip
The Jacobi and Green operators depend smoothly on the metric by the results of
the preceding section. Hence
$p\longmapsto \mathcal A_{\mathrm{core}}(p)$
is smooth by composition.

\smallskip
If
$\Psi: \mathscr X_{\mathrm{core}} \longrightarrow Y$
is smooth, then
$\Psi\circ\mathcal A_{\mathrm{core}}$
is smooth by the Banach-manifold chain rule.

\smallskip
The final assertion follows in the same way for those components of the
extended package whose smooth parameter dependence is available under their
respective supplementary hypotheses.
\end{proof}

\subsubsection*{Infinitesimal Einstein directions}
The geometric domain of the response theory is the tangent bundle of the
chosen smooth parameterized Einstein family.

\smallskip
For
$p\in S_{\mathcal U},$
a tangent vector
$X\in T_pS_{\mathcal U}$
determines an infinitesimal variation of normalized Einstein metrics
$\dot g_X = Dg_p[X].$
Because every metric $g_p$ in the family is Einstein with the fixed
normalization, $\dot g_X$ satisfies the linearized normalized Einstein
equation along the chosen family.

\smallskip
No converse is asserted. In particular, the response theory does not require
that every formal infinitesimal Einstein deformation of $g_p$ be tangent to
$S_{\mathcal U}$.
\smallskip
Differentiating the canonical probe assignment in the direction $X$ gives
$D\Pi_{g_p}[\dot g_X] = - G_pK_p[\dot g_X],$
where
$K_p = D_g\mathcal F(g_p,f_p)$
is the metric-forcing operator.

\smallskip
Thus every tangent vector
$X\in T_pS_{\mathcal U}$
determines a canonical infinitesimal probe response
$$\dot f_X = - G_pK_p[\dot g_X].$$
This is the basic infinitesimal response law on the chosen Einstein family.

\smallskip
The analytical package may then be differentiated in the same direction. In
particular,
$$
D\mathcal A_{\mathrm{core},p}:
T_pS_{\mathcal U}
\longrightarrow
T_{\mathcal A_{\mathrm{core}}(p)}
\mathscr X_{\mathrm{core}}
$$
records the first-order variation of the Jacobi and Green data generated by
the infinitesimal Einstein direction $X$.

\smallskip
If
$\Psi: \mathscr X_{\mathrm{core}} \longrightarrow Y$
is a smooth analytical observation, then
$$
D(\Psi\circ\mathcal A_{\mathrm{core}})_p
=
D\Psi_{\mathcal A_{\mathrm{core}}(p)}
\circ
D\mathcal A_{\mathrm{core},p}.
$$
This identity is the elementary chain-rule mechanism underlying the
observation theory developed below.

\subsubsection*{Relation with the Einstein matching problem}
The response theory on $S_{\mathcal U}$ is compatible with the Einstein
matching formulation of
Section~\ref{chap:intrinsic-moduli}.

\smallskip
For every metric in the chosen smooth parameterized Einstein family, the
Einstein evolution determines its canonical matching point on the hypersurface
$\Sigma=\{H=0\}.$
We therefore have the matching-point map
$j: S_{\mathcal U} \longrightarrow \Sigma.$
Under the matching-data extension hypothesis
Hypothesis~\ref{hyp:matching-dependence}, the core analytical package is the
restriction of a smooth package defined on an open neighbourhood
$U_\Sigma\subset\Sigma$ of the reference matching point. Consequently,
$$\mathcal A_{\mathrm{core}} =
\widetilde{\mathcal A}_{\mathrm{core}}\circ j$$
on the chosen family, and differentiation gives
$$D\mathcal A_{\mathrm{core},p}
=
D\widetilde{\mathcal A}_{\mathrm{core},j(p)}\circ Dj_p.$$
Thus the analytical response factors through the infinitesimal variation of
the Einstein matching datum.

\smallskip
This factorization is the precise bridge between the geometric and analytical
parts of the theory. The Einstein matching construction identifies the
finite-dimensional geometric data to be detected, while the canonical
harmonic probes generate analytical structures whose first variations may
distinguish those data.

\smallskip
The remaining question is therefore one of observability:
\begin{center}
\emph{When do the first variations of suitable analytical observations
separate the tangent directions in $T_pS_{\mathcal U}$?}
\end{center}
The following subsections formulate this question intrinsically and develop
the corresponding response and observability theory.

\subsection{Local Einstein deformation theory}
The response theory developed in this paper is formulated on a chosen smooth
parameterized Einstein family rather than on the full local Einstein moduli
space. Nevertheless, the classical gauge-fixed deformation theory provides
the natural local framework from which such families may arise. We therefore
recall the aspects of this theory that will be used below.

\smallskip
The purpose of the present subsection is not to develop the deformation theory
of Einstein metrics in full generality. Rather, we record the relation between
the gauge-fixed Einstein equation, infinitesimal Einstein deformations and the
finite-dimensional Kuranishi model of the local Einstein problem.

\subsubsection{Gauge fixing and infinitesimal Einstein deformations}
Fix the reference normalized Einstein metric
$g_0, \qquad \Ric(g_0)=\lambda g_0.$
The Einstein equation is invariant under diffeomorphisms. A local slice for
the diffeomorphism action therefore provides the standard framework for
separating genuine geometric deformations from infinitesimal gauge
directions.

\begin{Thm}[Local slice theorem {\cite{Ebin,Besse}}]
In the chosen Banach completion, there exists a local slice through $g_0$ for
the action of the full diffeomorphism group such that every sufficiently nearby
metric is diffeomorphic to a metric in the slice, with the usual residual
ambiguity coming from the stabilizer of $g_0$.
\end{Thm}

\smallskip
The classical theorem is a statement for the full diffeomorphism action. In
the $G$-invariant category used in this paper we therefore make, separately,
the standing assumption that the local gauge fixing is chosen compatibly with
the $G$-action and with the normalization, so that the selected smooth
parameterized Einstein family is represented in such a compatible local
slice. No equivariant slice theorem stronger than the cited Ebin theorem is
silently invoked.
\begin{proof}
This is the classical Ebin slice theorem for the action of the
diffeomorphism group on the Banach manifold of Riemannian metrics.
\end{proof}
\smallskip
Consequently, the local Einstein deformation problem may be studied in a
gauge-fixed slice.

\smallskip
With the sign conventions adopted in this paper, the Einstein deformation
operator on transverse-traceless symmetric $2$-tensors is
$L_{g_0}=\Delta_L-2\lambda$. Up to the conventional nonzero factor appearing
in the linearization of $\Ric-\lambda g$, this is the gauge-fixed linearized
Einstein operator, so its kernel is the infinitesimal deformation space used
below.
This motivates the following standard definition.

\begin{Dfn}[Essential infinitesimal Einstein deformations]
The space of essential infinitesimal Einstein deformations of $g_0$ is
$$
\mathcal H_{g_0}
=
\left\{
h\in\Gamma(S^2T^*M):
(\Delta_L-2\lambda)h=0,\;
\delta_{g_0}h=0,\;
\operatorname{tr}_{g_0}h=0
\right\}.
$$
\end{Dfn}
\smallskip
The divergence-free condition removes infinitesimal diffeomorphism
directions, while the trace-free condition is compatible with the fixed
Einstein normalization and removes the infinitesimal homothetic direction.
\begin{Prop}
\label{prop:finite-dimensional-einstein-deformations}
The vector space
$\mathcal H_{g_0}$
is finite-dimensional.
\end{Prop}
\begin{proof}
The Lichnerowicz operator is elliptic on the compact manifold $M$.
Consequently, the kernel of
$\Delta_L-2\lambda$
is finite-dimensional. The space $\mathcal H_{g_0}$ is a linear subspace of
this finite-dimensional kernel and is therefore finite-dimensional.
\end{proof}

\subsubsection{Infinitesimal rigidity}
\begin{Dfn}[Infinitesimal rigidity]
The normalized Einstein metric $g_0$ is called
\emph{infinitesimally rigid} if
$\mathcal H_{g_0} = \{0\}.$
\end{Dfn}
\smallskip
Under the standard gauge-fixed local deformation theory, infinitesimal
rigidity implies local rigidity.

\begin{Thm}[Local rigidity]
Assume that the gauge-fixed normalized Einstein equation is posed on the
local slice described above, that its linearization on the normalized slice is
the self-adjoint Fredholm index-zero realization described below, and that
$\mathcal H_{g_0}=\{0\}$.
Then $g_0$ is locally rigid among normalized Einstein metrics: every
sufficiently nearby normalized Einstein metric is diffeomorphic to $g_0$.
\end{Thm}
\begin{proof}
On the normalized gauge-fixed slice, the kernel of the linearized Einstein
operator is precisely the essential infinitesimal deformation space
$\mathcal H_{g_0}$. In the standard compact gauge-fixed theory the operator on the normalized
slice is elliptic and formally self-adjoint, hence Fredholm of index zero.
Under the hypothesis $\mathcal H_{g_0}=\{0\}$ its kernel is trivial, and
therefore its cokernel is trivial as well. Thus the linearized operator is an
isomorphism between the chosen Banach spaces. The Banach-space Implicit
Function Theorem then shows that $g_0$ is an isolated solution of the
gauge-fixed Einstein equation. Passing back through
the local slice gives the asserted rigidity modulo diffeomorphism.
\end{proof}

\subsubsection{The linearized deformation space}
The space
$\mathcal H_{g_0}$
is the natural first-order deformation space of the normalized Einstein
problem at $g_0$. It should not, however, be confused with the tangent space
of a smooth local Einstein moduli manifold unless such a manifold has
actually been shown to exist.

\smallskip
More precisely, the classical Kuranishi construction gives a
finite-dimensional local model in which
$\mathcal H_{g_0}$
serves as the space of infinitesimal deformation parameters. The local
Einstein solution set is then represented by the zero set of a
finite-dimensional obstruction map.

\smallskip
Thus every genuine smooth one-parameter family of normalized Einstein metrics
$$s\longmapsto g_s, \qquad g_0=g_{s=0},$$
determines an element
$$\left. \frac{d}{ds} \right|_{s=0} g_s \in \mathcal H_{g_0}$$
after gauge fixing. The converse need not hold: an element of
$\mathcal H_{g_0}$ may be obstructed at higher order and therefore need not
integrate to a family of Einstein metrics.

\subsubsection{Smooth parameterized families inside the local deformation model}
The distinction between infinitesimal and integrable deformations is
important for the response theory developed below.

\smallskip
The local Kuranishi zero set may be singular and need not itself be a smooth
manifold. Accordingly, the present paper does not formulate its
differential-geometric response theory on the entire local Einstein moduli
set.

\smallskip
Instead, whenever differential-geometric constructions are required, we work
with a chosen smooth parameterized Einstein family
$S$
in the sense of
Section~\ref{subsec:parameterized-einstein-families}. Locally, such a family
may be viewed as a smooth family of solutions contained in the gauge-fixed
Einstein solution set.

\smallskip
For
$$p_0\in S, \qquad g_{p_0}=g_0,$$
the differential of the parameterization gives a linear map
$$D\mathbf g_{p_0}: T_{p_0}S \longrightarrow \mathcal H_{g_0}.$$
Thus the actual infinitesimal metric variations realized by the family form
the image $D\mathbf g_{p_0}(T_{p_0}S)\subset\mathcal H_{g_0}$. If the
parameterization is immersive, $T_{p_0}S$ may be identified with this image;
otherwise parameter directions in $\ker D\mathbf g_{p_0}$ are redundant. No assertion is made that
$D\mathbf g_{p_0}(T_{p_0}S) = \mathcal H_{g_0}.$
Equality holds only when the chosen family realizes all infinitesimal
deformations under consideration.

\smallskip
This is the geometric domain on which the response operators in the
subsequent sections are defined. In particular, the theory requires no smooth
structure on the full local Einstein moduli space.

\subsubsection{Beyond first order}
The finite-dimensional space
$\mathcal H_{g_0}$
captures the linearized normalized Einstein equation, but it does not in
general determine the nonlinear local solution set. Higher-order obstruction
terms may prevent infinitesimal deformations from integrating to genuine
Einstein families.

\smallskip
The standard Kuranishi reduction makes these nonlinear obstructions explicit
through a finite-dimensional obstruction map. We recall this construction in
the next subsection, while maintaining the distinction between the possibly
singular Kuranishi zero set and the chosen smooth parameterized Einstein
families on which the response theory is formulated.

\subsection{Kuranishi reduction}
\label{sec:kuranishi-reduction}
The preceding subsection identified the finite-dimensional space of
infinitesimal Einstein deformations
$\mathcal H_{g_0}.$
In general, however, an infinitesimal Einstein deformation need not integrate
to a genuine nearby Einstein metric. The nonlinear relation between the
linearized deformation space and the local Einstein solution set is described
by the classical Kuranishi reduction.

\smallskip
After fixing the Bianchi gauge, the normalized Einstein equation becomes a
nonlinear elliptic equation on a Banach manifold. The tangent space at the
reference metric splits into the finite-dimensional kernel of the linearized
operator and a complementary subspace. Solving the projected nonlinear
equation on the complement by the Banach-space Implicit Function Theorem
reduces the Einstein equation to a finite-dimensional obstruction equation on
$\mathcal H_{g_0}.$
The resulting local model is summarized by the following classical theorem.

\begin{Thm}[Finite-dimensional Kuranishi reduction]
\label{thm:kuranishi}
There exist a neighbourhood
$0\in U\subset\mathcal H_{g_0},$
a finite-dimensional obstruction space
$\mathcal O,$
and a smooth map
$\kappa: U\longrightarrow\mathcal O$
with $\kappa(0)=0$ and $D\kappa_0=0$, such that the local gauge-fixed
normalized Einstein solution set is represented by the zero set
$\kappa^{-1}(0).$
Modulo the local diffeomorphism action described by the slice theorem, this
zero set provides a finite-dimensional local model for the normalized
Einstein moduli problem near $g_0$.
\end{Thm}
\begin{proof}
In the standard compact Einstein deformation setting, this is the usual
finite-dimensional reduction obtained from a local slice and elliptic
Lyapunov--Schmidt reduction; see Koiso~\cite{Koiso} and
Besse~\cite{Besse}. The theorem is used here only as a local model for the
gauge-fixed solution germ; the paper does not infer smoothness of the full
moduli space from the Kuranishi construction.
\end{proof}
\smallskip
If the obstruction map vanishes identically near the origin, then the local
gauge-fixed Einstein solution set is smooth and is locally modeled on
$\mathcal H_{g_0}.$
In that case, after quotienting by the residual stabilizer when necessary,
one obtains the usual smooth local moduli picture.

\smallskip
In general, however, the zero set
$\kappa^{-1}(0)$
may be singular, may have several local branches, and need not carry a
canonical smooth-manifold structure. For this reason, the response theory
developed in the present paper is not formulated on the entire Kuranishi zero
set.

\smallskip
Instead, whenever differential-geometric constructions are required, we fix a
smooth parameterized Einstein family
$S$
through the reference metric, in the sense of
Section~\ref{subsec:parameterized-einstein-families}, whose image is contained
in the local gauge-fixed Einstein solution set.

\smallskip
Thus, for some
$$p_0\in S, \qquad g_{p_0}=g_0,$$
the parameterization
$p\longmapsto g_p$
takes values in the Kuranishi zero set after the chosen gauge fixing.

\smallskip
The differential of this parameterization gives a linear map
$D\mathbf g_{p_0}: T_{p_0}S \longrightarrow \mathcal H_{g_0}.$
Its image consists of those infinitesimal Einstein deformations realized by
the chosen smooth family. No assertion is made that every element of
$\mathcal H_{g_0}$
is tangent to $S$, nor that the Kuranishi zero set admits a preferred smooth
branch through $g_0$.

\smallskip
All differential-geometric constructions in the response theory are carried
out on this chosen smooth parameterized Einstein family. In particular,
analytical observables arising from the canonical harmonic probe package are
restricted to $S$, and their differentials are taken on the tangent spaces
$T_pS.$
This provides the finite-dimensional smooth geometric domain required for the
response operators, without imposing any smoothness assumption on the full
local Einstein moduli set.

\begin{Bem}[Terminology]
Throughout the remainder of the paper we distinguish carefully between the
following objects.
\begin{enumerate}
\item
The \emph{Kuranishi zero set}
$\kappa^{-1}(0)$
is the finite-dimensional local model for the gauge-fixed normalized Einstein
solution set. It may be singular and may contain several local branches.
\item
A \emph{smooth parameterized Einstein family}
$S$
is a chosen smooth finite-dimensional family of genuine normalized Einstein
metrics whose image lies in the local solution set.
\item
Whenever tangent spaces, differentials, response operators or local
reconstruction maps are used, they are understood on the chosen smooth
parameterized Einstein family $S$.
\end{enumerate}
No tangent-cone response theory and no canonical smooth stratification of the
Kuranishi zero set are assumed in the present paper.

\smallskip
In particular, $T_{p_0}S$ must not be identified with the full essential
infinitesimal Einstein space unless an additional integrability or smoothness
statement has been proved. An obstructed infinitesimal Einstein deformation
provides the standard counterexample: it lies in the kernel of the linearized
Einstein operator but is not tangent to any genuine local Einstein curve.
\end{Bem}

\subsection{Descent of the analytical package}
The core analytical package is first constructed for representative metrics in
the ambient Banach manifold on which the harmonic probe problem is solved.
When the chosen smooth parameterized Einstein family is described through
local gauge-fixed representatives, one must verify that the resulting
analytical objects are compatible with changes of representative and with
changes of local slice.

\smallskip
This is a naturality question. We therefore impose explicitly the equivariance
properties required for descent.

\begin{Hyp}[Naturality of the canonical probe package]
\label{hyp:package-naturality}
Let $g$ and $\widetilde g$ be normalized Einstein metrics related by a
normalization-preserving equivariant diffeomorphism
$\varphi:(M,g)\longrightarrow(M,\widetilde g)$
that preserves the fixed probe data.

\smallskip
Then the canonical harmonic probes are related by the induced natural action,
and the associated normalized Jacobi and Green operators are intertwined by
the corresponding pullback or pushforward identifications.

\smallskip
Equivalently, the core analytical package
$\mathcal A_{\mathrm{core}}(g) = (J_g,G_g)$
is equivariant under all normalization-preserving equivariant
diffeomorphisms preserving the probe data.
\end{Hyp}

\smallskip
Under this hypothesis the core analytical package is independent, up to its
canonical identifications, of the representative metric used in its
construction.

\begin{Thm}[Descent of the core analytical package]
\label{thm:package-descent}
Assume normalized Jacobi nondegeneracy and
Hypothesis~\ref{hyp:package-naturality}.

\smallskip
Let $S$ be the chosen smooth parameterized Einstein family. Suppose that,
whenever local gauge-fixed representatives are used to describe $S$, the
transition maps between such representatives are induced by smoothly
parameter-dependent normalization-preserving equivariant diffeomorphisms
preserving the probe data, and that their induced identifications act smoothly
on the local Banach model $\mathscr X_{\mathrm{core}}$.

\smallskip
Then the core analytical package descends uniquely, up to the canonical
identifications determined by these transition maps, to a smooth map
$\mathcal A_{E,\mathrm{core}}: S \longrightarrow \mathscr X_{\mathrm{core}}.$
In particular, its value is independent of the choice of local gauge-fixed
representative and of the choice of normalized local slice.
\end{Thm}

\begin{proof}
By Hypothesis~\ref{hyp:package-naturality}, every component of the core package
is equivariant under equivariant diffeomorphisms preserving the normalization, and
the induced identifications vary smoothly. Hence local representatives on
overlapping normalized slices are intertwined by the slice transition maps
and agree after the canonical identifications. They therefore patch to a
well-defined smooth map on the quotient stratum $S$. This is precisely the
claimed descent.
\end{proof}

\begin{Prop}[Canonicality of the core analytical package]
\label{prop:canonicality-package}
Under the hypotheses of
Theorem~\ref{thm:package-descent}, the descended core analytical package
$\mathcal A_{E,\mathrm{core}} = (J_g,G_g)$
is determined by the chosen normalized Einstein geometry together with the
fixed probe data, up to the canonical identifications induced by
normalization-preserving equivariant diffeomorphisms.

\smallskip
In particular, it is independent of
\begin{enumerate}
\item
the choice of normalized local slice;
\item
the choice of representative metric within the corresponding
normalization-preserving equivariant diffeomorphism class;
\item
the local trivializations used to represent the Jacobi and Green bundles.
\end{enumerate}
Consequently, every functorial construction depending only on the core
analytical package inherits the same canonicality.
\end{Prop}
\begin{proof}
Independence of the representative metric and of the local slice follows from
Theorem~\ref{thm:package-descent} and the naturality assumption.

\smallskip
Changing a local trivialization alters only the coordinate representation of
the Jacobi and Green operators and does not alter the underlying bundle
operators themselves.

\smallskip
Hence the descended core analytical package is well defined up to the
canonical identifications inherent in the geometric problem. Any construction
defined functorially from this package is therefore compatible with the same
identifications.
\end{proof}
\smallskip
The preceding results concern only the core analytical package. The extended
spectral data require separate hypotheses.
\begin{Cor}[Descent of extended spectral data]
\label{cor:extended-package-descent}
Assume the hypotheses of
Theorem~\ref{thm:package-descent}. Suppose, in addition, that a component of
the extended spectral package is defined on the parameter region under
consideration and is natural under the same
normalization-preserving equivariant diffeomorphisms.

\smallskip
Then that component descends to $S$ and is independent, up to its canonical
identification, of the choice of representative metric and local slice.
\end{Cor}
\begin{proof}
Each such spectral object is constructed functorially from the self-adjoint
Jacobi realization. Naturality of the Jacobi family therefore induces
naturality of the spectral construction whenever its defining hypotheses
remain valid. The same patching argument as in
Theorem~\ref{thm:package-descent} then gives the descended object on $S$.
\end{proof}

\begin{Bem}[Meaning of canonicality]
\label{rem:meaning-canonicality}
From this point onward, the word \emph{canonical} is always understood
relative to
\begin{enumerate}
\item
the fixed Einstein normalization;
\item
the fixed probe data;
\item
the normalization-preserving equivariant identifications specified above.
\end{enumerate}
Thus canonicality does not mean independence of the chosen probe package.
Rather, it means independence of the auxiliary representative, local slice
and trivialization once the geometric normalization and probe data have been
fixed.

\smallskip
In particular, the core response theory developed below is based only on the
descended package
$\mathcal A_{E,\mathrm{core}}.$
Extended spectral observables may be used additionally whenever the
hypotheses of
Corollary~\ref{cor:extended-package-descent}
are satisfied.
\end{Bem}

\subsection{Response maps and operators}
The descended analytical package constructed in the preceding subsections
provides analytical data attached to every point of the chosen smooth
parameterized Einstein family. The purpose of the present subsection is to
extract from these data finite-dimensional observables whose first variations
may detect infinitesimal Einstein deformations.

\smallskip
The construction proceeds at two related levels. First, an analytical
observable may be defined directly from the descended analytical package on
the chosen smooth parameterized Einstein family. Second, under
Hypothesis~\ref{hyp:matching-dependence}, an observable whose underlying probe
package is determined by the Einstein matching datum may be regarded as a
smooth function on the Einstein matching hypersurface. The latter formulation
will be used below to connect analytical response with the differential of the
Einstein matching map.

\subsubsection*{Analytical observables}
Let
$\mathcal A_{E,\mathrm{core}}: S \longrightarrow \mathscr X_{\mathrm{core}}$
denote the descended core analytical package. Let $Y$ be a
finite-dimensional vector space, and let
$\Psi: \mathscr X_{\mathrm{core}} \longrightarrow Y$
be a smooth map defined on a neighbourhood of the image of the reference
metric.
\begin{Dfn}[Observable]
An \emph{analytical observable} is a smooth finite-dimensional quantity
obtained from the descended analytical package. Thus, in the core theory, it
is a smooth map of the form
$\Theta_E = \Psi\circ\mathcal A_{E,\mathrm{core}}: S \longrightarrow Y.$
Whenever the additional hypotheses required for a component of the extended
analytical package are satisfied, $\Psi$ may instead be defined using that
extended spectral data.
\end{Dfn}
\smallskip
Examples of core-package-generated observables include smooth
finite-dimensional evaluations or functionals of the Jacobi and Green
operators, provided the required evaluation/trivialization data are included
in the admissible postprocessing and respect the canonical identifications.
A harmonic-map energy is an observable of the probe assignment, but it is not
claimed to be core-package-generated unless a separate factorization through
$(J_g,G_g)$ has been established.

\smallskip
This distinction prevents a false converse: two probes may have the same
linearized Jacobi/Green data at a point while an additional zeroth-order probe
quantity, such as the value of the energy, varies. The universal
factorization theorem therefore applies only to observables explicitly
factoring through the stated package.

\smallskip
Under additional spectral hypotheses, one may also use spectrally separated
eigenvalue clusters, Riesz projectors, heat quantities, or
zeta-regularized quantities whenever the corresponding construction is
defined and smooth on the parameter region under consideration.

\smallskip
Thus the response theory itself requires only the core analytical package.
Extended spectral data provide additional possible observables, but are not
required for the Einstein Detection Principle.

\subsubsection*{Observables on the Einstein matching hypersurface}
Assume Hypothesis~\ref{hyp:matching-dependence}. Then the relevant probe
package is determined smoothly by the Einstein matching datum. In particular,
for an observable defined from such a package there exists, after possibly
shrinking to a neighbourhood of the reference matching point
$q_0\in\Sigma$, a smooth map
$\Theta: \Sigma \longrightarrow Y$
such that the observable associated with either propagated endpoint germ
factors as
$\Theta\circ\Phi_\pm.$
Equivalently, if
$\widetilde{\mathcal A}: \Sigma \longrightarrow \mathscr X$
denotes the corresponding package on the matching hypersurface, then
$\Theta = \Psi\circ\widetilde{\mathcal A}.$
Its differential
$D\Theta_q: T_q\Sigma \longrightarrow Y$
records the first-order variation of the chosen analytical observation with
respect to infinitesimal variations of the Einstein matching datum.

\smallskip
For the chosen smooth parameterized Einstein family $S$, let
$j: S \longrightarrow \Sigma$
denote the canonical matching-point map introduced previously. Compatibility
of the two constructions gives
$$\Theta_E = \Theta\circ j.$$
Thus the observable on the chosen Einstein family is the restriction, through
the matching-point map, of the corresponding observable on the Einstein
matching hypersurface.

\begin{Prop}[Smoothness of the restricted observable]
\label{prop:smooth-restricted-observable}
The map
$\Theta_E: S \longrightarrow Y$
is smooth.
\end{Prop}
\begin{proof}
The matching-point map
$j: S \longrightarrow \Sigma$
is smooth, and the observable
$\Theta: \Sigma \longrightarrow Y$
is smooth by construction. Hence
$\Theta_E = \Theta\circ j$
is smooth by the chain rule.
\end{proof}

\subsubsection*{The canonical response operator}
We may now linearize the analytical observable along the chosen smooth
Einstein family.
\begin{Dfn}[Canonical response operator]
\label{def:canonical-response-operator}
Let $p_0\in S$ denote the reference parameter, with $g_{p_0}=g_0$. The
\emph{canonical response operator} associated with the analytical observable
$\Theta_E$ is
$$\mathcal R_{g_0} = D_{p_0}\Theta_E: T_{p_0}S \longrightarrow Y.$$
Thus $\mathcal R_{g_0}$ records the first-order variation of the chosen
analytical observable along genuine infinitesimal deformations tangent to the
chosen smooth parameterized Einstein family.
\end{Dfn}
Since
$\Theta_E = \Theta\circ j,$
the chain rule gives the fundamental factorization
$$\mathcal R_{g_0} = D\Theta_{q_0}\circ Dj_{p_0},$$
where
$q_0=j(p_0).$
This identity separates the analytical resolving power of the observation
$D\Theta_{q_0}$ from the infinitesimal variation of the Einstein matching
datum
$Dj_{p_0}$.

\begin{Thm}[Intrinsicity of the response operator]
\label{thm:response-intrinsicity}
Assume the hypotheses of
Theorem~\ref{thm:package-descent}, and let $\Theta_E$ be an analytical
observable obtained functorially from the descended analytical package.
\smallskip
Then
$$\mathcal R_{g_0} = D_{p_0}\Theta_E$$
is independent of the choice of local coordinates on $S$, of the normalized
local slice used to represent the Einstein family, and of the representative
metric within the corresponding normalization-preserving equivariant
diffeomorphism class.

\smallskip
Consequently, the response operator is canonically determined by the chosen
normalized Einstein geometry, the chosen smooth parameterized Einstein
family, and the fixed probe data.
\end{Thm}
\begin{proof}
Coordinate independence follows from the intrinsic definition of the
differential of a smooth map between manifolds.

\smallskip
By Theorem~\ref{thm:package-descent}, different normalized local
representatives of the same Einstein family determine canonically identified
core analytical packages. Since the observable is constructed functorially
from this descended package, the corresponding local representatives of
$\Theta_E$ agree under the transition identifications.

\smallskip
Differentiating these compatible local representatives therefore produces
canonically identified linear maps. Hence the resulting response operator is
independent of the local slice and representative used in its construction.
\end{proof}

\subsubsection*{Naturality}
The preceding intrinsicity statement may equivalently be expressed in terms
of the naturality of the response construction.

\begin{Prop}[Naturality of the response operator]
\label{prop:response-naturality}
Let
$\varphi: (M,g_0) \longrightarrow (M,h_0)$
be a normalization-preserving equivariant isometry preserving the fixed probe
data. Then the response operators associated with the corresponding
identified Einstein families are intertwined by the natural maps induced by
$\varphi$ on the deformation and observation spaces.

\smallskip
In particular, after making these canonical identifications, the response
operator is invariant under normalization-preserving equivariant isometries.
\end{Prop}
\begin{proof}
By Hypothesis~\ref{hyp:package-naturality}, the core analytical package is
natural under the action of $\varphi$. The same is therefore true for every
functorially constructed analytical observable.

\smallskip
Differentiating the resulting commutative diagram at the reference Einstein
metric gives the corresponding intertwining relation for the response
operators.
\end{proof}

\subsubsection*{Observability}
The kernel of the response operator consists of infinitesimal Einstein
deformations tangent to $S$ that are invisible, to first order, to the chosen
analytical observable. Its image consists of the first-order responses that
can be produced by such deformations.

\begin{Dfn}[Infinitesimally complete observable]
\label{def:infinitesimally-complete-observable}
An analytical observable is called \emph{infinitesimally complete} at
$p_0\in S$ if
$\ker\mathcal R_{g_0} = \{0\}.$
Equivalently, distinct nonzero tangent directions in
$T_{p_0}S$ produce nonzero first-order analytical responses.
\end{Dfn}

\smallskip
This notion concerns only genuine infinitesimal deformations tangent to the
chosen smooth parameterized Einstein family. It is therefore weaker than
requiring the corresponding observation to distinguish every tangent
direction of the full Einstein matching hypersurface.

\smallskip
We now formulate this stronger condition.
\begin{Hyp}[Probe observability]
\label{hyp:probe-observability}
Let
$\Theta: \Sigma \longrightarrow Y$
be a finite-dimensional observation map defined near the reference matching
point $q_0\in\Sigma$.

\smallskip
The observation is said to satisfy the \emph{probe observability hypothesis}
at $q_0$ if
$D\Theta_{q_0}: T_{q_0}\Sigma \longrightarrow Y$
is injective.

\smallskip
If, in addition,
$\dim Y=\dim\Sigma,$
then injectivity is equivalent to bijectivity, and the observation is called
\emph{complete} at $q_0$.
\end{Hyp}
\smallskip
Probe observability is a condition on all infinitesimal matching-data
directions, whereas infinitesimal completeness concerns only directions
arising from the chosen smooth Einstein family through
$Dj_{p_0}: T_{p_0}S \longrightarrow T_{q_0}\Sigma.$

\smallskip
This matching-space notion is not an assumption of the universal detection
theorems below. It is used only when one wishes to infer detection on $S$ from
an observation already known to be injective on the larger space
$T_{q_0}\Sigma$.
The factorization
$\mathcal R_{g_0} = D\Theta_{q_0}\circ Dj_{p_0}$
makes their relationship explicit. In particular, probe observability implies
that
$\ker\mathcal R_{g_0} = \ker Dj_{p_0}.$
Hence, if the matching-point map is immersive at $p_0$, then probe
observability implies infinitesimal completeness of the restricted observable
on $S$.

\begin{Bem}
\label{rem:response-role}
The response operator is the linear object through which the canonical
analytical package interacts with the chosen Einstein deformation family.
Its definition requires neither that the entire Kuranishi germ be smooth nor
that every infinitesimal Einstein deformation be integrable. It is defined
only on
$T_{p_0}S,$
the tangent space of the chosen smooth parameterized Einstein family.

\smallskip
The following subsection derives the more explicit response identity obtained
by differentiating the harmonic probe equation,
$D(\Pi_E)_p=-J_p^{-1}K_p.$
That identity explains how the Green operator and the metric-forced Jacobi
equation generate concrete analytical observations. Combined with the
matching-data factorization developed above, it provides the mechanism
underlying the First-Order Einstein Detection Principle.
\end{Bem}

\subsection{Universal factorization of response operators}
\label{sec:universal-response}
The previous subsection associated a canonical response operator with every
smooth finite-dimensional analytical observable. We now identify the common
linear mechanism through which all observables generated by the core
analytical package respond to infinitesimal Einstein deformations.

\smallskip
The essential observation is simple but structurally important. Once the
probe data have been fixed, the core analytical package
$\mathcal A_{E,\mathrm{core}}$
is fixed. Different finite-dimensional observables are obtained by applying
different smooth postprocessing maps to this same package. Consequently,
their differentials all factor through a single linear map: the differential
of the core analytical package itself.

\smallskip
This universal factorization is the basic linear mechanism underlying the
Einstein Detection Principle.

\subsubsection*{The universal response differential}
Let
$\mathcal A_{E,\mathrm{core}}: S \longrightarrow \mathscr X_{\mathrm{core}}$
denote the descended core analytical package on the chosen smooth
parameterized Einstein family.

\smallskip
Let
$p_0\in S$
denote the parameter corresponding to the reference Einstein metric
$g_{p_0}=g_0.$
\begin{Dfn}[Universal response differential]
\label{def:universal-response-differential}
The \emph{universal response differential} at the reference metric is
$$
\mathcal D_{g_0}
:=
D_{p_0}\mathcal A_{E,\mathrm{core}}
:
T_{p_0}S
\longrightarrow
T_{\mathcal A_{E,\mathrm{core}}(p_0)}
\mathscr X_{\mathrm{core}}.
$$
\end{Dfn}
\smallskip
The map $\mathcal D_{g_0}$ records the complete first-order variation of the
core analytical package
$(J_g,G_g)$
along infinitesimal Einstein deformations tangent to the chosen smooth
parameterized Einstein family. It depends on the fixed probe package but not
on any particular finite-dimensional observation extracted from it.

\smallskip

The two components of the core package are not independent. Since
$G_p=J_p^{-1}$, differentiation of
$J_pG_p=\operatorname{Id}$
shows that the Green response is determined by the Jacobi response. This
reduces the injectivity problem for the universal package differential to a
single operator-valued linear map.

\begin{Dfn}[Effective Jacobi-response operator]
\label{def:effective-jacobi-response}
Let
\[
V:=T_{p_0}S,\qquad h_X:=D\mathbf g_{p_0}[X],
\]
and let
\[
\dot f_X:=D(\Pi_E)_{p_0}[X]=-G_0K_0[X]
\]
be the solution of the metric-forced Jacobi equation.  After the fixed local
trivializations used throughout the paper, define
\[
\mathfrak J_{g_0}:V\longrightarrow
\mathcal L\!\left(
\mathscr H_{p_0}^{k+2,\alpha},
\mathscr Y_{p_0}^{k,\alpha}
\right)
\]
by
\[
\mathfrak J_{g_0}(X)
:=
D(J\circ\Pi_E,\mathbf g)_{p_0}[X]
=
D_gJ_{(g_0,f_0)}[h_X]
+
D_fJ_{(g_0,f_0)}[\dot f_X].
\]
Equivalently, using the forced Jacobi equation,
\[
\mathfrak J_{g_0}(X)
=
D_gJ_{(g_0,f_0)}[h_X]
-
D_fJ_{(g_0,f_0)}\!\left[G_0K_0[X]\right].
\]
\end{Dfn}

Here $D_gJ$ and $D_fJ$ denote the partial derivatives of the Jacobi operator
with respect to the metric and probe variables in the fixed Banach
trivialization. Thus $\mathfrak J_{g_0}$ is computable from the infinitesimal
Einstein variation, the metric forcing $K_0[X]$, and one solution of the
normalized forced Jacobi equation.

\begin{Thm}[Effective Jacobi criterion for universal injectivity]
\label{thm:effective-jacobi-criterion}
Under the standing smoothness and normalized Jacobi-nondegeneracy hypotheses,
the universal response differential of the core package satisfies
\[
\mathcal D_{g_0}(X)
=
\left(
\mathfrak J_{g_0}(X),
-
G_0\,\mathfrak J_{g_0}(X)\,G_0
\right)
\]
in the fixed local trivialization. Consequently,
\[
\ker\mathcal D_{g_0}
=
\ker\mathfrak J_{g_0}.
\]
In particular, the core Jacobi--Green package is infinitesimally complete on
the chosen Einstein family if and only if the effective Jacobi-response
operator
\[
\mathfrak J_{g_0}:T_{p_0}S\longrightarrow
\mathcal L(\mathscr H_{p_0}^{k+2,\alpha},
\mathscr Y_{p_0}^{k,\alpha})
\]
is injective.
\end{Thm}

\begin{proof}
The Jacobi component of
$D_{p_0}\mathcal A_{E,\mathrm{core}}[X]$
is, by the chain rule, precisely
$\mathfrak J_{g_0}(X)$.
Since $G_p=J_p^{-1}$, differentiation of
\[
J_pG_p=\operatorname{Id}
\]
at $p_0$ in the direction $X$ gives
\[
\mathfrak J_{g_0}(X)G_0
+
J_0\,DG_{p_0}[X]
=0.
\]
Multiplying on the left by $G_0=J_0^{-1}$ yields
\[
DG_{p_0}[X]
=
-G_0\mathfrak J_{g_0}(X)G_0.
\]
Hence the displayed formula for $\mathcal D_{g_0}(X)$ follows. Its first
component is $\mathfrak J_{g_0}(X)$, so
$\mathcal D_{g_0}(X)=0$ if and only if
$\mathfrak J_{g_0}(X)=0$.
\end{proof}

The theorem turns the abstract injectivity condition into a concrete
linearized-Jacobi calculation. Since the domain $V$ is finite dimensional,
this operator-valued injectivity can itself be certified by finitely many
linear measurements.

\begin{Thm}[Finite operator-pairing certificate]
\label{thm:finite-jacobi-pairing-certificate}
Let
\[
m=\dim V,
\]
choose a basis
\[
e_1,\ldots,e_m
\]
of $V$, and let
\[
\Lambda_1,\ldots,\Lambda_m
\]
be continuous linear functionals on the Banach space containing the Jacobi
variations
$\mathfrak J_{g_0}(V)$.
Form the matrix
\[
B_{ij}
=
\Lambda_i\!\left(\mathfrak J_{g_0}(e_j)\right).
\]
If
\[
\det B\neq0,
\]
then
\[
\mathfrak J_{g_0}
\quad\text{and hence}\quad
\mathcal D_{g_0}
\]
are injective. Therefore the chosen Einstein family is infinitesimally
observable by the core package, and the finite scalar detection and local
reconstruction conclusions of the preceding theory apply.

Conversely, if $\mathfrak J_{g_0}$ is injective, then such a collection of
$m$ continuous linear functionals exists.
\end{Thm}

\begin{proof}
If $\det B\neq0$ and
\[
X=\sum_{j=1}^m x_je_j
\]
satisfies $\mathfrak J_{g_0}(X)=0$, then
\[
0
=
\Lambda_i(\mathfrak J_{g_0}(X))
=
\sum_{j=1}^m B_{ij}x_j
\]
for every $i$. Thus $Bx=0$, and invertibility of $B$ gives $x=0$.
Theorem~\ref{thm:effective-jacobi-criterion} then gives injectivity of
$\mathcal D_{g_0}$.

Conversely, if $\mathfrak J_{g_0}$ is injective, then
$\mathfrak J_{g_0}(V)$ is an $m$-dimensional subspace of the ambient Banach
operator space. Choose a basis of its dual and extend those covectors to
continuous linear functionals on the ambient Banach space by Hahn--Banach.
The resulting matrix is the identity after choosing the corresponding basis
of $\mathfrak J_{g_0}(V)$, and hence is invertible.
\end{proof}

\begin{Bem}[Scalar cohomogeneity-one implementation]
\label{rem:compressed-scalar-implementation}
For a scalar equivariant reduction written in divergence form
\[
-\frac1{V_g}\frac{d}{dt}\left(V_g\rho'\right)
+U_\rho(t,\rho;g)=0,
\qquad A_g=(\log V_g)',
\]
the reference Jacobi operator is
\[
J_0v=-v''-A_0v'+U_{\rho\rho}(t,\rho_0;g_0)v.
\]
For a genuine Einstein tangent direction $X$, let
\[
\dot A_X=D_gA_{g_0}[h_X],\qquad
\dot U_{\rho,X}=D_g[U_\rho(t,\rho_0;g)]_{g_0}[h_X].
\]
The induced probe variation $w_X$ is the normalized solution of
\[
J_0w_X=\dot A_X\rho_0'-\dot U_{\rho,X},
\]
and direct differentiation gives
\begin{equation}
\label{eq:compressed-effective-jacobi}
\mathfrak J_{g_0}(X)v
=
-\dot A_Xv'
+
\left(
\dot U_{\rho\rho,X}
+
U_{\rho\rho\rho}(t,\rho_0;g_0)w_X
\right)v.
\end{equation}
Thus Theorem~\ref{thm:finite-jacobi-pairing-certificate} reduces, after
symmetry reduction, to linear regular-singular ODE solves followed by a
finite-dimensional rank test.
For the doubly warped identity-type sphere probe
\[
g=dt^2+a(t)^2g_{S^p}+b(t)^2g_{S^q},
\qquad
U(t,\rho;g)=
\frac p{2a(t)^2}\sin^2\rho+
\frac q{2b(t)^2}\cos^2\rho,
\]
put
\[
C_g=\frac p{a^2}-\frac q{b^2},
\qquad
\dot C_X=-2p\frac{\alpha_X}{a_0^3}
+2q\frac{\beta_X}{b_0^3}.
\]
Then
\begin{equation}
\label{eq:compressed-identity-response}
\mathfrak J_{g_0}(X)v
=
-\dot A_Xv'
+
\left[
\dot C_X\cos(2\rho_0)
-
2C_0\sin(2\rho_0)w_X
\right]v.
\end{equation}
These formulas are included to demonstrate that the abstract criterion
is concretely computable in the standard scalar model.
\end{Bem}

\subsubsection*{Package-generated observables}
\begin{Dfn}[Core-package-generated observable]
Let $Y$ be a finite-dimensional vector space. A smooth observable
$\Theta: S \longrightarrow Y$
is called \emph{core-package-generated} if there exists a smooth map
$\Psi: \mathscr X_{\mathrm{core}} \longrightarrow Y$
such that
$\Theta = \Psi\circ \mathcal A_{E,\mathrm{core}}.$
\end{Dfn}
\smallskip
Thus different core-package-generated observables differ only through the
choice of the finite-dimensional postprocessing map $\Psi$.

\smallskip
Typical examples include finite-dimensional evaluations, matrix coefficients,
traces or other smooth functionals constructed from the normalized Jacobi and
Green operators.

\smallskip
When the additional hypotheses required for the extended analytical package
are satisfied, one may similarly form observables from spectral projectors,
heat operators, spectrally separated eigenvalue data, zeta functions or
regularized determinants. Such observables are
\emph{extended-package-generated} and are not required for the Einstein
Detection Principle.

\subsubsection*{Universal factorization theorem}
\begin{Thm}[Universal factorization of response operators]
\label{thm:universal-response}
Let
$\Theta = \Psi\circ \mathcal A_{E,\mathrm{core}} : S \longrightarrow Y$
be a core-package-generated observable, and let
$\mathcal R_{g_0} = D_{p_0}\Theta: T_{p_0}S \longrightarrow Y$
be its canonical response operator.

\smallskip
Then
$\mathcal R_{g_0} = D_{\mathcal A_{E,\mathrm{core}}(p_0)}\Psi \circ \mathcal D_{g_0}.$
Equivalently,
$$D_{p_0}\Theta = D_{\mathcal A_{E,\mathrm{core}}(p_0)}\Psi \circ D_{p_0}\mathcal A_{E,\mathrm{core}}.$$
Consequently, every core-package-generated response operator factors through
the single universal response differential
$\mathcal D_{g_0}.$
\end{Thm}
\begin{proof}
Both maps
$\mathcal A_{E,\mathrm{core}}: S \longrightarrow \mathscr X_{\mathrm{core}}$
and
$\Psi: \mathscr X_{\mathrm{core}} \longrightarrow Y$
are smooth. Since
$\Theta = \Psi\circ \mathcal A_{E,\mathrm{core}},$
the asserted factorization follows immediately from the chain rule.
\end{proof}
\smallskip
The same argument applies to an extended-package-generated observable on any
parameter region on which the relevant extended spectral data are defined
smoothly. The Einstein Detection Principle itself, however, requires only the
core factorization above.

\subsubsection*{Equality criterion}
The universal factorization also characterizes when two postprocessings have
the same first-order response.
\begin{Prop}[Equality criterion]
\label{prop:response-equality-criterion}
Let
$\Theta_i = \Psi_i\circ \mathcal A_{E,\mathrm{core}}, \qquad i=1,2,$
be two core-package-generated observables with values in the same
finite-dimensional vector space $Y$.

\smallskip
Then
$D_{p_0}\Theta_1 = D_{p_0}\Theta_2$
if and only if
$$
D_{\mathcal A_{E,\mathrm{core}}(p_0)}\Psi_1
\big|_{\operatorname{Im}\mathcal D_{g_0}}
=
D_{\mathcal A_{E,\mathrm{core}}(p_0)}\Psi_2
\big|_{\operatorname{Im}\mathcal D_{g_0}}.
$$
\end{Prop}
\begin{proof}
By Theorem~\ref{thm:universal-response},
$D_{p_0}\Theta_i = D_{\mathcal A_{E,\mathrm{core}}(p_0)}\Psi_i \circ \mathcal D_{g_0}.$
Subtracting the two identities gives
$$
D_{p_0}\Theta_1
-
D_{p_0}\Theta_2
=
\left(
D\Psi_1-D\Psi_2
\right)_{\mathcal A_{E,\mathrm{core}}(p_0)}
\circ
\mathcal D_{g_0}.
$$
The left-hand side vanishes precisely when the two differentials
$D\Psi_1$ and $D\Psi_2$ agree on
$\operatorname{Im}\mathcal D_{g_0}.$
\end{proof}

\subsubsection*{Naturality}
The universal response differential inherits the naturality of the descended
core analytical package.
\begin{Prop}[Naturality of the universal response differential]
\label{prop:universal-response-naturality}
Let
$\varphi: (M,g_0) \longrightarrow (M,h_0)$
be a normalization-preserving equivariant isometry preserving the fixed probe
data.

\smallskip
Suppose that the corresponding smooth parameterized Einstein families are
identified by a local diffeomorphism
$F_\varphi: S_g \longrightarrow S_h$
sending the parameter of $g_0$ to the parameter of $h_0$, and let
$\Xi_\varphi: \mathscr X_{\mathrm{core},g} \longrightarrow \mathscr X_{\mathrm{core},h}$
denote the canonical identification of the corresponding core-package data
induced by $\varphi$.

\smallskip
Then
$\mathcal A_{E,\mathrm{core},h}\circ F_\varphi = \Xi_\varphi\circ \mathcal A_{E,\mathrm{core},g}.$
Differentiating at the reference point gives
$\mathcal D_{h_0} \circ D F_\varphi = D\Xi_\varphi \circ \mathcal D_{g_0}.$
Hence the universal response differential is natural under
normalization-preserving equivariant isometries preserving the probe data.
\end{Prop}
\begin{proof}
The first identity is precisely the naturality of the descended core
analytical package established in
Theorem~\ref{thm:package-descent} and
Proposition~\ref{prop:canonicality-package}.

\smallskip
Differentiating this identity and applying the chain rule yields
$D\mathcal A_{E,\mathrm{core},h} \circ D F_\varphi = D\Xi_\varphi \circ D\mathcal A_{E,\mathrm{core},g}.$
By the definition of the universal response differentials, this is exactly
$\mathcal D_{h_0} \circ D F_\varphi = D\Xi_\varphi \circ \mathcal D_{g_0}.$
\end{proof}

\subsubsection*{Universal invisible directions}
The universal response differential determines the infinitesimal directions
that are invisible to every core-package-generated observable.
\begin{Dfn}[Universal invisible space]
\label{def:universal-invisible-space}
The \emph{universal invisible space} at $g_0$ is
$\mathcal K_{g_0} := \ker\mathcal D_{g_0} \subset T_{p_0}S.$
\end{Dfn}
If
$X\in\mathcal K_{g_0},$
then for every core-package-generated observable
$\Theta = \Psi\circ \mathcal A_{E,\mathrm{core}},$
one has
$D_{p_0}\Theta[X] = 0.$
Thus no first-order observation obtained solely by smooth postprocessing of
the core analytical package can detect a direction in
$\mathcal K_{g_0}$.

\smallskip
Conversely, whether every direction outside
$\mathcal K_{g_0}$ can be detected by an admissible finite-dimensional
postprocessing depends on the richness of the chosen class of admissible
observations. This additional issue will be treated explicitly below.

\begin{Bem}
\label{rem:meaning-universal-response}
The significance of
Theorem~\ref{thm:universal-response}
is not the chain rule itself, but the identification of a single universal
linear object,
$\mathcal D_{g_0} = D_{p_0}\mathcal A_{E,\mathrm{core}},$
through which every core-package-generated first-order response factors.

\smallskip
This separates the detection problem into two logically distinct questions.
\begin{enumerate}
\item
Is the universal response differential
$\mathcal D_{g_0}$
injective on
$T_{p_0}S?$
\item
If a vector
$0\neq Z \in \operatorname{Im}\mathcal D_{g_0}$
is present in the observable image, does the admissible class of
finite-dimensional postprocessings contain a differential that does not
annihilate $Z$?
\end{enumerate}
\smallskip
The first question concerns the intrinsic information retained by the core
analytical package. The second concerns the richness of the admissible
observation class.

\smallskip
Together they form the analytical foundation of the Einstein Detection
Principle.
\end{Bem}

\subsection{Observability theory for Einstein metrics}
The Universal Factorization Theorem shows that every
core-package-generated observable determines a linear response operator on the
chosen smooth parameterized Einstein family. The purpose of the present
subsection is to study the intrinsic linear algebra of these response
operators and to identify the maximal first-order information carried by the
core analytical package.

\smallskip
Let
$p_0\in S$
denote the parameter corresponding to the reference Einstein metric
$g_{p_0}=g_0,$
and write
$V:=T_{p_0}S.$
For a core-package-generated observable
$\Theta = \Psi\circ\mathcal A_{E,\mathrm{core}} : S\longrightarrow Y,$
with values in a finite-dimensional vector space $Y$, the associated response
operator is
$\mathcal R_{g_0} = D_{p_0}\Theta: V\longrightarrow Y.$
All statements below are formulated on the fixed smooth parameterized Einstein
family $S$. No tangent-cone response theory and no smoothness assumption on
the full Kuranishi zero set are used.

\subsubsection*{Invisible directions and response spaces}
\begin{Dfn}[Invisible deformation space]
The \emph{invisible deformation space} of the observable $\Theta$ at $g_0$ is
$\mathcal I_{g_0}(\Theta) := \ker\mathcal R_{g_0} \subset V.$
The corresponding \emph{response space} is
$\mathcal O_{g_0}(\Theta) := \operatorname{Im}\mathcal R_{g_0} \subset Y.$
\end{Dfn}

\smallskip
Thus $\mathcal I_{g_0}(\Theta)$ consists precisely of those infinitesimal
Einstein variations tangent to $S$ that produce no first-order variation of
the chosen observable.

\smallskip
A vector
$X\in V$
is called \emph{infinitesimally detectable by $\Theta$} if
$\mathcal R_{g_0}(X)\neq0.$
\subsubsection*{Infinitesimal completeness}
\begin{Dfn}[Infinitesimal completeness]
The observable $\Theta$ is called \emph{infinitesimally complete} at $g_0$
if
$\ker\mathcal R_{g_0} = \{0\}.$
\end{Dfn}

\begin{Prop}
\label{prop:equivalent-completeness}
The following conditions are equivalent.
\begin{enumerate}
\item
$\ker\mathcal R_{g_0} = \{0\};$
\item
$\mathcal R_{g_0}$ is injective;
\item
every nonzero element of $V$ is infinitesimally detectable by $\Theta$;
\item
distinct elements of $V$ produce distinct first-order responses.
\end{enumerate}
\end{Prop}
\begin{proof}
These are equivalent formulations of injectivity of the linear map
$\mathcal R_{g_0}:V\longrightarrow Y.$
\end{proof}

\subsubsection*{Observable quotient}
Since invisible deformation directions cannot be distinguished by the chosen
first-order observation, it is natural to quotient them out.
\begin{Dfn}[Observable infinitesimal quotient]
The \emph{observable infinitesimal quotient} associated with $\Theta$ is
$V_{\mathrm{obs}}(\Theta) := V/\ker\mathcal R_{g_0}.$
\end{Dfn}

\smallskip
The response operator induces a canonical linear isomorphism
$$
\overline{\mathcal R}_{g_0}:
V_{\mathrm{obs}}(\Theta)
\stackrel{\cong}{\longrightarrow}
\operatorname{Im}\mathcal R_{g_0},
$$
defined by
$[X] \longmapsto \mathcal R_{g_0}(X).$
Thus the observable quotient is canonically identified with the space of
realizable first-order responses.

\subsubsection*{Observability rank and stable observability}
The rank of the response operator measures the dimension of the
first-order information retained by the observable.
\begin{Dfn}[Observability rank]
The \emph{observability rank} of $\Theta$ at $g_0$ is
$\operatorname{orank}_{g_0}(\Theta) := \operatorname{rank}\mathcal R_{g_0}.$
\end{Dfn}
It satisfies
$0 \le \operatorname{orank}_{g_0}(\Theta) \le \dim V,$
with equality
$\operatorname{orank}_{g_0}(\Theta)=\dim V$
if and only if $\Theta$ is infinitesimally complete.

\smallskip
Injectivity is qualitative. A quantitative version measures how robustly the
observable distinguishes infinitesimal Einstein directions.

\begin{Dfn}[Stable infinitesimal observability]
Fix norms on $V$ and $Y$. The observable $\Theta$ is called
\emph{stably infinitesimally observable} at $g_0$ if there exists a constant
$c>0$
such that
$\left| \mathcal R_{g_0}X \right| \ge c|X|$
for every
$X\in V.$
\end{Dfn}

\begin{Prop}
\label{prop:stable-observability}
Since $V$ is finite-dimensional, stable infinitesimal observability is
equivalent to injectivity of
$\mathcal R_{g_0}.$
If Hilbert norms are fixed on $V$ and $Y$, then the optimal constant is the
least singular value
$\sigma_{\min}(\mathcal R_{g_0}).$
\end{Prop}

\begin{proof}
If the response operator is bounded below, it is injective.

\smallskip
Conversely, if $\mathcal R_{g_0}$ is injective, the continuous function
$X \longmapsto |\mathcal R_{g_0}X|$
has a positive minimum on the unit sphere of the finite-dimensional space
$V$. This gives the required lower bound.

\smallskip
For Hilbert norms, the optimal lower bound is the least singular value.
\end{proof}

\begin{Bem}[Dependence of quantitative constants on auxiliary norms]
The injectivity, kernel and rank of $\mathcal R_{g_0}$ are intrinsic under
reparameterization of the chosen Einstein family and under the canonical
identifications used in the descent construction. By contrast, the numerical
value of $\sigma_{\min}(\mathcal R_{g_0})$ depends on the Hilbert norms chosen
on the domain and observation spaces. Under changes of bases that are
isometries for these fixed Hilbert structures the singular values are
unchanged, whereas under a general coordinate reparameterization they need
not be. Thus a quoted singular-value bound is a quantitative statement
relative to explicitly fixed norms and coordinates, while positivity of the
least singular value is equivalent to the intrinsic injectivity statement.
\end{Bem}

\subsubsection*{Redundant parameter directions}
The response theory cannot detect a direction that does not change the
underlying metric.

\begin{Prop}[Redundant parameter directions are invisible]
\label{prop:redundant-directions-invisible}
Let $p\in S$ and let $X\in T_pS$. If $D\mathbf g_p[X]=0$, then
$D(\Pi_E)_p[X]=0$ and
$D\mathcal A_{E,\mathrm{core},p}[X]=0$.
Consequently every core-package-generated response operator vanishes on $X$.
In particular, injectivity of the universal response differential at $p$
forces $D\mathbf g_p$ to be injective.
\end{Prop}

\begin{proof}
By definition,
$K_p=D_g\mathcal F(g_p,f_p)\circ D\mathbf g_p$, so
$D\mathbf g_p[X]=0$ implies $K_p[X]=0$. The response identity
$D(\Pi_E)_p=-G_pK_p$ gives $D(\Pi_E)_p[X]=0$.
The core package is constructed from the metric $g_p$ and its canonical probe
$f_p$; therefore its first variation also vanishes in a parameter direction
for which both first variations vanish. Hence
$D\mathcal A_{E,\mathrm{core},p}[X]=0$. Universal factorization then implies
that every core-package-generated response vanishes on $X$.
\end{proof}

\smallskip
Thus no observability theorem can remove redundancy introduced solely by the
choice of parameterization. A nonimmersive parameterization is an immediate
counterexample to any formulation that attempts to detect every nonzero
element of $T_pS$ without accounting for $D\mathbf g_p$.

\subsubsection*{Universal package-detectable quotient}
We now pass from a single observable to the total first-order information
contained in the core analytical package.

\smallskip
Recall the universal response differential
$\mathcal D_{g_0} = D_{p_0}\mathcal A_{E,\mathrm{core}} : V \longrightarrow W,$
where
$W := T_{\mathcal A_{E,\mathrm{core}}(p_0)} \mathscr X_{\mathrm{core}}.$
Its kernel
$\mathcal K_{\mathrm{pkg}} := \ker\mathcal D_{g_0}$
is the universal invisible space introduced above.

\begin{Prop}[Universal detectable quotient]
\label{prop:universal-detectable-quotient}
Every core-package-generated response operator vanishes on
$\mathcal K_{\mathrm{pkg}}.$
Consequently, every such response factors canonically through
$V/\mathcal K_{\mathrm{pkg}}.$
More precisely, for every
$\Theta = \Psi\circ\mathcal A_{E,\mathrm{core}},$
there exists a unique linear map
$\widehat{\mathcal R}_{g_0}: V/\mathcal K_{\mathrm{pkg}} \longrightarrow Y$
such that
$\mathcal R_{g_0} = \widehat{\mathcal R}_{g_0} \circ \pi,$
where
$\pi: V\longrightarrow V/\mathcal K_{\mathrm{pkg}}$
is the quotient map.
\end{Prop}
\begin{proof}
By the Universal Factorization Theorem,
$\mathcal R_{g_0} = D\Psi_{\mathcal A_{E,\mathrm{core}}(p_0)} \circ \mathcal D_{g_0}.$
Hence
$\mathcal K_{\mathrm{pkg}} = \ker\mathcal D_{g_0} \subset \ker\mathcal R_{g_0}.$
The factorization through the quotient then follows from the universal
property of quotient vector spaces.
\end{proof}
\smallskip
Thus
$V/\mathcal K_{\mathrm{pkg}}$
is the maximal quotient of $V$ on which the core analytical package itself
can carry first-order information. Whether every nonzero vector in this
quotient can be detected by an admissible finite-dimensional postprocessing
depends on the richness of the admissible observation class.

\subsubsection*{Joint observables}
Several observables may be combined into a single joint observation.

\smallskip
Let
$$\Theta_j: S\longrightarrow Y_j, \qquad j=1,\ldots,N,$$
be core-package-generated observables.

\smallskip
Their joint response operator is
$$
\mathcal R_{g_0}^{(N)}
=
\left(
D\Theta_1,
\ldots,
D\Theta_N
\right)_{p_0}
:
V
\longrightarrow
\prod_{j=1}^N Y_j.
$$
Its kernel is
$\ker\mathcal R_{g_0}^{(N)} = \bigcap_{j=1}^N \ker D_{p_0}\Theta_j.$
Consequently, the family is infinitesimally complete if and only if
$\bigcap_{j=1}^N \ker D_{p_0}\Theta_j = \{0\}.$

\subsubsection*{Admissible scalar postprocessings}
To formulate a rigorous finite-scalar detection theorem, we must specify
which scalar postprocessings are admissible.

\smallskip
Let
$a_0 = \mathcal A_{E,\mathrm{core}}(p_0),$
and let
$L := \mathcal D_{g_0} : V \longrightarrow W.$
Let $E$ be a Banach space of admissible smooth scalar postprocessings defined
near $a_0$,
$\psi: \mathscr X_{\mathrm{core}} \longrightarrow \mathbb R.$
For each
$\psi\in E,$
its differential at $a_0$ restricts to a covector on the finite-dimensional
subspace
$L(V)\subset W.$
This defines the linear restriction map
$\rho: E \longrightarrow L(V)^*,$
given by
$\rho(\psi) = d\psi_{a_0}\big|_{L(V)}.$

\begin{Hyp}[Scalar postprocessing richness]
\label{hyp:scalar-richness}
The admissible scalar postprocessing class $E$ is said to satisfy the
\emph{scalar richness hypothesis} at $g_0$ if $E$ is a Banach space
continuously embedded in the local $C^1$ postprocessing space, the derivative
restriction map $\rho:E\longrightarrow L(V)^*$ is continuous, and $\rho$ is
surjective.
\end{Hyp}

\smallskip
This hypothesis states exactly what is required for the admissible scalar
observations to realize arbitrary linear measurements of the first-order
information contained in the core analytical package.
Without it, finite scalar completeness can fail even when the universal
response differential is injective; for example, an admissible scalar class
consisting only of constant functions has identically zero differential.

The richness condition is nevertheless automatic for the largest natural
postprocessing class. Thus it is a genuine extra hypothesis only when one
prescribes a restricted family of observables.

\begin{Prop}[Automatic scalar richness for unrestricted local observations]
\label{prop:automatic-scalar-richness}
Let
\[
a_0=\mathcal A_{E,\mathrm{core}}(p_0),\qquad
L=\mathcal D_{g_0}:V\longrightarrow W,
\]
where the local Banach model of $\mathscr X_{\mathrm{core}}$ at $a_0$ has
tangent space $W$. Suppose the admissible scalar postprocessings contain, in
a local Banach chart centred at $a_0$, all restrictions of continuous linear
functionals $\ell\in W^*$ to a neighbourhood of the origin. Then
\[
\rho:E\longrightarrow L(V)^*,\qquad
\rho(\psi)=d\psi_{a_0}\big|_{L(V)},
\]
is surjective. In particular, the scalar richness conclusion required in
Theorem~\ref{thm:finite-scalar-detection} holds automatically for the
unrestricted class of smooth local scalar observations.
\end{Prop}

\begin{proof}
The space $L(V)$ is finite dimensional because $V$ is finite dimensional.
Let $\lambda\in L(V)^*$. Every linear functional on the finite-dimensional
normed space $L(V)$ is continuous, so the Hahn--Banach theorem extends
$\lambda$ to a continuous linear functional
\[
\ell\in W^*,\qquad \ell|_{L(V)}=\lambda.
\]
Choose a local Banach chart
\[
\chi:U\subset\mathscr X_{\mathrm{core}}\longrightarrow W
\]
with $\chi(a_0)=0$ and $D\chi_{a_0}=\operatorname{id}_W$; composing any local
chart with the inverse of its derivative at $a_0$ gives such a chart. The
local scalar observation $\psi=\ell\circ\chi$ is smooth and admissible by
assumption, and
\[
d\psi_{a_0}\big|_{L(V)}
=\ell\big|_{L(V)}
=\lambda.
\]
Thus $\rho$ is surjective.
\end{proof}

\begin{Cor}[Finite scalar detection for unrestricted local observations]
\label{cor:unrestricted-finite-scalar-detection}
Assume that
\[
\mathcal D_{g_0}:T_{p_0}S\longrightarrow
T_{a_0}\mathscr X_{\mathrm{core}}
\]
is injective, and allow unrestricted smooth local scalar postprocessings of
the core package. If $m=\dim S$, then there exist $m$ smooth scalar
core-package-generated observations whose joint differential at $p_0$ is an
isomorphism onto $\mathbb R^m$. Consequently their joint map is a local
diffeomorphism near $p_0$ and locally reconstructs the parameter point. If
$p\mapsto g_p$ is locally injective, it locally reconstructs the corresponding
Einstein metric within the chosen family.
\end{Cor}

\begin{proof}
Proposition~\ref{prop:automatic-scalar-richness} supplies the scalar richness
needed in Theorem~\ref{thm:finite-scalar-detection}. That theorem gives $m$
scalar observations with injective joint differential
$T_{p_0}S\to\mathbb R^m$. Since both spaces have dimension $m$, the
differential is an isomorphism. The finite-dimensional inverse function
theorem gives the local reconstruction statement.
\end{proof}

\begin{Bem}[The genuine first-order obstruction]
\label{rem:genuine-observability-obstruction}
Proposition~\ref{prop:automatic-scalar-richness} removes one possible source
of conditionality but not the essential one. For unrestricted scalar
postprocessings, the only first-order obstruction internal to the core package
is
\[
\ker\mathcal D_{g_0}.
\]
Thus, once the descended package exists, failure of finite scalar detection
cannot be attributed to an insufficient supply of scalar functions: it occurs
precisely when the core Jacobi--Green package itself loses an Einstein tangent
direction. By contrast, if one insists on a geometrically or physically
prescribed restricted observation class, scalar richness again becomes a
substantive condition that must be verified in that model.
\end{Bem}

\subsubsection*{Finite scalar detection}
\begin{Thm}[Finite scalar detection]
\label{thm:finite-scalar-detection}
Let
$m=\dim V,$
and let
$L = \mathcal D_{g_0}: V\longrightarrow W.$
Then the following statements hold.
\begin{enumerate}
\item
Every core-package-generated scalar response vanishes on
$\ker L.$
\item
An infinitesimally complete family of core-package-generated scalar
observations can exist only if $L$ is injective.
\item
Assume that $L$ is injective and that
Hypothesis~\ref{hyp:scalar-richness} holds. Then there exist
$\psi_1,\ldots,\psi_m\in E$
such that the joint scalar observable
$$
\Theta
=
\left(
\psi_1\circ\mathcal A_{E,\mathrm{core}},
\ldots,
\psi_m\circ\mathcal A_{E,\mathrm{core}}
\right)
:
S
\longrightarrow
\mathbb R^m
$$
is infinitesimally complete at $g_0$.
\end{enumerate}
\end{Thm}
\begin{proof}
The first statement follows immediately from the universal factorization
\[
D(\psi\circ\mathcal A_{E,\mathrm{core}})_{p_0}
=
d\psi_{a_0}\circ L.
\]
Hence every scalar response vanishes on $\ker L$.

\smallskip
If $L$ is not injective, every package-generated scalar response vanishes on
the nontrivial space $\ker L$. Therefore no joint family of such observations
can be infinitesimally complete.

\smallskip
Assume now that $L$ is injective. Then
$\dim L(V)=m.$
Choose a basis
$\lambda_1,\ldots,\lambda_m$
of
$L(V)^*.$
By Hypothesis~\ref{hyp:scalar-richness}, there exist
$\psi_1,\ldots,\psi_m\in E$
such that
$d\psi_i{}_{a_0}\big|_{L(V)} = \lambda_i.$
The differential of the resulting joint observable is
\[
D\Theta_{p_0}(X)
=
\bigl(\lambda_1(LX),\ldots,\lambda_m(LX)\bigr).
\]
Since the $\lambda_i$ form a basis of $L(V)^*$ and $L$ is injective, this
map is injective. Hence the joint scalar observable is infinitesimally
complete.
\end{proof}

\subsubsection*{Generic completeness}
The same richness hypothesis gives a rigorous generic-completeness statement.
\smallskip
For
$m=\dim V,$
define
$\mathcal P: E^m \longrightarrow \operatorname{Hom}(V,\mathbb R^m)$
by
\[
\mathcal P(\psi_1,\ldots,\psi_m)(X)
=
\bigl(
d\psi_1{}_{a_0}(LX),\ldots,d\psi_m{}_{a_0}(LX)
\bigr).
\]
\begin{Cor}[Generic scalar completeness]
\label{cor:generic-scalar-completeness}
Assume that
$L:V\longrightarrow W$
is injective and that
Hypothesis~\ref{hyp:scalar-richness} holds.

\smallskip
Then the set of $m$-tuples
$(\psi_1,\ldots,\psi_m)\in E^m$
whose joint response is injective is open and dense in $E^m$.
\end{Cor}
\begin{proof}
Since $L$ is injective and $\rho:E\to L(V)^*$ is surjective, the linear map
$\mathcal P: E^m \longrightarrow \operatorname{Hom}(V,\mathbb R^m)$
is continuous and surjective.

\smallskip
Since $E^m$ and
$\operatorname{Hom}(V,\mathbb R^m)$ are Banach spaces, the Open Mapping
Theorem implies that $\mathcal P$ is open.

\smallskip
The set of invertible maps
$\operatorname{Iso}(V,\mathbb R^m)$
is open and dense in
$\operatorname{Hom}(V,\mathbb R^m).$
Its preimage under the continuous map $\mathcal P$ is therefore open. To see density, let $U\subset E^m$ be a nonempty open set. Since $\mathcal P$ is an open surjection, $\mathcal P(U)$ is a nonempty open subset of $\operatorname{Hom}(V,\mathbb R^m)$, and hence meets the dense set $\operatorname{Iso}(V,\mathbb R^m)$. Thus $U$ meets $\mathcal P^{-1}(\operatorname{Iso}(V,\mathbb R^m))$. Therefore
$\mathcal P^{-1} \left( \operatorname{Iso}(V,\mathbb R^m) \right)$
is open and dense in $E^m$.
\end{proof}

\subsubsection*{Existence of complete scalar observations}
\begin{Cor}[Existence of a complete scalar family]
\label{cor:existence-complete-scalar-family}
Assume that
$\mathcal D_{g_0}$
is injective and that
Hypothesis~\ref{hyp:scalar-richness} holds.

\smallskip
Then there exists an infinitesimally complete
core-package-generated observable with values in
$\mathbb R^{\dim V}.$
\end{Cor}
\begin{proof}
This is the third assertion of
Theorem~\ref{thm:finite-scalar-detection}.
\end{proof}

\subsubsection*{Verification in concrete models}
Observability is a genuine mathematical property and is not a formal
consequence of the existence of the harmonic probe package.

\smallskip
In concrete models, a finite-dimensional response calculation typically
proceeds as follows.
\begin{enumerate}
\item
Choose a basis of the tangent space
$T_{p_0}S.$
\item
Compute the corresponding metric variations of the chosen Einstein family.
\item
Solve the induced metric-forced Jacobi equations.
\item
Evaluate the selected admissible analytical observations.
\item
Assemble the resulting finite-dimensional response matrix.
\item
Verify that the response matrix has rank
$\dim T_{p_0}S.$
\end{enumerate}

\smallskip
If Hilbert structures are fixed, a quantitative verification may additionally
produce a positive lower bound for the least singular value of the response
matrix.

\smallskip
Thus infinitesimal observability is amenable to rigorous analytical or
validated numerical verification in concrete cohomogeneity-one Einstein
problems. Proposition~\ref{prop:automatic-scalar-richness} also separates two
verification tasks cleanly: for unrestricted local scalar postprocessings one
need only establish injectivity of the universal response differential,
whereas a prescribed restricted observation family requires an additional
rank or richness verification.

\section{The Einstein Detection Principle}
The preceding sections developed the two principal ingredients of the present
theory:
\begin{enumerate}
\item
the finite-dimensional geometric reduction of the normalized Einstein
boundary-value problem through the Einstein matching construction;
\item
the core analytical package generated by the canonical harmonic probes.
\end{enumerate}
\smallskip
The geometric reduction replaces the original two-ended Einstein
boundary-value problem by a finite-dimensional matching problem. Independently,
the harmonic probe construction associates with each metric in the chosen
smooth parameterized Einstein family a canonical analytical package consisting
of the normalized Jacobi and Green operators.

\smallskip
The central idea of the present paper is that infinitesimal Einstein
deformations need not be studied only through the linearized Einstein
equations themselves. Instead, they may be studied through the first-order
variation of these canonically associated analytical structures.

\smallskip
The purpose of the present section is to combine the geometric and analytical
constructions into the \emph{Einstein Detection Principle}. The Universal
Factorization Theorem shows that every admissible first-order observation
obtained from the core analytical package factors through a single universal
response differential. Its kernel therefore represents the infinitesimal
Einstein directions that are invisible to every such observation.

\smallskip
Under explicit injectivity and richness hypotheses, the same framework yields
local nonlinear consequences. In particular, the core analytical package
becomes a local complete invariant on the chosen smooth parameterized Einstein
family, and finitely many admissible scalar observations suffice for local
reconstruction.

\subsection*{Standing analytical framework}
Throughout this section we work with the chosen smooth parameterized Einstein
family
$S$
through the reference Einstein metric
$g_{p_0}=g_0.$
The following assumptions and conventions, established in the preceding
sections, remain in force.
\begin{enumerate}
\item
The descended core analytical package
$\mathcal A_{E,\mathrm{core}}: S \longrightarrow \mathscr X_{\mathrm{core}}$
is smooth.
\item
All operator families are represented using the fixed smooth local
trivializations and common graph-domain conventions introduced previously.
\item
The core package consists of the normalized Jacobi and Green operators,
$\mathcal A_{E,\mathrm{core}}(p) = (J_p,G_p).$
\item
Extended spectral constructions are used only when their additional
self-adjointness, spectral-separation or spectral-cut hypotheses are
explicitly assumed.
\item
All tangent spaces, differentials and response operators are taken on the
chosen smooth parameterized Einstein family $S$. No smoothness assumption on
the full Kuranishi zero set is required.
\end{enumerate}
\subsection*{Roadmap}
We first reduce injectivity of the universal response differential to the
effective Jacobi-response operator and give a finite operator-pairing
certificate for that injectivity. We then formulate the First-Order Einstein
Detection Principle and derive its local nonlinear consequences. Finally, for
restricted scalar classes under the scalar postprocessing richness hypothesis,
and automatically for unrestricted smooth local scalar postprocessings, we
obtain finite local reconstruction.

\subsection{First-order detection principle}
Let
$V:=T_{p_0}S$
and let
$$
\mathcal D_{g_0}
=
D_{p_0}\mathcal A_{E,\mathrm{core}}
:
V
\longrightarrow
T_{\mathcal A_{E,\mathrm{core}}(p_0)}
\mathscr X_{\mathrm{core}}
$$
denote the universal response differential.
For unrestricted smooth local scalar postprocessings,
Proposition~\ref{prop:automatic-scalar-richness} sharpens the interpretation
of the universal response differential: the core package is infinitesimally
observable by finitely many scalar measurements if and only if
$\mathcal D_{g_0}$ is injective. Hence $\ker\mathcal D_{g_0}$ is the complete
first-order information-loss space of the Jacobi--Green package itself.

\begin{Thm}[Universal First-Order Einstein Detection Principle]
\label{thm:first-order-detection}
Let
$$\Theta = \Psi\circ\mathcal A_{E,\mathrm{core}} : S \longrightarrow Y$$
be a core-package-generated observable with values in a finite-dimensional
vector space $Y$.

\smallskip
Then its response operator satisfies
$$D_{p_0}\Theta = D_{\mathcal A_{E,\mathrm{core}}(p_0)}\Psi \circ \mathcal D_{g_0}.$$
Consequently:
\begin{enumerate}
\item
$\Theta$ is infinitesimally complete at $g_0$ if and only if
$D_{\mathcal A_{E,\mathrm{core}}(p_0)}\Psi \circ \mathcal D_{g_0}$
is injective on $V$;
\item
no core-package-generated observable, nor any finite family of such
observables, can be infinitesimally complete unless
$\mathcal D_{g_0}$
is injective;
\item
the universal invisible space
$\mathcal K_{g_0} = \ker\mathcal D_{g_0}$
is contained in the kernel of every core-package-generated response operator;
\item
if the admissible postprocessing differentials separate every nonzero vector
of
$\operatorname{Im}\mathcal D_{g_0},$
then an element of $V$ is invisible to every admissible
core-package-generated observable if and only if it belongs to
$\ker\mathcal D_{g_0}.$
\end{enumerate}
\end{Thm}
\begin{proof}
The factorization follows from
$\Theta = \Psi\circ\mathcal A_{E,\mathrm{core}}$
and the chain rule:
$$D_{p_0}\Theta = D_{\mathcal A_{E,\mathrm{core}}(p_0)}\Psi \circ D_{p_0}\mathcal A_{E,\mathrm{core}}.$$
By definition,
$D_{p_0}\mathcal A_{E,\mathrm{core}} = \mathcal D_{g_0}.$
This proves the universal factorization.

\smallskip
The first assertion is exactly the definition of infinitesimal completeness.

\smallskip
For the second and third assertions, if
$X\in\ker\mathcal D_{g_0},$
then
$$
D_{p_0}\Theta[X]
=
D\Psi\!\left(
\mathcal D_{g_0}X
\right)
=
0
$$
for every core-package-generated observable. Thus no family of such
observations can distinguish a nonzero element of
$\ker\mathcal D_{g_0}$.

\smallskip
For the final assertion, let
$X\notin\ker\mathcal D_{g_0}.$
Then
$\mathcal D_{g_0}X\neq0.$
By the separation hypothesis, there exists an admissible postprocessing
$\Psi$ such that
\[
D_{\mathcal A_{E,\mathrm{core}}(p_0)}\Psi
\bigl(\mathcal D_{g_0}X\bigr)\neq0.
\]
Hence the corresponding package-generated observable detects $X$. Therefore
the common invisible space of all admissible package-generated observables is
precisely
$\ker\mathcal D_{g_0}.$
\end{proof}
\begin{Bem}
\label{rem:first-order-detection-scope}
Theorem~\ref{thm:first-order-detection} depends on the existence and descent
of the core analytical package and on no matching-data extension. In
particular, Hypothesis~\ref{hyp:matching-dependence} is not used in its proof.

\smallskip
The universality in
Theorem~\ref{thm:first-order-detection}
is relative to the fixed core analytical package and the chosen class of
admissible postprocessings. The theorem does not claim that every conceivable
observable of Einstein metrics factors through the harmonic probe package.

\smallskip
The first-order detection problem therefore separates naturally into two
questions:
\begin{enumerate}
\item
does the core analytical package retain every infinitesimal direction, that
is, is
$\mathcal D_{g_0}$
injective;
\item
is the admissible class of postprocessings sufficiently rich to detect every
nonzero element of
$\operatorname{Im}\mathcal D_{g_0}?$
\end{enumerate}
The first question concerns the information encoded by the canonical probe
package itself. The second concerns the resolving power of the admissible
observation class.
\end{Bem}
\begin{Bem}[Reduction of the conditional hypotheses]
\label{rem:conditionality-reduction}
There are two conceptually different conditions in the detection theory.
Injectivity of $\mathcal D_{g_0}$ is a property of the chosen Einstein family
and the harmonic probe package and remains genuinely geometric and
model-dependent. Scalar richness, by contrast, is not an additional
obstruction when arbitrary smooth local functions of the package are admitted:
Proposition~\ref{prop:automatic-scalar-richness} proves it automatically.
Accordingly, in the unrestricted observation class the substantive
first-order question is reduced to the single condition
$\ker\mathcal D_{g_0}=\{0\}$.
\end{Bem}

\subsection{Local nonlinear detection}
The results of this subsection depend only on the smooth descended core
package on the finite-dimensional parameter manifold $S$ and on the stated
injectivity/richness assumptions. They do not use the matching-data extension,
the Kuranishi reduction, or any extended spectral invariant.

\smallskip
The Universal First-Order Einstein Detection Principle is an infinitesimal
statement. Because the chosen parameter space $S$ is finite-dimensional,
injectivity of the universal response differential also has a local nonlinear
consequence.

\smallskip
The precise statement requires one standard Banach-manifold step: before
projecting onto the finite-dimensional image of the differential, the target
Banach manifold must first be represented in a local chart.

\subsubsection*{Local reconstruction from the core analytical package}
\begin{Thm}[Local reconstruction from the core analytical package]
\label{thm:local-package-detection}
Let $S$ be the chosen finite-dimensional smooth parameter manifold and assume
that $\mathscr X_{\mathrm{core}}$ is a Banach manifold in a neighbourhood of
$a_0=\mathcal A_{E,\mathrm{core}}(p_0)$. Assume that
$$
\mathcal D_{g_0}
=
D_{p_0}\mathcal A_{E,\mathrm{core}}
:
T_{p_0}S
\longrightarrow
T_{a_0}\mathscr X_{\mathrm{core}},
$$
where
$a_0 = \mathcal A_{E,\mathrm{core}}(p_0),$
is injective.

\smallskip
Then, after shrinking $S$ around $p_0$, the map
$\mathcal A_{E,\mathrm{core}}$
is a smooth embedding onto its image.

\smallskip
Consequently, the parameter point $p\in S$ is locally determined by its core
analytical package. If the parameterization $p\mapsto g_p$ is locally
injective (in particular, on an immersed family after shrinking to an
embedded neighbourhood), the corresponding Einstein metric is locally
uniquely determined as well.

\smallskip
Neither conclusion is global. An injective differential at one point does not
exclude self-intersections or repeated package values far from that point,
and a redundant parameterization may assign the same metric to distinct
nearby parameter values unless local injectivity of $p\mapsto g_p$ is imposed.
\end{Thm}
\begin{proof}
Work in a Banach chart around
$\mathcal A_{E,\mathrm{core}}(p_0)$ and write the package map as
$F:S\to B$. Since $DF_{p_0}$ is injective and $T_{p_0}S$ is finite
dimensional, $E:=DF_{p_0}(T_{p_0}S)$ is finite dimensional and therefore
complemented in $B$. Let $P:B\to E$ be a bounded projection. Then
$D(P\circ F)_{p_0}:T_{p_0}S\to E$ is an isomorphism. The
finite-dimensional inverse function theorem makes $P\circ F$ a local
diffeomorphism after shrinking $S$. Consequently $F$, and hence
$\mathcal A_{E,\mathrm{core}}$, is injective and an immersion there, with
smooth inverse on its image obtained from $(P\circ F)^{-1}\circ P$.
Thus the package map is a local smooth embedding.
\end{proof}

\subsubsection*{Finite scalar reconstruction}
The preceding theorem uses the entire core analytical package. Under the
scalar postprocessing richness hypothesis introduced in the observability
theory, finitely many admissible scalar observations already suffice.
\begin{Thm}[Finite scalar reconstruction]
\label{thm:finite-local-detection}
Let
$m=\dim S,$
and assume that
$\mathcal D_{g_0}: T_{p_0}S \longrightarrow T_{a_0}\mathscr X_{\mathrm{core}}$
is injective.
\smallskip
Assume furthermore that the admissible scalar postprocessing class $E$
satisfies
Hypothesis~\ref{hyp:scalar-richness}.
\smallskip
Then there exist admissible smooth scalar postprocessings
$\psi_1,\ldots,\psi_m\in E$
such that the joint observable
$$
F
=
\left(
\psi_1\circ\mathcal A_{E,\mathrm{core}},
\ldots,
\psi_m\circ\mathcal A_{E,\mathrm{core}}
\right)
:
S
\longrightarrow
\mathbb R^m
$$
has invertible differential at $p_0$.
\smallskip
Consequently, after shrinking $S$ around $p_0$,
$F$
is a local diffeomorphism onto an open subset of $\mathbb R^m$. In particular, the $m$ scalar observations
$\psi_1,\ldots,\psi_m$ uniquely determine the parameter point $p$ near $p_0$.
If the parameterization $p\mapsto g_p$ is locally injective, they therefore
uniquely determine the corresponding Einstein metric within the chosen
family.
\end{Thm}
\begin{proof}
By Theorem~\ref{thm:finite-scalar-detection}, injectivity of
$\mathcal D_{g_0}$ together with
Hypothesis~\ref{hyp:scalar-richness}
provides admissible scalar postprocessings
$\psi_1,\ldots,\psi_m$
for which the joint differential
$D_{p_0}F: T_{p_0}S \longrightarrow \mathbb R^m$
is injective.

\smallskip
Since
$\dim T_{p_0}S = m = \dim\mathbb R^m,$
the differential is an isomorphism.
\smallskip
The finite-dimensional inverse function theorem therefore implies that $F$
is a local diffeomorphism at $p_0$.
\end{proof}
\begin{Cor}[Generic local scalar reconstruction]
\label{cor:generic-local-reconstruction}
Assume the hypotheses of
Theorem~\ref{thm:finite-local-detection}.
Then every admissible $m$-tuple whose joint differential at $p_0$ is
invertible provides a local scalar reconstruction map.

\smallskip
Moreover, under the hypotheses of
Corollary~\ref{cor:generic-scalar-completeness}, such $m$-tuples form an open
dense subset of the admissible product space $E^m$.
\end{Cor}
\begin{proof}
Invertibility of the joint differential implies local reconstruction by the
finite-dimensional inverse function theorem.

\smallskip
The open-dense assertion is precisely
Corollary~\ref{cor:generic-scalar-completeness}.
\end{proof}
\begin{Bem}
\label{rem:detection-vs-reconstruction}
The preceding results separate local Einstein reconstruction into two
independent ingredients.
\begin{enumerate}
\item
The universal response differential
\[
\mathcal D_{g_0}
\]
must be injective. This is the geometric--analytical observability condition
for the core harmonic probe package.
\item
The admissible scalar postprocessing class must be sufficiently rich to
realize enough independent covectors on
$\operatorname{Im}\mathcal D_{g_0}.$
\end{enumerate}
Under these two hypotheses, the chosen smooth parameterized Einstein family is
locally reconstructible from finitely many scalar observations generated by
the canonical harmonic probe package.

\smallskip
The conclusion is local and relative to the chosen smooth parameterized
Einstein family and the fixed probe data. It does not assert reconstruction
of an entire possibly singular Einstein moduli space, nor does it assert that
every infinitesimal Einstein deformation outside the chosen family is
detected.
A crossing of two different local Einstein branches gives the basic
counterexample to any stronger statement: reconstruction on one branch says
nothing about a second branch through the same Kuranishi point unless that
branch is included in the domain of the reconstruction theorem.
\end{Bem}

\subsection{Constant-rank observability and invisible foliations}
The preceding subsections treated complete infinitesimal observability and
local reconstruction. We now consider the intermediate situation in which a
response map has positive but nonmaximal rank.

\smallskip
In this setting, the constant-rank theorem gives a local geometric realization
of the invisible infinitesimal directions.
\begin{Prop}[Constant-rank observability]
\label{prop:constant-rank-observability}
Let
$\Theta: S\longrightarrow Y$
be a smooth analytical observable, where $Y$ is a finite-dimensional
manifold or vector space.

\smallskip
Assume that
$D\Theta$
has constant rank
$r$
on a neighbourhood $U\subset S$ of the reference point $p_0$.

\smallskip
Then, after possibly shrinking $U$, the following statements hold.
\begin{enumerate}
\item
For every $y\in\Theta(U)$, each nonempty level set
$\Theta^{-1}(y)\cap U$
is a smooth embedded submanifold of dimension
$\dim S-r.$
\item
At every $p\in U$,
$T_p\!\left( \Theta^{-1}(\Theta(p))\cap U \right) = \ker D\Theta_p.$
\item
There exist local coordinates on $U$ in which $\Theta$ depends only on
$r$ transverse coordinates. Consequently, the level sets form a local smooth
foliation of $U$ by invisible leaves.
\item
The local leaf space is represented by an $r$-dimensional transverse
coordinate space, and the observable induces local coordinates on this
transverse quotient up to a smooth change of target coordinates.
\end{enumerate}
\end{Prop}

\begin{proof}
This is the constant-rank theorem for smooth maps between finite-dimensional
manifolds.

\smallskip
Near every point of $U$, there exist local coordinates
$(x_1,\ldots,x_n)$
on the domain and suitable local coordinates on the target such that
$\Theta(x_1,\ldots,x_n) = (x_1,\ldots,x_r,0,\ldots,0).$
The level sets are therefore obtained by fixing
$x_1,\ldots,x_r$
and are smooth submanifolds of codimension $r$.

\smallskip
Their tangent spaces are precisely the vectors annihilated by
$D\Theta$, proving
$T_p\!\left( \Theta^{-1}(\Theta(p)) \right) = \ker D\Theta_p.$
The remaining coordinates
$(x_1,\ldots,x_r)$
provide the asserted local transverse parameterization.
\end{proof}

\smallskip
Thus, under the constant-rank hypothesis, the infinitesimal observable
quotient
$T_pS/\ker D\Theta_p$
admits a local geometric realization as the tangent space of the transverse
coordinate space supplied by the constant-rank theorem.

\smallskip
At the reference point this agrees with the observable infinitesimal quotient
$V_{\mathrm{obs}}(\Theta) = T_{p_0}S/ \ker D\Theta_{p_0}.$

\smallskip
Without the constant-rank hypothesis, the dimensions of the kernels may
change and no smooth foliation by level sets is implied.
 The elementary map $(x,y)\mapsto x^2$ already shows the issue: the rank drops
on $x=0$, so the kernel distribution changes dimension and cannot define the
constant-rank foliation asserted above across that locus.
Even under constant rank, the conclusion is local: globally, level sets may
have several connected components and the global quotient by level sets need
not be a manifold or Hausdorff. In that general
case,
$T_{p_0}S/ \ker D\Theta_{p_0}$
should be regarded only as an infinitesimal quotient and not automatically as
the tangent space of a genuine local quotient manifold.
\begin{Bem}
When $D\Theta$ has constant rank, the observable is constant along the local
invisible leaves and varies only in the transverse directions.

\smallskip
Thus partial observability has a natural local geometric interpretation:
the chosen smooth Einstein family decomposes locally into deformation
directions invisible to the observable and complementary directions detected
by its first-order response.
\end{Bem}

\subsection{Rigidity and singular local deformation models}
The Einstein Detection Principle is formulated on a chosen smooth
parameterized Einstein family. Two limiting situations require separate
comment: a zero-dimensional chosen family and a singular Kuranishi zero set.

\subsubsection*{Zero-dimensional parameterized families}
\begin{Cor}
\label{cor:rigidity-detection}
Assume that the chosen smooth parameterized Einstein family $S$ is
zero-dimensional near $p_0$. Equivalently,
$T_{p_0}S = \{0\}.$
Then every smooth analytical observable
$\Theta: S\longrightarrow Y$
has trivial first-order response at $p_0$:
$D_{p_0}\Theta = 0.$
\end{Cor}
\begin{proof}
The domain of the differential is the zero vector space. Hence every linear
map
$T_{p_0}S \longrightarrow T_{\Theta(p_0)}Y$
is the zero map.
\end{proof}
\smallskip
This statement should not be confused with local rigidity of the full
Einstein problem.

\smallskip
A zero-dimensional chosen parameterized Einstein family contains no
nontrivial tangent directions to detect, but this alone does not rule out
other nearby Einstein families, distinct local branches, or nonlinear
solutions not contained in the chosen family.

\smallskip
Local rigidity for the full normalized Einstein problem requires additional
information, for example infinitesimal rigidity together with the standard
gauge-fixed local rigidity theorem, or a direct analysis of the Kuranishi
obstruction map.

\subsubsection*{Singular Kuranishi points}
The second limiting situation occurs when the Kuranishi zero set
$\kappa^{-1}(0)$
is singular at the reference metric.

\smallskip
In this case there may exist several smooth parameterized Einstein families
through the same reference point. The response theory on one chosen family
$S$
controls only infinitesimal and local nonlinear deformations contained in
that family.

\smallskip
Thus injectivity of
$$
\mathcal D_{g_0}:
T_{p_0}S
\longrightarrow
T_{\mathcal A_{E,\mathrm{core}}(p_0)}
\mathscr X_{\mathrm{core}}
$$
does not exclude the existence of other Einstein branches through the same
Kuranishi point.

\smallskip
Similarly, a local reconstruction theorem on $S$ does not by itself yield
reconstruction on the full possibly singular Kuranishi zero set.

\smallskip
A statement about all nearby Einstein solutions at such a point therefore
requires additional information. For example, one may assume that
\begin{enumerate}
\item
the chosen family $S$ contains all nearby normalized Einstein solutions under
consideration; or
\item
all relevant local smooth Einstein families have been identified and the
response theory has been established on each of them together with the
necessary compatibility on their intersections.
\end{enumerate}
Unless explicitly stated otherwise, every detection and reconstruction result
in the present paper is understood relative to the chosen smooth
parameterized Einstein family.
 A second smooth branch through the same singular Kuranishi point is therefore
a counterexample to any attempted promotion of a branchwise reconstruction
theorem to the whole local moduli germ without additional branching
information.
\begin{Bem}
The present theory does not assume a canonical smooth stratification of the
Kuranishi zero set

\smallskip
Rather, it is a theory on smooth parameterized Einstein families contained in
that zero set. At a singular Kuranishi point, additional analysis of the local
branching structure is required before conclusions obtained on individual
families can be promoted to statements about all nearby Einstein solutions.
\end{Bem}

\subsection{Structure and scope of the Einstein Detection Principle}
The logical dependencies give a strict separation between the core theorem and
its optional interfaces. The core detection chain is
\[
\begin{aligned}
\text{Jacobi nondegeneracy}
&\Longrightarrow \text{canonical probe branch}
\Longrightarrow (J,G)\\
&\Longrightarrow \text{descent to }S
\Longrightarrow \mathcal D_{g_0}\\
&\Longrightarrow \text{conditional detection/reconstruction}.
\end{aligned}
\]
The matching-data extension and the extended spectral package are lateral
extensions of this chain, not prerequisites for it. The former connects the
core response to $\Sigma$ and the latter enlarges the class of observables.
We may now summarize the structure of the Einstein Detection Principle.

\smallskip

The theory consists of two logically distinct stages.

\smallskip
The first stage is universal. Starting from the fixed auxiliary harmonic
probe data, one constructs the canonical harmonic probes and the descended
core analytical package
$\mathcal A_{E,\mathrm{core}} = (J_g,G_g).$
Differentiation along the chosen smooth parameterized Einstein family produces
the universal response differential
$$
\mathcal D_{g_0}
=
D_{p_0}\mathcal A_{E,\mathrm{core}}
:
T_{p_0}S
\longrightarrow
T_{\mathcal A_{E,\mathrm{core}}(p_0)}
\mathscr X_{\mathrm{core}}.
$$
Every core-package-generated observable
$\Theta = \Psi\circ\mathcal A_{E,\mathrm{core}}$
then satisfies
$$D_{p_0}\Theta = D_{\mathcal A_{E,\mathrm{core}}(p_0)}\Psi \circ \mathcal D_{g_0}.$$
Thus every first-order observation generated by the core package factors
through the same universal linear map.

\smallskip
The second stage is model-dependent. For a concrete Einstein family one must
determine the kernel of
$\mathcal D_{g_0}$.
Theorem~\ref{thm:effective-jacobi-criterion} reduces this to the kernel of the
effective Jacobi-response operator $\mathfrak J_{g_0}$, while
Theorem~\ref{thm:finite-jacobi-pairing-certificate} gives a finite determinant
certificate for injectivity.  Remark~\ref{rem:compressed-scalar-implementation} records the scalar
cohomogeneity-one implementation and the doubly warped identity-type
specialization in compact form.  The round $(2,7)$ model is used there only
as a scope check: its exact shooting modes are not identified here with
nonzero tangent vectors to a smooth global normalized Einstein family. For unrestricted smooth local scalar postprocessings no
further richness obstruction remains; for a prescribed restricted observation
class one must additionally verify richness on
$\operatorname{Im}\mathcal D_{g_0}$.
The architecture of the theory may therefore be summarized schematically as
$$
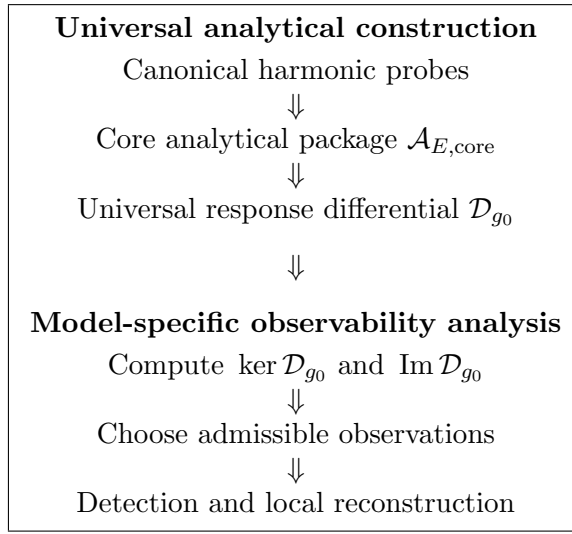
\boxed{
\begin{array}{c}
\textbf{Universal analytical construction}
\\[1mm]
\text{Canonical harmonic probes}
\\
\Downarrow
\\
\text{Core analytical package }
\mathcal A_{E,\mathrm{core}}
\\
\Downarrow
\\
\text{Universal response differential }
\mathcal D_{g_0}
\\[3mm]
\Downarrow
\\[3mm]
\textbf{Model-specific observability analysis}
\\[1mm]
\text{Compute }
\ker\mathcal D_{g_0}
\text{ and }
\operatorname{Im}\mathcal D_{g_0}
\\
\Downarrow
\\
\text{Choose admissible observations}
\\
\Downarrow
\\
\text{Detection and local reconstruction}
\end{array}
}
$$

\smallskip
If
$\mathcal D_{g_0}$
is injective, then the core analytical package retains all first-order
information along the chosen smooth parameterized Einstein family.

\smallskip
By
Theorem~\ref{thm:local-package-detection}, this injectivity also implies that
the complete core analytical package is locally an embedding of $S$ near
$p_0$ and hence locally determines the Einstein metric within that family.

\smallskip
If, in addition, the scalar postprocessing richness hypothesis holds, then
Theorem~\ref{thm:finite-local-detection} shows that
$\dim S$
admissible scalar observations may be chosen so that their joint map gives
local coordinates on $S$ near the reference metric.

\smallskip
Accordingly, the Einstein Detection Principle separates the local Einstein
deformation problem into the following two tasks.
\begin{enumerate}
\item
Construct the canonical core analytical package and its universal response
differential.
\item
For the concrete Einstein family under consideration, determine the
observability properties of the universal response differential and construct
a sufficiently rich admissible family of finite-dimensional observations.
\end{enumerate}
Once these tasks are completed, the general theory yields infinitesimal
detection and, under the hypotheses stated above, local reconstruction on the
chosen smooth parameterized Einstein family.
\begin{Bem}
The auxiliary harmonic probes are analytical intermediaries. Their role is to
generate the Jacobi and Green structures from which the response theory is
constructed.

\smallskip
After these analytical structures have been formed, the detection problem is
expressed entirely in terms of the descended core analytical package, its
universal response differential and the chosen admissible observations.

\smallskip
What survives the elimination of the auxiliary probe variables is therefore
a canonical analytical structure on the chosen smooth parameterized Einstein
family, relative to the fixed normalization and probe data.

\smallskip
No claim is made that this structure is defined on an entire singular Einstein
moduli space without additional hypotheses.
\end{Bem}

\subsection{An explicit illustration of the Einstein Detection Principle:
the homothetic round family}
\label{subsec:round-homothetic-detection}
We conclude the theoretical development with a completely explicit model
illustrating the basic response mechanism of the Einstein Detection
Principle. The example belongs to the \emph{unnormalized} Einstein problem,
since the Einstein constant varies along the family, but every relevant
quantity can be computed in closed form.

\smallskip
The example shows directly how a canonically associated harmonic probe gives
rise to a scalar analytical observable whose differential detects the unique
infinitesimal direction in a one-dimensional Einstein family.

\smallskip
Let
$m\ge3,$
let $g_{\mathrm{rd}}$ denote the unit round metric on $S^m$, and consider the
smooth one-parameter family
$$g_r = r^2g_{\mathrm{rd}}, \qquad r>0.$$
Since constant rescaling leaves the Ricci tensor unchanged as a
$(0,2)$-tensor,
$$\Ric(g_r) = (m-1)g_{\mathrm{rd}} = \frac{m-1}{r^2}g_r.$$
Thus every $g_r$ is Einstein, with Einstein constant
$\lambda_r = \frac{m-1}{r^2}.$
The parameter interval
$S_{\mathrm{hom}} = (0,\infty)$
therefore defines a one-dimensional smooth parameterized family of Einstein
metrics through
$r\longmapsto g_r.$

\subsubsection*{The canonical harmonic probe}
Fix the target
$(N,h) = (S^m,g_{\mathrm{rd}})$
and consider the identity map
$$f_r = \operatorname{Id}_{S^m} : (S^m,g_r) \longrightarrow (S^m,g_{\mathrm{rd}}).$$
A constant rescaling of the domain metric does not change its Levi--Civita
connection. Hence the identity map is harmonic for every
$r>0.$
Thus
$r\longmapsto f_r$
gives a smoothly varying harmonic probe family along
$S_{\mathrm{hom}}$.

\smallskip
For $m\ge3$, the normalized Jacobi nondegeneracy calculation of
Section~\ref{subsec:round-sphere-probe} applies to the corresponding round
identity probe, up to the harmless positive scaling induced by $g_r$.
Consequently, after fixing the same reference-probe branch and probe
normalization, the identity maps provide the locally unique canonical harmonic
probes along the homothetic family.

\subsubsection*{The harmonic-energy observable}
As a scalar analytical observable we choose the harmonic-map energy
$$
\Theta(r)
=
E_{g_r}(f_r)
=
\frac12
\int_{S^m}
|df_r|_{g_r,g_{\mathrm{rd}}}^{\,2}
\,d\operatorname{vol}_{g_r}.
$$
Since
$g_r^{-1} = r^{-2}g_{\mathrm{rd}}^{-1}$
and
$d\operatorname{vol}_{g_r} = r^m \,d\operatorname{vol}_{g_{\mathrm{rd}}},$
we have
$|df_r|_{g_r,g_{\mathrm{rd}}}^{\,2} = \operatorname{tr}_{g_r}(g_{\mathrm{rd}}) = mr^{-2}.$
Therefore
\begin{equation}
\label{eq:round-probe-energy}
\Theta(r)
=
\frac{m}{2}
r^{m-2}
\operatorname{Vol}(S^m,g_{\mathrm{rd}}).
\end{equation}
Differentiating gives
\begin{equation}
\label{eq:round-probe-energy-response}
\Theta'(r)
=
\frac{m(m-2)}{2}
r^{m-3}
\operatorname{Vol}(S^m,g_{\mathrm{rd}}).
\end{equation}
Since
$m\ge3,$
we obtain
$\Theta'(r)>0$
for every
$r>0.$
Thus the harmonic energy varies strictly monotonically along the entire
homothetic Einstein family.

\begin{Prop}[Detection on the homothetic round family]
\label{prop:round-homothetic-detection}
Let
$S_{\mathrm{hom}} = (0,\infty)$
denote the homothetic round Einstein family, parametrized by $r$.

\smallskip
The response operator associated with the harmonic-energy observable is
injective at every point of this family.

\smallskip
More precisely,
$D_r\Theta: T_rS_{\mathrm{hom}} \longrightarrow \mathbb R$
is given by
$$D_r\Theta(\dot r) = \frac{m(m-2)}{2} r^{m-3} \operatorname{Vol}(S^m,g_{\mathrm{rd}}) \,\dot r.$$
Hence
$$\ker D_r\Theta = \{0\}.$$
Consequently, the harmonic-energy observable detects the unique
infinitesimal deformation direction of the homothetic round family.

\smallskip
Moreover, $\Theta$ itself is a local coordinate on
$S_{\mathrm{hom}}$ at every point.
\end{Prop}
\begin{proof}
Equation~\eqref{eq:round-probe-energy-response} gives
$\Theta'(r) = \frac{m(m-2)}{2} r^{m-3} \operatorname{Vol}(S^m,g_{\mathrm{rd}}).$
For
$m\ge3$
and
$r>0,$
this quantity is strictly positive. Since
$T_rS_{\mathrm{hom}}$
is one-dimensional, the response operator is therefore injective.

\smallskip
The finite-dimensional inverse function theorem then implies that $\Theta$
is a local coordinate on the homothetic family near every value of $r$.
\end{proof}

\begin{Bem}
\label{rem:round-homothetic-normalization}
This example illustrates the response mechanism in its simplest nontrivial
form. The Einstein family is one-dimensional, and a single scalar analytical
observable has nonvanishing differential and therefore detects the entire
infinitesimal deformation space.

\smallskip
The example is deliberately unnormalized. Along the family,
$\lambda_r = \frac{m-1}{r^2}$
varies with $r$. If instead one imposes the normalization
$\Ric(g) = (m-1)g,$
then the homothetic family contains only
$r=1.$
Thus the scaling direction disappears from the normalized Einstein problem.

\smallskip
This distinction illustrates precisely why the normalization used throughout
the main theory is essential: homothetic variation should not be counted as a
genuine normalized Einstein deformation.

\smallskip
The nontrivial applications of the Einstein Detection Principle therefore
concern smooth parameterized families in which genuine deformation
directions remain after the Einstein normalization has been fixed.

\smallskip
Strictly speaking, the harmonic-energy observable used here is an observable
of the canonical harmonic probe assignment. Unless an additional
factorization through the core package
$(J_g,G_g)$
has been established, we do not claim that the energy itself is
core-package-generated. The example is therefore an explicit illustration
of the detection mechanism rather than a verification of every structural
factorization theorem proved above.
In particular, it must not be cited as evidence that
$D\mathcal A_{E,\mathrm{core}}$ is injective: the computation proves
injectivity only for the displayed energy response along the unnormalized
one-dimensional family.
\end{Bem}

\section{Numerical realization of the Einstein Detection Principle}
\label{sec:numerical-realization}

The preceding sections establish the Einstein Detection Principle as a general
analytical framework for studying local Einstein deformations. The purpose of
the present section is to demonstrate that the finite-dimensional response
mechanism underlying this framework can be implemented successfully in a
nontrivial cohomogeneity-one Einstein shooting problem.

\smallskip

The emphasis is not on proving the existence of a previously unknown Einstein
metric. Rather, the objective is to realize concretely the main ingredients of
the construction: a normalized Einstein shooting problem, a geometrically
defined matching hypersurface, fixed harmonic-probe channels, a
finite-dimensional observation map, and its first-order response.

\smallskip

The two parameters used below parametrize
smooth initial Einstein germs and their propagated trajectories. They do
\emph{not} automatically parametrize a positive-dimensional family of global
normalized Einstein metrics. Accordingly, the matrix computed below is the
differential of a probe observation map on a two-dimensional
\emph{shooting family}. It is not identified with the universal response
differential
$\mathcal D_{g_0}=D\mathcal A_{E,\mathrm{core}}$
on a positive-dimensional global Einstein family unless the additional
restriction and factorization hypotheses from the abstract theory are
verified.

\smallskip

Throughout this section the Einstein equation is normalized by
\[
\Ric(g)=7g.
\]

\subsection{The cohomogeneity-one Einstein boundary-value problem}

We consider the normalized cohomogeneity-one Einstein equation associated with
the principal factor dimensions
\[
(2,2,3).
\]
On the regular part of the orbit interval, the invariant metric is described
by three positive warping functions
\[
a(t),\qquad b(t),\qquad c(t),
\]
satisfying the reduced Einstein equations together with the smoothness
conditions at the initial singular orbit.

\smallskip

For the shooting problem considered here, the admissible smooth Einstein germs
form locally a two-parameter family. Rather than working directly with the
positive initial metric coefficients, we introduce logarithmic coordinates
\[
x=\log b_0,
\qquad
y=\log\!\left(\frac{c_0}{\sqrt2\,b_0}\right).
\]
These remove the positivity constraints and provide convenient local
coordinates on the chosen family of regular Einstein germs. We denote the
resulting local shooting parameter space by
\[
P_{\mathrm{shoot}}\subset\mathbb R^2.
\]
Thus every
$p=(x,y)\in P_{\mathrm{shoot}}$
determines a smooth initial Einstein germ and, as long as the reduced solution
remains regular, a corresponding Einstein trajectory.

\smallskip

The numerical computations below are performed entirely on
$P_{\mathrm{shoot}}$. Only after imposing the appropriate second-end matching
condition does a shooting trajectory represent a global Einstein metric.
Consequently, the two-dimensional shooting space is not itself identified
with the local normalized Einstein moduli space.

\subsection{The canonical matching hypersurface}

Let
\[
H:\mathcal P\longrightarrow\mathbb R
\]
be the mean-curvature function introduced in
Section~\ref{chap:intrinsic-moduli}. Along every regular positive Einstein
trajectory,
\[
\frac{dH}{dt}
=
-7-\operatorname{Tr}(L^2)<0.
\]
Hence every intersection with the level $H=0$ is transverse and unique.

\smallskip

The matching hypersurface is the fixed geometrically defined hypersurface
\[
\Sigma=H^{-1}(0)\subset\mathcal P.
\]
For a shooting parameter $p=(x,y)$, let $t_H(p)$ denote the unique hitting
time satisfying
\[
H\bigl(g_p(t_H(p)),L_p(t_H(p))\bigr)=0.
\]
Thus the hitting time depends on the trajectory, whereas the hypersurface
$\Sigma$ itself is fixed.

\smallskip

The corresponding propagation map
\[
\Phi_{\mathrm{shoot}}:
P_{\mathrm{shoot}}\longrightarrow\Sigma
\]
is defined by
\[
\Phi_{\mathrm{shoot}}(p)
=
\bigl(g_p(t_H(p)),L_p(t_H(p))\bigr).
\]
Transversality of the $H=0$ event implies that $t_H(p)$ and
$\Phi_{\mathrm{shoot}}(p)$ depend smoothly on $p$ in a sufficiently small
neighbourhood of the reference trajectory.

\smallskip

This geometrically defined observation surface is the concrete matching
section used throughout the numerical calculation.

\subsection{Harmonic probe channels}

For the present shooting problem we employ two fixed probe channels, labelled
by
\[
\kappa=0.3,
\qquad
\kappa=1.
\]
These values are fixed before the response calculation and are not adapted to
the resulting numerical data.

\smallskip

For each channel, the computation solves the corresponding normalized
linearized probe equation along the Einstein trajectory; we denote the
resulting variation field by $\eta_\kappa$. Under the well-posedness and
identification hypotheses of the abstract probe theory, these fields are the
metric-forced Jacobi responses of the canonical nonlinear probe branch. The
numerical calculation itself tests the resulting linear response system and
does not use this identification as an additional numerical conclusion.

\smallskip

At the moving matching point $t=t_H(p)$ we record the probe values and their
normal derivatives,
\[
\left(
\eta_{0.3},
\eta'_{0.3},
\eta_1,
\eta'_1
\right).
\]
Accordingly, the numerical procedure defines the smooth observation map
\[
\mathcal O:
P_{\mathrm{shoot}}\longrightarrow\mathbb R^4
\]
by
\[
\mathcal O(p)
=
\begin{pmatrix}
\eta_{0.3}\\
\eta'_{0.3}\\
\eta_1\\
\eta'_1
\end{pmatrix}_{t=t_H(p)}.
\]

\smallskip

Whenever the matching-data-locality hypothesis
(Hypothesis~\ref{hyp:matching-dependence}) holds for these observations, there
is a corresponding map
\[
\Theta:\Sigma\longrightarrow\mathbb R^4
\]
such that
\[
\mathcal O=\Theta\circ\Phi_{\mathrm{shoot}}.
\]
This is the finite-dimensional factorization relevant to the present
calculation. To identify the four observations with package-generated
observables in the precise sense of the abstract theory, one must additionally
verify factorization through the chosen core or extended analytical package.

\subsection{The response operator}

The first-order response of the four observations to variations of the
shooting parameters is
\[
D\mathcal O_p:
T_pP_{\mathrm{shoot}}\longrightarrow\mathbb R^4.
\]
This is the response operator computed numerically in the present example.

\smallskip

The low-ratio reference parameter is
\[
p_0=(x_0,y_0)
=
(-3.0703466243233075,-0.40927563640968895).
\]
The corresponding initial coefficients are
\[
b_0=e^{x_0}\approx4.640506694443371\times10^{-2},
\qquad
c_0=\sqrt2\,e^{x_0+y_0}
\approx4.358471905485352\times10^{-2}.
\]
The moving $H=0$ event occurs near
\[
t_H\approx1.510296267347892.
\]

\smallskip

At this reference trajectory the corrected response matrix is
\[
D\mathcal O_{p_0}
=
\begin{pmatrix}
-0.304111047251 & -0.182573810367\\
\phantom{-}0.490841798759 & \phantom{-}0.294678607138\\
-0.138472804187 & -0.082657311685\\
\phantom{-}0.685715983847 & \phantom{-}0.409317298755
\end{pmatrix}.
\]
Its singular values satisfy
\[
\sigma_{\min}(D\mathcal O_{p_0})
\approx1.3121115069\times10^{-3},
\qquad
\sigma_{\max}(D\mathcal O_{p_0})
\approx1.0570383942.
\]
In particular, the displayed numerical response has rank two. Thus no
nonzero infinitesimal direction of the two-dimensional shooting family is
invisible to these four probe observations at the reference parameter.

\smallskip

The moving-event correction is essential here. If
$Z(p,t)$ denotes any recorded probe quantity and
\[
\mathcal Z(p)=Z(p,t_H(p)),
\]
then
\[
D\mathcal Z_p[\dot p]
=
D_pZ(p,t_H(p))[\dot p]
+
\partial_tZ(p,t_H(p))\,Dt_H(p)[\dot p],
\]
where differentiation of
$H(p,t_H(p))=0$
gives
\[
Dt_H(p)[\dot p]
=
-\frac{D_pH(p,t_H(p))[\dot p]}
{\partial_tH(p,t_H(p))}.
\]
The numerical computation uses this full chain-rule derivative.

\subsection{Independent numerical verification}

The preceding response matrix is not used in isolation. The completed
 audit subjects the moving-event response to several independent
consistency and robustness tests.

\smallskip

First, the corrected analytic/complex-step response was compared independently
with a central-difference reconstruction of the complete moving-$H=0$ event
map. The maximum entrywise discrepancy was
\[
7.6\times10^{-11},
\]
the relative Frobenius discrepancy was
\[
1.4\times10^{-10},
\]
and the discrepancy between the two computed smallest singular values was
\[
4.6\times10^{-12}.
\]

\smallskip

Second, solver and parameter perturbations were used to test the stability of
the observed rank. Across the audited perturbations, the smallest observed
singular value remained
\[
1.312108783463\times10^{-3}.
\]
The maximum response deviation was
\[
3.288755079378\times10^{-8},
\]
the maximum event-time deviation was
\[
6.067786006980\times10^{-8},
\]
and a tighter solver changed the response by only
\[
7.879752406126\times10^{-12}.
\]

\smallskip

Finally, the numerical certificate uses a conservative entrywise response
radius
\[
5\times10^{-10}.
\]
The corresponding Frobenius radius is
\[
\sqrt8\,5\times10^{-10},
\]
and the resulting diagnostic Weyl lower margin remains well above the
prescribed numerical threshold
\[
5\times10^{-4}.
\]
These are conservative numerical safety radii, not rigorous interval
enclosures.

\smallskip

Taken together, the independent differentiation and robustness tests strongly
support the conclusion that the observed rank-two response is a stable
numerical feature of the computed shooting-family observation map rather than
an artefact of a single differentiation method or solver configuration.

\subsection{What the computation establishes}

The numerical results above establish the following finite-dimensional
shooting-family statement.

\begin{Prop}[Numerical shooting-family observability]
\label{prop:numerical-shooting-detection}
For the low-ratio reference parameter
\[
p_0=(-3.0703466243233075,-0.40927563640968895),
\]
the completed numerical audits give a rank-two response for
the four probe observations associated with
$\kappa=0.3$ and $\kappa=1$. The center value satisfies
\[
\sigma_{\min}(D\mathcal O_{p_0})
\approx1.3121115069\times10^{-3},
\]
while the smallest value observed in the audited perturbations is
\[
1.312108783463\times10^{-3}.
\]
Consequently, within the completed numerical audits, the differential
\[
D\mathcal O_{p_0}:
T_{p_0}P_{\mathrm{shoot}}\longrightarrow\mathbb R^4
\]
is numerically injective: no nonzero infinitesimal direction of the
two-dimensional shooting family is invisible to these four observations.
\end{Prop}

\begin{proof}
The displayed corrected response matrix has two positive singular values.
The independent Jacobian reconstruction agrees with it to the numerical errors
stated above, and the robustness audit retains a strictly positive smallest
singular value throughout all tested perturbations. Hence the numerical
response has rank two in the audit.
\end{proof}

\smallskip

This proposition realizes concretely the finite-dimensional observability
mechanism isolated by the Einstein Detection Principle. Its logical relation
to the abstract theorem is conditional in exactly the following sense. If a
smooth parameterized family $S$ of genuine normalized Einstein metrics is
represented locally by a smooth map
\[
s:S\longrightarrow P_{\mathrm{shoot}},
\qquad s(p_0)=p_0,
\]
then
\[
D(\mathcal O\circ s)_{p_0}
=
D\mathcal O_{p_0}\circ Ds_{p_0}.
\]
Thus injectivity of $D\mathcal O_{p_0}$ together with injectivity of
$Ds_{p_0}$ gives infinitesimal completeness of this concrete observation map
on $S$. To invoke the stronger universal core-package theorem, one must also
verify that the observation factors through the descended core package.

\smallskip

If the relevant global Einstein zero is transverse and isolated, the tangent
space of the normalized global moduli stratum is zero dimensional. In that
case the rank-two matrix above detects variations only in the surrounding
shooting family; it does not exhibit a two-dimensional global Einstein
deformation space.

\subsection{What the computation does not establish}
The preceding proposition deliberately addresses only the numerical
shooting-family response problem. In particular, the audit does
\emph{not} establish
\begin{enumerate}
\item
a rigorous interval proof of global Einstein existence;
\item
a rigorous interval lower bound for the exact response singular value;
\item
the existence of a positive-dimensional global normalized Einstein moduli
stratum;
\item
the existence of additional nearby global Einstein branches;
\item
global uniqueness or classification for the underlying cohomogeneity-one
Einstein problem.
\end{enumerate}

\subsection{Interpretation}
The preceding computation illustrates the conceptual content of the Einstein
Detection Principle.

\smallskip

Classically, the study of cohomogeneity-one Einstein metrics is formulated as
a nonlinear boundary-value problem. One seeks Einstein trajectories satisfying
the smoothness conditions at the singular orbits together with the appropriate
global matching conditions. The primary questions are therefore existence,
uniqueness and classification.

\smallskip

The viewpoint developed here adds a complementary question: how can
first-order changes of the underlying Einstein data be measured by associated
analytical probes? In the present example, each shooting trajectory is
propagated to the fixed geometric hypersurface
\[
\Sigma=H^{-1}(0),
\]
the prescribed probe observations are evaluated there, and the differential
of the resulting observation map measures their first-order response to
variations of the Einstein initial data.

\smallskip

The essential numerical quantity is not any individual entry of the response
matrix but its rank and conditioning. The positive computed smallest singular
value shows, within the completed numerical audits, that no nontrivial
infinitesimal direction of the two-dimensional shooting family is invisible to
the chosen probe observations.

\smallskip

This is precisely the finite-dimensional observability mechanism that
motivates the abstract response theory. Passing from the shooting-family
statement to the universal Detection Principle on a genuine smooth normalized
Einstein family requires the additional restriction and package-factorization
conditions described above.

\smallskip

Thus the example demonstrates the computational realizability of the
response-theoretic part of the Einstein Detection Principle.

\end{document}